\documentclass[a4paper, 12pt]{amsart}
\usepackage[a4paper,hmargin=2.5cm,vmargin=2.5cm]{geometry}
\usepackage[T2A]{fontenc}
\usepackage[utf8]{inputenc}

\usepackage[russian,english]{babel}
\usepackage{amsmath,amssymb,amsthm,amsfonts,mathrsfs,amscd}
\usepackage[linkcolor=blue,citecolor=blue,colorlinks=true,urlcolor=blue]{hyperref}
\usepackage{comment}
\usepackage{tikz}
\usepackage{graphics}
\usepackage{xfrac}
\usepackage{faktor}
\usepackage{rotating}
\numberwithin{equation}{section}
\newtheorem{theorem}{Theorem}[section]
\newtheorem{lemma}[theorem]{Lemma}
\newtheorem{propos}[theorem]{Proposition}
\newtheorem{cor}[theorem]{Corollary}
\theoremstyle{definition}
\newtheorem{definition}[theorem]{Definition}

\newtheorem{proof}{Proof}
\let\origendproof\endproof
\def\endproof{\unskip\nobreak\hskip5pt plus 1fill$\square$\origendproof}
\newtheorem{ex}[theorem]{Example}
\newtheorem{rem}[theorem]{Remark}

\def\QQ{\mathbb Q}

\def\CC{\mathbb C}

\def\ZZ{\mathbb Z}
\def\PP{\mathbb P}

\def\top{\mathrm{top}}

\def\Gal{\mathop{\rm Gal}\nolimits}
\def\Bs{\mathop{\rm Bs}\nolimits}
\def\Pic{\mathop{\rm Pic}\nolimits}
\def\Cl{\mathop{\rm Cl}\nolimits}
\def\Sing{\mathop{\rm Sing}\nolimits}

\def\PGL{\mathop{\rm PGL}\nolimits}
\def\GL{\mathop{\rm GL}\nolimits}

\def\Cr{\mathop{\rm Cr}\nolimits}

\def\Aut{\mathop{\rm Aut}\nolimits}
\def\Bir{\mathop{\rm Bir}\nolimits}

\def\gcd{\mathop{\rm gcd}\nolimits}
\def\sch{\mathop{\rm sch}\nolimits}
\def\se{\mathop{\rm se}\nolimits}
\def\sm{\mathop{\rm sm}\nolimits}
\def\id{\mathop{\rm id}\nolimits}
\begin{document}
\title{Finite subgroups of Cremona group of rank 3 over the field of rational numbers}
\author{A.\,V.~Zaitsev}
\address{National research university ''Higher school of economics'', Laboratory of algebraic geometry, 6 Usacheva str., Moscow, 119048, Russia}
\email{\href{alvlzaitsev1@gmail.com}{alvlzaitsev1@gmail.com}}

%\date{03.11.2021}
%\udk{512.774.4}

\maketitle
\begin{abstract}
    We give an explicit bound on orders of finite subgroups of Cremona group of rank three over $\mathbb{Q}$.
\end{abstract}

\tableofcontents

\section{Introduction}

Our starting point is a famous Minkowski theorem from the $19^{\text{th}}$ century, which reveals a sharp multiplicative bound on orders of finite subgroups of general linear groups over a field of rational numbers. More precisely, the following theorem holds.

\begin{theorem}[{\cite[Minkowski's theorem]{Min}}]\label{theo: Minkowski}
    Let $n \geqslant 1$ be an integer and $p$ be a prime number. Define $$M(n,p) = \left\lfloor \frac{n}{p-1}\right\rfloor + \left\lfloor \frac{n}{p(p-1)}\right\rfloor + \left\lfloor \frac{n}{p^2(p-1)}\right\rfloor + \dots$$

    Then: 
    \begin{enumerate}
        \item If $G$ is a finite subgroup of $\GL_n(\QQ)$, then $\nu_p(|G|) \leqslant M(n,p)$, where $\nu_p$ is a $p$-adic valuation.
        \item There exists a finite $p$-subgroup $G$ of $\GL_n(\QQ)$ with $\nu_p(|G|) = M(n,p)$.
    \end{enumerate}
    In particular, if we denote by $\mathcal{P}$ the set of all prime numbers, then for any finite subgroup~$G$ of $\GL_n(\QQ)$ we have $$|G| \leqslant \prod_{p \in \mathcal{P}} p^{M(n,p)}.$$
\end{theorem}

In 1905 Schur generalized the first part of this theorem to the case of an arbitrary number field in his paper~\cite{Schur}. A hundred years later, Serre proved similar theorems with explicit bounds for reductive groups (see~\cite[Theorem 5]{Serre red}) and, what is more interesting for us, for Cremona group of rank $2$  over $\textit{small}$ fields, in particular for $\Cr_2(\QQ)$, see~\hbox{\cite[Theorem 2.4]{Serre MSB}} (recall that Cremona group $\Cr_n(K)$ of rank $n$ over a field $K$ is the group of birational automotphisms of projective space $\PP_K^n$). Also, in the latter paper Serre asked a series of related questions, in particular: are there analogous bounds (or some other bounds) on orders of finite subgroups in higher dimensions? This question was answered positively by Prokhorov and Shramov in~\cite{PrShr higher dim} in more general case, but their bounds were no longer explicit.
\begin{theorem}[{see \cite[Theorem 1.4]{PrShr higher dim}}]\label{theo: Prokgorov Shramov any dim}
    Suppose that $K$ is a finitely generated field over~$\QQ$. Let $X$ be a variety over $K$. Then the group $\Bir(X)$ of birational automorphisms of $X$ over~$K$ has bounded finite subgroups.
\end{theorem}

A natural question in this direction is to find some explicit and effective bounds, and it is also natural to start with $\Cr_3(\QQ)$. Note, that the result about non-multiplicative bound on the orders of finite subgroups of $\Cr_2(\QQ)$ was obtained quite recently in~\cite{Ahmed A} and the obtained bound is sharp. 

\begin{theorem}[{\cite[Theorem 1.1]{Ahmed A}}]\label{theo: cr_2}
        Let $G$ be a finite subgroup of $\Cr_2(\QQ)$. Then $|G| \leqslant 432$. Moreover, there is a finite subgroup of $\Cr_2(\QQ)$ of order $432$. 
\end{theorem}

In the present paper we give an explicit bound for $\Cr_3(\QQ)$, namely, the following theorem holds.

\begin{theorem}\label{theo: cr_3}
    Let $X$ be a $3$-dimensional variety over $\QQ$ with a $\QQ$-point such that~$\relpenalty=10000 \overline{X} = X \otimes_\QQ \overline{\QQ}$ is rationally connected. Let $G \subset \Bir(X)$ be a finite subgroup. Then $$|G| \leqslant  24\,103\,053\,950\,976\,000 < 10^{17}.$$ In particular, this applies to finite subgroups of $\Cr_3(\QQ)$.%2^{15}\cdot3^{6}\cdot5^{4}\cdot7^2\cdot11^{4} =
\end{theorem}

This bound seems to be far from being sharp. For example, the largest finite subgroup of~$\Cr_3(\QQ)$ which we know is $(D_{12} \times D_{12} \times D_{12}) \rtimes \mathfrak{S}_3$. It acts faithfully on $\PP^1_{\QQ} \times \PP^1_{\QQ} \times \PP^1_{\QQ}$ and contains $10\,368$ elements. Moreover, the bound can be improved by a factor of $10^5$, if one can deal with some particular case of Fano threefolds, see Remark~\ref{remark: problems with i = 1, g = 10} and Theorem~\ref{theo: Cr_3 for future}.

The plan of the paper is as follows. In Section~\ref{sect: GL and PGL} we give a motivation for bounding orders finite subgroups of $\GL_n$ and $\PGL_n$ (in case of $\GL_n$ not only over $\QQ$, but over number fields), recall classical results and prove auxiliary claims. In Section~\ref{sect: MFS with dimB>0} we bound orders of finite subgroups of automorphism groups of Mori fiber spaces over bases of positive dimension. In Section~\ref{sect: non-Gorenstein Fano} we bound numbers of singularities on terminal non-Gorenstein Fano threefolds and bound orders of finite subgroups acting on such varieties in Proposition~\ref{prop: bound for non-Gorenstein}. In Section~\ref{sect: Gorenstein Fano} we bound orders of finite groups acting on terminal Gorenstein Fano threefolds in Propositions~\ref{prop: Gorenstein rho > 1} and~\ref{propos: Gorenstein, rho = 1} and prove Theorem~\ref{theo: cr_3}. Appendix~\ref{sect: appendix} contains the technical part of the paper, which is devoted to bounding of invariants of number fields of degrees at most $15$.

We will use the following notation. By $\mu_n$ we denote the cyclic group of order $n$. By~$D_{2n}$ we denote the dihedral group of order~${2n}$. By $\xi_n$ we denote a primitive root of unity of order $n$, in particular, $\xi_4$ denotes an imaginary unit. By $\nu_p$ we denote the $p$-adic valuation. By $\varphi$ we denote the Euler totient function. By $\overline{K}$ we denote the algebraic closure of a field $K$. If $K \subset L$ is a field extension and $X$ is a variety over $K$, then by~$X_L$ we denote the extension of scalars of $X$ to $L$, and we denote the extension of scalars of~$X$ to~$\overline{K}$ by~$\overline{X}$. Zariski tangent space to a variety~$X$ at a point $P$ we denote by~$T_PX$. 
By $\Pic(X)$ we denote the Picard group of a variety $X$. By $\Cl(X)$ we denote the Weil divisor class group of a variety $X$.
%Let $\pi: X \xrightarrow{} B$ be a morphism of algebraic varieties; then by $\Aut(X,\pi)$ we denote the group of automorphisms compatible with the morphism $\pi$, that is, the group of such automorphisms $g\in\Aut(X)$ that there exists~$g'\in\Aut(B)$ for which $\pi\circ g = g'\circ\pi$. % We write $k = (m)$ saying that integer number $k$ contains $m$ digits (for the convenience of comparing large numbers arising during the computations). 
By $\mathcal{P}$ we denote the set of all prime numbers.

\textbf{Acknowledgements.} I would like to thank my advisor Constantin Shramov for stating the problem, useful discussions and constant attention to this work. Also, I would like to thank Mikhail Panov for help with \LaTeX.

\section{Bounds for linear groups}\label{sect: GL and PGL}

Since we deal with finite subgroups of Cremona group, then it is clear that in particular we should consider all finite subgroups of automorphism groups of all rational Fano varieties. The connection with Fano varieties is even more obvious if we keep in mind the minimal model program, but we will see this in the next sections. Anyway, we have to bound orders of finite subgroups of automorphism groups of Fano varieties.

If we are given a Fano variety $X$ over $\QQ$ and we want to bound finite subgroups of~$\Aut(X)$, the first idea is to consider an $\Aut(X)$-equivariant embedding of $X$ to some projective space~$\PP^d_{\QQ}$ by some degree of the anticanonical system and to use the bound for~$\PGL_{d+1}(\QQ)$. Of~course, this strategy works well if number $d$ is not very large. In some cases, the bounds obtained in this way will be sufficient for our purposes. This is the motivation of bounding finite subgroups of projective linear groups.

Another possible strategy for studying finite groups acting on varieties is to use the following standard theorem (for proof see for example~\cite[Theorem 2.7]{Zai}) and to apply the bound for $\GL_n$. 

\begin{theorem}\label{theo: action with fixed point}
    Let $X$ be an irreducible algebraic variety over a field $K$ of characteristic~$0$. Let $G$ be a finite group acting faithfuly on $X$ with a fixed~$K$-point $P$. Then the natural homomorphism $$d:G\xrightarrow{}\GL(T_PX)$$ is injective.
\end{theorem}

Of course, to apply this theorem we need a fixed rational point, which is almost never the case. However, working with singular varieties, whose singularities are isolated points, you can get a fixed rational point by considering an extension of the base field and passing to a subgroup, which will be enough for our purposes. This is the motivation of bounding finite subgroups of general linear groups (not only over $\QQ$, but over number fields). 

To find these bounds, we apply Schur's and Serre's theorems. But before formulating them, we should recall definitions of cyclotomic invariants $m$ and $t$. All fields we work with in the present paper have characteristic $0$, so we give all definitions and statements in this generality.
\begin{definition}(\cite[\S 1]{Serre MSB})\label{def: cyclotomic invariants}
    Let $K$ be a field of characteristic $0$ and $p \neq 2$ be a prime number. Let $\xi_n$ be an $n$-th primitive root of unity and, then we put 
    \begin{itemize}
        \item $t_p= t_{p}(K) = [K(\xi_p): K]$;
        \item $t_2= t_{2}(K) = [K(\xi_4): K]$;
        \item $m_p = m_{p}(K) = \sup\{n \mid \xi_{p^n} \in K(\xi_p)\}$;
        \item $m_2 = m_{2}(K) = \sup \{n \mid \xi_{2^n} \in K\}$, if $\xi_4 \in K$;
        \item $m_2 = m_{2}(K) = \sup \{n \mid \xi_{2^n} + \xi_{2^n}^{-1} \in K\}$, if $\xi_4 \not\in K$.
    \end{itemize}
\end{definition}

\begin{rem}\label{remark: m and t}
    %First of all, note that $m_2 \geqslant 2$ in any case. 
   In paper~\cite{Serre red}, Serre gives seemingly different definitions of cyclotomic invariants, but, as he mentions himself, they are equivalent to those in Definition~\ref{def: cyclotomic invariants}. Also, it is worth to note that $m_2 \geqslant 2$ for any $K$.
    %It is worth noting that $t$ divides $p-1$ for $p > 2$, and if $p = 2$ or $3$, then $t = 1$ or $2$. Also, note that for $p = 2$, we have $m \geqslant 2$.
\end{rem}
In addition, it is convenient to introduce the following invariants.

\begin{definition}\label{def: schur and serre invariants}
    Let $K$ be a field of characteristic $0$. Let $p$ be a prime number, then we put 
    \begin{itemize}
        \item $\nu^{\sch}_{p,n} = \nu^{\sch}_{p,n}(K) = m_p\left\lfloor \frac{n}{t_p}\right\rfloor + \left\lfloor \frac{n}{pt_p}\right\rfloor + \left\lfloor \frac{n}{p^2t_p}\right\rfloor + \dots$, if $p \neq 2$;
        \item $\nu^{\sch}_{p,n} = \nu^{\sch}_{p,n}(K) = m_2n+ \left\lfloor \frac{n}{2}\right\rfloor + \left\lfloor \frac{n}{2^2}\right\rfloor + \dots$, if $p=2$ and $\xi_4 \in K$;
        \item $\nu^{\sch}_{p,n} = \nu^{\sch}_{p,n}(K) = n + m_2\left\lfloor \frac{n}{2}\right\rfloor + \left\lfloor \frac{n}{2^2}\right\rfloor + \dots$, if $p = 2$ and $\xi_4 \not\in K$;
        \item $\nu^{\se}_{p,n} = \nu^{\se}_{p,n}(K) = m_p\left\lfloor\frac{n-1}{\varphi(t_p)}\right\rfloor + \nu_p((n-1)!)$.
    \end{itemize}
    Also, we define 
    $$ \nu^{\sch}_n = \nu^{\sch}_n(K) = \prod_{p \in \mathcal{P}} p^{\nu^{\sch}_{p,n}}, \;  \nu^{\se}_n = \nu^{\se}_n(K) = \prod_{p \in \mathcal{P}} p^{\nu^{\se}_{p,n}}.$$
    \end{definition}

Now, to bound orders of finite subgroups of $\GL_n$ let us recall Schur's theorem from~\cite{Schur}, but in Serre's formulation.

\begin{theorem}[{cf. \cite[Theorem $2'$]{Serre red}}]\label{theorem: schur}
    Let $K$ be a number field, $p$ be a prime number, and~$n$ be a positive integer number. Then for any finite subgroup $G \subset \GL_n(K)$ we have $$\nu_p(|G|) \leqslant \nu^{\sch}_{p,n}.$$
\end{theorem}

\begin{rem}
    The bound from Theorem~\ref{theorem: schur} could be improved a bit in case $p = 2$ for some types of fields. However, it will be difficult to identify these types in our approach, so we limit ourselves to the specified bound.
\end{rem}
    Theorem~\ref{theorem: schur} gives an obvious bound on orders of finite subgroups of $\GL_n$.
\begin{cor}\label{corollary: schur}
    Let $K$ be a number field. Then for any finite subgroup $G \subset \GL_n(K)$ we have $$|G| \leqslant \nu^{\sch}_n.$$
\end{cor}

The following corollaries are important for applications. 
\begin{comment}
\begin{cor}\label{corollary: additive bound}
Let $K$ be a number field of degree $d \leqslant 21$. Let $G \subset \GL_4(K)$ be a finite subgroup. Then $$|G| \leqslant 510\,040\,903\,680\,000
 < 10^{15}.$$
\end{cor}

\begin{proof}
    By Corollary~\ref{corollary: schur} we have $|G| \leqslant \nu^{\sch}_4$ and the result follows from Proposition~\ref{proposition: bound of schur <= 21}.
\end{proof}
\end{comment}
%\begin{cor}\label{corollary: additive bound <=19}
%Let $K$ be a number field of degree $d \leqslant 19$. Let $G \subset \GL_4(K)$ be a finite subgroup. Then $$|G| \leqslant 2\,546\,633\,779\,200
% < 10^{13}.$$
%\end{cor}

%\begin{proof}
%    By Corollary~\ref{corollary: schur} we have $|G| \leqslant \nu^{\sch}_4$ and the result follows from Proposition~\ref{proposition: bound of schur <= 19}.
%\end{proof}

\begin{cor}\label{corollary: additive bound GL_4(<=7)}
    Let $K$ be a number field of degree $d \leqslant 7$. Let $G \subset \GL_4(K)$ be a finite subgroup. Then $$|G| \leqslant 1\,132\,185\,600 < 10^{10}.$$
\end{cor}

\begin{proof}
    By Corollary~\ref{corollary: schur} we have $|G| \leqslant \nu^{\sch}_4$ and the result follows from Proposition~\ref{proposition: bound of schur GL_4 <= 7}.
\end{proof}

\begin{cor}\label{corollary: additive bound GL_4(<=2)}
    Let $K$ be a number field of degree $d \leqslant 2$. Let $G \subset \GL_4(K)$ be a finite subgroup. Then $$|G| \leqslant 87\,091\,200 < 10^8.$$
\end{cor}

\begin{proof}
    By Corollary~\ref{corollary: schur} we have $|G| \leqslant \nu^{\sch}_4$ and the result follows from Lemma~\ref{lemma: corollary: bound of schur GL_4 <= 2}.
\end{proof}

\begin{cor}\label{corollary: additive bound GL_3(<=15)}
    Let $K$ be a number field of degree $d \leqslant 15$. Let $G \subset \GL_3(K)$ be a finite subgroup. Then $$|G| \leqslant 240\,045\,120 < 10^{9}.$$
\end{cor}

\begin{proof}
    By Corollary~\ref{corollary: schur} we have $|G| \leqslant \nu^{\sch}_3$ and the result follows from Proposition~\ref{proposition: bound of schur GL_3 <= 15}.
\end{proof}

\begin{cor}\label{corollary: additive bound GL_3(<=2)}
    Let $K$ be a number field of degree $d \leqslant 2$. Let $G \subset \GL_3(K)$ be a finite subgroup. Then $$|G| \leqslant 2\,903\,040 < 10^7.$$
\end{cor}

\begin{proof}
    By Corollary~\ref{corollary: schur} we have $|G| \leqslant \nu^{\sch}_3$ and the result follows from Lemma~\ref{lemma: corollary: bound of schur GL_3 <= 2}.
\end{proof}

\begin{rem}\label{roughness of bounds for <=2}
Note that the bounds in Corollaries~\ref{corollary: additive bound GL_4(<=7)} and~\ref{corollary: additive bound GL_3(<=15)} require much technical work, which is done in Appendix~\ref{section: appendix}, weather bounds in Corollaries~\ref{corollary: additive bound GL_4(<=2)} and~\ref{corollary: additive bound GL_3(<=2)} are quite rough, but they are enough for our purposes.
\end{rem}

To deal with $\PGL_n$ we use the following Serre's theorem for reductive groups in characteristic $0$.

\begin{theorem}[{\cite[Theorem 5]{Serre red}}]\label{theo: serre's theorem for reductive groups}
Let $K$ be a finitely generated field over $\QQ$. Let $\Gamma$ be a connected reductive group over $K$ of rank $r$ with Weyl group $W$. 

Let $G \subset \Gamma$ be a finite subgroup and $p$ be a prime number. Then $$\nu_p(|G|) \leqslant m_p\left\lfloor\frac{r}{\varphi(t_p)}\right\rfloor + \nu_p(W).$$
\end{theorem}

\begin{cor}\label{corollary: Serre for PGL}
   Let $K$ be a number field, $p$ be a prime number, and $n$ be a positive integer number. Then for any finite subgroup $G \subset \PGL_n(K)$ we have
    $$\nu_p(|G|) \leqslant \nu^{\se}_{p,n}.$$
\end{cor}

\begin{proof}
   Since $\PGL_n(K)$ is reductive of rank $n-1$ with Weyl group isomorphic to $\mathfrak{S}_{n-1}$, corollary immediately follows from Definition~\ref{def: schur and serre invariants} and Theorem~\ref{theo: serre's theorem for reductive groups}.
\end{proof}
    Corollary~\ref{corollary: Serre for PGL} gives an obvious bound on orders of finite subgroups of $\PGL_n$.
\begin{cor}\label{corollary: Serre bound for PGL}
    Let $K$ be a number field. Then for any finite subgroup $G \subset \PGL_n(K)$ we have $$|G| \leqslant \nu^{\se}_n.$$
\end{cor}

Over $\QQ$ Corollary~\ref{corollary: Serre bound for PGL} gives explicit bounds. Let us compute these bounds for $\PGL_5(\QQ)$, we will need this for applications. 

\begin{cor}\label{corollary: PGL_5(Q) and PGL_6(Q)}
    Let $G \subset \PGL_5(\QQ)$ be a finite subgroup. Then $$|G| \leqslant 2^{11}\cdot 3^5\cdot5^2\cdot7^2\cdot11\cdot13 = 87\,178\,291\,200 < 10^{11}.$$   
\end{cor}

The following two corollaries are also important for applications.

\begin{cor}\label{corollary: serr for PGL_3(<=2)}
    Let $K$ be a number field of degree $d \leqslant 2$. Let $G \subset \PGL_3(K)$ be a finite subgroup. Then $$|G| \leqslant 2\,620\,800 < 10^{7}.$$
\end{cor}
\begin{proof}
    By Corollary~\ref{corollary: Serre bound for PGL} we have $|G| \leqslant \nu^{\se}_3$ and the result follows from Lemma~\ref{lemma: serr for PGL_3(<=2)}.
\end{proof}
\begin{cor}\label{corollary: serr for PGL_4(<=2)}
    Let $K$ be a number field of degree $d \leqslant 2$. Let $G \subset \PGL_4(K)$ be a finite subgroup. Then $$|G| \leqslant 943\,488\,000 < 10^{9}.$$
\end{cor}
\begin{proof}
    By Corollary~\ref{corollary: Serre bound for PGL} we have $|G| \leqslant \nu^{\se}_4$ and the result follows from Lemma~\ref{lemma: serr for PGL_4(<=2)}.
\end{proof}

\begin{comment}
Let us consider a bunch of examples of bounds for $\PGL_n(\QQ)$ and $\GL_n(\QQ)$ given by previous corollary, which will be important when working with smooth Fano threefolds of index $1$ and considering their anticanonical embeddings.

\begin{ex}\label{example: for PGL_14}
    Let $G$ be a finite subgroup of \textcolor{blue}{(мб нужны будут не все, но пока так)}
    \begin{enumerate}
    \item $\GL_{14}(\QQ)$, then $$|G| \leqslant 2^{26} \cdot 3^{13} \cdot 5^6 \cdot 7^6 \cdot 11^3 \cdot 13^3 \cdot 17 \cdot 19^2 \cdot 23 \cdot 29 \cdot 31 \cdot 37 \cdot 43 = (51);$$
        \item 
    $\PGL_{14}(\QQ)$ or $\GL_{13}(\QQ)$, then $$|G| \leqslant 2^{24}\cdot 3^{12}\cdot 5^{6} \cdot 7^6 \cdot 11^3 \cdot 13^3 \cdot 17 \cdot 19^2 \cdot 23 \cdot 29 \cdot 31 \cdot 37 \cdot 43 = (49);$$
    \item $\PGL_{12}(\QQ)$ or $\GL_{11}(\QQ)$, then $$|G| \leqslant 2^{20}\cdot 3^{10} \cdot 5^5 \cdot 7^5 \cdot 11^2 \cdot 13^2 \cdot 17 \cdot 19 \cdot 23 \cdot 31 = (35);$$
    \item $\PGL_{11}(\QQ)$ or $\GL_{10}(\QQ)$, then $$|G| \leqslant 2^{18} \cdot 3^9 \cdot  5^4 \cdot 7^4 \cdot 11^2 \cdot 13^2 \cdot 17 \cdot 19 \cdot 31 = (30);$$
    \item $\PGL_{10}(\QQ)$ or $\GL_{9}(\QQ)$, then $$|G| \leqslant 2^{16} \cdot 3^8 \cdot  5^4 \cdot 7^4 \cdot 11^2 \cdot 13^2 \cdot 17 \cdot 19 \cdot 31 = (28);$$
    \item $\PGL_{n}(\QQ)$ or $\GL_{n-1}(\QQ)$ with $n\leqslant 9$, then $$|G| \leqslant 2^{14} \cdot 3^7 \cdot  5^3 \cdot 7^3 \cdot 11 \cdot 13 \cdot 19 = (20);$$
    \end{enumerate}
\end{ex}

\end{comment}

For small ranks, we will
need the following well-known results.

\begin{lemma}\label{lemma: Burnside}
    Let $G \subset \GL_2(\QQ)$ be a finite subgroup. Then $|G| \leqslant 12$.
\end{lemma}

\begin{proof}
    By Burnside's theorem each finite subgroup of $\GL_2(\QQ)$ is conjugate to a finite subgroup of $\GL_2(\ZZ)$, see~\cite{Burn}. Therefore, the group $G$ embeds to $\GL_2(\ZZ)$ and
    by~\mbox{\cite[Chapter IX, §14]{New}} we have $|G| \leqslant 12$. 
\end{proof}

\begin{propos}[{\cite[Proposition 1.1]{Bea}}]\label{prop:Beauville}
Let $K$ be a field of characteristic $0$ and $\xi_m$ be a primitive $m$-th root of unity.
\begin{enumerate}
    \item $\emph{PGL}_2(K)$ contains $\mu_m$ and $D_{2m}$ if and only if $K$ contains~$\xi_m + \xi_m^{-1}$ $($in particular, always contains $D_6$, $D_8$ и $D_{12}$$)$;
    \item $\emph{PGL}_2(K)$ contains $\mathfrak{A}_4$ and $\mathfrak{S}_4$ if and only if $-1$ is a sum of two squares in $K$;
    \item $\emph{PGL}_2(K)$ contains $\mathfrak{A}_5$ if and only if $-1$ is a sum of two squares in $K$, and $K$ contains $\sqrt{5}$.
\end{enumerate}
\end{propos}

\begin{cor}\label{corollary: bound of finite subgroups of PGL_2(Q)}
   All finite subgroups of $\PGL_2(\mathbb{Q})$ are: $\mu_2$, $\mu_3$, $\mu_4$, $\mu_6$, $D_4$, $D_6$, $D_8$ and~$D_{12}$. In particular, their orders are bounded by $12$.
\end{cor}

\begin{comment}

\begin{lemma}\label{lemma: eraising big primes for GL_n over K}
    Let $\QQ \subset K$ be a field extension of degree $d$ and $G$ be a finite subgroup of~$\GL_n(K)$. If $p > dn + 1$ is a prime number, then $\nu_p(|G|) = 0$. 
\end{lemma}

\begin{proof}
    Let us prove by a contradiction. Assume that $\nu_p(|G|) \geqslant 1$. Then there is a matrix~$A \in G$ of order $p$. Then $A$ has an eigenvalue equals to $\xi_p$, and this eigenvalue is a root of a characteristic polynomial of $A$, which has degree $n$. Therefore, $[K(\xi_p) : K] \leqslant n$ and $[K(\xi_p) : \QQ] \leqslant nd < p -1$, which is a contradiction with $\relpenalty = 10000 \binoppenalty = 10000 [\QQ(\xi_p): \QQ] = p - 1$.
\end{proof}

\begin{cor}\label{corollary: eraising big primes for GL_n over Q}
    Let $G$ be a finite subgroup of $\GL_n(\QQ)$. If $p > n + 1$ is a prime number, then $\nu_p(|G|) = 0$.
\end{cor}
\end{comment}

Also, the following bunch of simple lemmas will be needed.

\begin{lemma}\label{lemma: degree of semi-syclic extension}
Let $\xi_n$ be a primitive $n$-th root of unity, where $n > 2$. Then $$[\mathbb{Q}(\xi_n + \xi_n^{-1}):\mathbb{Q}] = \frac{\varphi(n)}{2},$$ where $\varphi$ is the Euler's totient function.
\end{lemma}

\begin{proof}
Consider the tower of fields $$\mathbb{Q} \subset \mathbb{Q}(\xi_n + \xi_n^{-1}) \subset \mathbb{Q}(\xi_n).$$ It is well-known that $[\mathbb{Q}(\xi_n):\mathbb{Q}] = \varphi(n)$. Let us show that $[\mathbb{Q}(\xi_n):\mathbb{Q}(\xi_n + \xi_n^{-1})] = 2$. Firstly, these fields do not coincide, since $\mathbb{Q}(\xi_n + \xi_n^{-1})$ is isomorphic to a subfield of $\mathbb{R}$, weather $\mathbb{Q}(\xi_n)$ is not. Secondly, $\xi_n$ is a root of the polynomial $x^2 - (\xi_n + \xi_n^{-1})x + 1$ of degree~$2$. That is, the degree of extension is equal to $2$.

Finally, we have $$[\mathbb{Q}(\xi_n + \xi_n^{-1}):\mathbb{Q}] = \frac{[\mathbb{Q}(\xi_m):\mathbb{Q}]}{[\mathbb{Q}(\xi_n):\mathbb{Q}(\xi_n + \xi_n^{-1})]} = \frac{\varphi(n)}{2}.$$
\end{proof}

\begin{lemma}\label{lemma: PGL_2(Q(sqrt(d)))}
Let $\QQ \subset K$ be a field extension of degree at most $2$. Let $G \subset \PGL_2(K)$ be a finite subgroup. Then $|G| \leqslant 24$.
\end{lemma}

\begin{proof}
There exist a rational number $d$ such that $K = \QQ(\sqrt{d})$. It follows from Proposition~\ref{prop:Beauville} that $G$ is either cyclic, or dihedral, or isomorphic to $\mathfrak{A}_4$, $\mathfrak{S}_4$ or~$\mathfrak{A}_5$. Note, that~$G$ can not be isomorphic to $\mathfrak{A}_5$. Indeed, suppose $G$ is isomorphic to $\mathfrak{A}_5$, then $\mathbb{Q}(\sqrt{d})$ should contain $\sqrt{5}$, and $-1$ should be decomposable in sum of two squares in  $\mathbb{Q}(\sqrt{d})$ by Proposition~\ref{prop:Beauville}. The first condition implies that $\mathbb{Q}(\sqrt{d})$ coincides with $\mathbb{Q}(\sqrt{5})$, which contradicts the second condition.

If $G$ is isomorphic to $\mathfrak{A}_4$ or $\mathfrak{S}_4$, then $|G| \leqslant 24$. If $G$ is isomorphic to $\mu_m$ or $D_{2m}$, then~$\mathbb{Q}(\sqrt{d})$ contains $\xi_m + \xi_m^{-1}$. Then we have $[\mathbb{Q}(\xi_m + \xi_m^{-1}):\mathbb{Q}] \leqslant 2$. Applying Lemma~\ref{lemma: degree of semi-syclic extension} we get the inequality $\varphi(m) \leqslant 4$, which implies $m \leqslant 12$ and \hbox{$|G| \leqslant 24$}.
\end{proof}

\begin{cor}\label{corollary: automorphsms of conic over Q}
    Let $G$ be a finite subgroup of $\Aut(C)$, where $C$ is a smooth conic over~$\mathbb{Q}$. Then $|G| \leqslant 24$.
\end{cor}

\begin{proof}
    Since conic $C$ is isomorphic to $\PP^1$ over quadratic extension of $\mathbb{Q}$, then $G$ is a subgroup of $\PGL_2(K)$, where $\QQ \subset K$ is an extension of degree $2$. Therefore, we have~$\relpenalty=10000 |G| \leqslant 24$ by Lemma~\ref{lemma: PGL_2(Q(sqrt(d)))}.
\end{proof}

\begin{lemma}\label{lemma: PGL_2(L), L <= m}
    Let $\QQ \subset K$ be a field extension of degree at most $d$. Let $n$ be the maximal integer number with  the property $\varphi(n) \leqslant 2d$. Let $G \subset \PGL_2(K)$ be a finite subgroup. Then $|G| \leqslant \max\{60, 2n\}$. 
\end{lemma}

\begin{proof}
    By Proposition~\ref{prop:Beauville}, we know that $G$ is either cyclic, or dihedral, or isomorphic to $\mathfrak{A}_4$, $\mathfrak{S}_4$ or~$\mathfrak{A}_5$. In the last three cases we have $|G| \leqslant 60$. If $G$ is isomorphic to $\mu_m$ or $D_{2m}$, then by Proposition~\ref{prop:Beauville} the field $K$ contains an element $\xi_m + \xi_m^{-1}$  and we have an inequality~$\binoppenalty = 10000 [\QQ(\xi_m + \xi_m^{-1}): \QQ] \leqslant d$. Applying Lemma~\ref{lemma: degree of semi-syclic extension}, obtain $\varphi(m) \leqslant 2d$, which implies~$\relpenalty = 10000 m \leqslant n$ and $|G| \leqslant 2n$.
\end{proof}

\begin{lemma}\label{lemma: PGL_2(L), L <= 4,6,24}
Let $\QQ \subset L$ be a field extension of degree at most $d$. Let $G \subset \PGL_2(L)$ be a finite subgroup. Then 
\begin{enumerate}
    \item $|G| \leqslant 60$ if $d=4$;
    \item $|G| \leqslant 84$ if $d=6$;
    \item $|G| \leqslant 180$ if $d = 12$; 
    \item $|G| \leqslant 420$ if $d = 24$.
\end{enumerate}
\end{lemma}
\begin{proof}
    Note that $\max\{n \mid \varphi(n) \leqslant 8\} = 30$, $\max\{n \mid \varphi(n) \leqslant 12\} = 42$, $\relpenalty=10000 \max\{n \mid \varphi(n) \leqslant 24\} = 90$ and~$\relpenalty=10000 \max\{n \mid \varphi(n) \leqslant 48\} = 210$. Now the result follows from Lemma~\ref{lemma: PGL_2(L), L <= m}. 
\end{proof}

\begin{lemma}\label{lemma: conic over L <=6}
    Let $\QQ \subset L$ be a field extension of degree at most $6$. Let $C$ be a conic over~$L$. Let $G \subset \Aut(C)$ be a finite subgroup. Then $|G| \leqslant 180$.
\end{lemma}

\begin{proof}
    Conic $C$ is isomorphic to $\PP^1$ over some quadratic extension of $L$. Therefore, $G$ is a finite subgroup of $\PGL_2(K)$, where $\QQ \subset K$ has degree at most $12$. applying Lemma~\ref{lemma: PGL_2(L), L <= 4,6,24}, obtain~$\relpenalty=10000 |G| \leqslant 180$.
\end{proof}

\begin{lemma}\label{lemma: GL_2(L), L <= 6}
    Let $\QQ \subset L$ be a field extension of degree at most $6$. Let $G \subset \GL_2(L)$ be a finite subgroup. Then $|G| \leqslant 1512$.
\end{lemma}

\begin{proof}
    If we denote the image of $G$ under standard projection to $\PGL_2(L)$ by $\overline{G}$, then there is a short exact sequence $$0 \rightarrow \mu_l \rightarrow G \rightarrow \overline{G} \rightarrow 1,$$ where $\mu_l \subset L^{\times}$. The largest degree of a primitive root in $L^{\times}$ is $18$, therefore $|\mu_l| \leqslant 18$. Also, applying Lemma~\ref{lemma: PGL_2(L), L <= 4,6,24}, obtain $|\overline{G}| \leqslant 84$. Finally, we have $$|G| = |\mu_r|\cdot|\overline{G}| \leqslant 1512.$$
\end{proof}

%\begin{lemma}\label{lemma: normality of G cap SL_3(Q)}
%    Let $G$ be a subgroup of $GL_n(\QQ)$, and $\varphi: \GL_n(\QQ) \rightarrow \SL_n(\RR)$ be a homomorphism defined by formula 
%$$\varphi(x) = \frac{x} { \sqrt[n]{\det(x)}}.$$ If $H = \varphi(G)$ and $H' = H \cap \SL_n(Q)$, then $H'$ is a normal subgroup of $H$.
%\end{lemma}

%\begin{proof}
%    Note, that $H'$ consists of elements
%\end{proof}

\section{$G$-mori fiber spaces with base of positive dimension}\label{sect: MFS with dimB>0}

In this section, we give a sharp bound on orders of finite subgroups acting faithfully on conic bundles and del Pezzo fibrations. But first, let us fix the notation and recall some definitions. 

From now on $G$ always denotes a finite group unless otherwise stated. If $G$ acts on a variety~$X$, then~$X$ is called a $G$-variety (action is assumed to be faithful unless otherwise stated). Morphisms between two $G$-varieties are assumed to be~$G$-equivariant. By~$G\QQ$-factorial variety we mean a $G$-variety such that any of its $G$-invariant Weil divisors is~\hbox{$\QQ$-Cartier}. A smooth $G$-surface $S$ is said to be $G$-minimal if any $G$-equivariant birational morphism from $S$ to a normal surface is an isomorphism, or, equivalently, if~$G$-invariant Picard number $\rho(S)^G$ equals 1.

Let $X$ be a terminal $G\QQ$-factorial variety of dimension $3$. Recall that a $G$-equivariant morphism $$f: X \rightarrow \mathcal{B}$$ is called a $G$-\textit{Mori fiber space} if
\begin{enumerate}
    \item $\mathcal{B}$ is a normal variety with $\dim \mathcal{B} < \dim X$;
    \item fibers of $f$ are connected;
    \item the anticanonical class $-K_X$ is $f$-ample;
    \item relative $G$-invariant Picard number $\rho(X/\mathcal{B})^G$ equals 1.
\end{enumerate}
If $\mathcal{B}$ is a point, then $-K_X$ is ample, and $X$ is said to be a $G\QQ$-Fano threefold with~$\relpenalty=10000 \rho(X)^G = 1$. If $\mathcal{B}$ is a curve, then a general fiber of $f$ is a del Pezzo surface; in this case~$X$ is said to be a $G\QQ$-del Pezzo fibration. If~$\mathcal{B}$ is a surface, then a general fiber of $f$ is a conic; in this case $X$ is said to be
a~$G\QQ$-conic bundle.

With each surjective $G$-equivaraint morphism with connected fibers $X \xrightarrow{} \mathcal{B}$ we associate an exact sequence of groups $$1 \xrightarrow{}{} G_F\xrightarrow{}{} G\xrightarrow{}{} G_B\xrightarrow{}{} 1,$$ where $G_B$ is a finite subgroup in $\Aut(\mathcal{B})$, and $G_F$ acts on $X$ preserving fibers.

The following lemma gives an alternative way to look at the group $G_F$ in some cases.

%\begin{lemma}\label{lemma: bound on cardinality of G in Bir(X), X - rational}
%Let $X$ be a normal geometrically rational surface over $\mathbb{Q}$ with $\mathbb{Q}$-point, and~\mbox{$G \subset \Aut(X)$} be a finite subgroup. Then $|G| \leqslant ?$.
%\end{lemma}

%\begin{lemma}
 %   Let $X \rightarrow \mathcal{B}$ be a rational $G$-Mori fiber space over $\mathbb{Q}$. Then there exist a smooth fiber $Y$ with $\mathbb{Q}$-point such that $G_F \subset \Aut(Y)$. 
%\end{lemma}

%\begin{proof}
 %   With each element $g \in G$ we can associate closed subset $Z_g \subset X$ consisting of fixed points under action of $g$. 
  %  Consider an open subset $U \subset X$ given by the formula $$U = X \setminus \bigcup\limits_{g\in G} Z_g,$$ where $Z_g$ 
%\end{proof}

\begin{lemma}\label{lemma: smooth rational fiber in fibrations}
Let $X$ be a smooth projective $G$-variety over a field $K$ of characteristic $0$. Let $$f: X \rightarrow \mathcal{B}$$ be a~$G$-equivariant morphism with connected fibers. Then there exists a smooth fiber~$Y$, such that $G_F \subset \Aut(Y)$. Moreover, if $K$-points are dense in $X$, then Y can be chosen to contain a $K$-point.
\end{lemma}

\begin{proof}
Firstly, consider open subset $V \subset \mathcal{B}$ over which fibers are smooth  (such $V$ exists because a fiber over generic point is smooth and we can apply~\cite[Proposition 17.7.8]{Grothendieck}), and denote preimage by $X_V$. Secondly, with each element $g \in G$ associate closed subset~$\relpenalty=10000 Z_g \subset X$ consisting of fixed points under action of $g$. Consider an open subset $U \subset X_V$ given by the formula $$U = X_V \setminus \bigcup\limits_{g\in G} Z_g$$ and denote any fiber intersecting $U$ by $Y$. Then $Y$ is smooth and $G_F \subset \Aut(Y)$ by the construction of $U$. If $K$-points are dense in $X$, we can find {$K$-point}~$P$ in $U$, and take fiber $Y$ passing through~$P$. 
\end{proof}

Since we work over algebraically nonclosed fields, the following Chatelet's theorem is useful.

\begin{theorem}[{see for example~\cite[Theorem 5.1.3]{Phil Gille}}]\label{theo: Chatelet's theorem}
    Let $X$ be a Severi--Brauer variety of dimension $n$ over a field $K$. Then $X$ is isomorphic to $\PP^n_K$ if and only if $X$ contains a~\mbox{$K$-point}.
\end{theorem}

Now, we bound orders of finite subgroups acting on geometrically rational surfaces with rational point. In particular, we need to bound orders of finite subgroups of automorphism groups of all del Pezzo surfaces. But before proving the general case, let us write down lemmas on del Pezzo surfaces of degree $6$ and $\PP^2_{\QQ}$.

\begin{lemma}[{\cite[Lemma 6.12]{Ahmed A}}]\label{lemma: P^2_Q}
    Let $G \subset \PGL_3(\QQ)$ be a finite subgroup. Then $|G| \leqslant 24$.
\end{lemma}

\begin{lemma}[{\cite[Lemma 4.1, Lemma 4.2]{Ahmed A}}]\label{lemma: dp_6}
     Let $X$ be a del Pezzo surface of degree $6$ over~$\QQ$. Let $G \subset \Aut(X)$ be a finite subgroup. Then $|G| \leqslant 432$. Moreover, there exist a rational del Pezzo surface $S$ such that the group $\Aut(S)$ contains a finite subgroup of order~$432$.
\end{lemma}

%\begin{lemma}[{\cite[Lemma 4.2]{Ahmed A}}]\label{lemma: dp6}
%    Let $X$ be a del Pezzo surface over $\QQ$ and $G \subset \Aut(X)$ be a finite subgroup. Then $|G| \leqslant 432$.
%\end{lemma}

The following lemma follows from results of~\cite{Ahmed A}, but we prove it for self-containedness. Also, for some del Pezzo surfaces in~\cite{Ahmed A} rationality was required, while we assume only existence of a $\QQ$-point.

\begin{lemma}\label{lemma: del Pezzo surfaces}
    Let $X$ be a del Pezzo surface over $\mathbb{Q}$ with a $\mathbb{Q}$-point and $G \subset \Aut(X)$ be a finite subgroup. Then~{$\relpenalty 10000 |G| \leqslant 432$}.
\end{lemma}

\begin{proof}
    Let us consider surfaces case by case, depending on their degree $d$.

    Let $d = 9$. In this case $X$ is a Severi--Brauer surface containing $\mathbb{Q}$-point. By Theorem~\ref{theo: Chatelet's theorem}, we obtain that $X$ is isomorphic to $\mathbb{P}^2_{\mathbb{Q}}$. By Lemma~\ref{lemma: P^2_Q}, we obtain $|G| \leqslant 24$.

    Let $d = 8$. There are two cases, either~$\overline{X}$ is isomorphic to a blow up of $\PP^2$ at a point, or~$\overline{X}$ is isomorphic to a smooth quadric. In the former case there is a \mbox{$G$-equivariant} blow down to $\PP^2_{\QQ}$, and we have $|G| \leqslant 24$ as in the case of degree $9$. By Theorem~\ref{theo: action with fixed point} we have~$\relpenalty=10000 G \subset \GL_2(\QQ)$ and $|G| \leqslant 12$ by Lemma~\ref{lemma: Burnside}. In the latter case, since~$X$ contains a~$\mathbb{Q}$-point, it is isomorphic either to~$\mathbb{P}^1_{\mathbb{Q}} \times \mathbb{P}^1_{\mathbb{Q}}$, or to Weil scalar restriction $R_{L/\mathbb{Q}}(\mathbb{P}^1_L)$, where~$\relpenalty=10000 \mathbb{Q} \subset L$ is a quadratic extension (see~\cite[Lemma 7.3]{SV}). In the former case automorphism group is isomorphic to~\mbox{$(\PGL_2(\mathbb{Q}) \times \PGL_2(\mathbb{Q})) \rtimes \mu_2$}, and the largest finite group it consists of $288$ elements due to Corollary~\ref{corollary: bound of finite subgroups of PGL_2(Q)}. In the latter case automorphism group is isomorphic to~\mbox{$\PGL_2(L) \rtimes \mu_2$} and the order of it's largest finite subgroup is bounded by~$48$ due to Lemma~\ref{lemma: PGL_2(Q(sqrt(d)))}. Therefore, in this case we have~$\relpenalty = 10000 |G| \leqslant 288$.

    Let $d = 7$. There is a $G$-equivariant contraction to $\PP^2_{\QQ}$, and we have $|G| \leqslant 24$ as in the case of degree $9$.
    
    Let $d = 6$. By Lemma~\ref{lemma: dp_6} we have~$|G| \relpenalty = 10000 \leqslant 432$.
    
    Let $d = 5$. By \cite[Theorem 8.5.6]{D} we have $\Aut(\overline{X}) \simeq \mathfrak{S}_5$, therefore~\mbox{$|G| \leqslant 120$}.

    Let $d=4$. By \cite[Theorem 8.6.6]{D} we have $\Aut(\overline{X}) \simeq \mu_2^4 \rtimes \Gamma$, where $|\Gamma| \leqslant 10$, therefore~$|G| \leqslant 160$.

    Let $d = 3$. By~\cite[Theorem 3]{Vikulova} we have $|\Aut(X)| \leqslant 120$, therefore~\mbox{$|G| \leqslant 120$}.

    Let $d = 2$. The~\cite[Table 8.9]{D} shows that $|\Aut(\overline{X})| \leqslant 336$, therefore $|G| \leqslant 336$.

    Let $d = 1$. In this case there is a $\mathbb{Q}$-point on $X$ which is invariant with respect to all automorphisms, namely, the intersection point of all divisors from the complete linear system $\mathopen|-K_X\mathclose|$. Therefore, $G$ embeds in~$\GL_2(\mathbb{Q})$ by Theorem~\ref{theo: action with fixed point}, and $|G| \leqslant 12$ by Lemma~\ref{lemma: Burnside}. \end{proof}

The following lemma serves the most general case we need for surfaces.

\begin{lemma}\label{lemma: bound of auto of rational normal surfaces}
    Let $\mathcal{B}$ be a normal projective geometrically rational surface over~$\mathbb{Q}$ with a~$\mathbb{Q}$-point. Let $G \subset \Aut(X)$ be a finite subgroup. Then $|G| \leqslant 432$.
\end{lemma}

\begin{proof}
Taking $G$-equivariant resolution of singularities, we obtain a smooth geometrically rational $G$-surface $S$ with~\mbox{$\mathbb{Q}$-point} due to the Lang--Nishimura theorem (see for example~\mbox{\cite[Theorem 3.6.11]{Poo}}). 
Furthermore, applying the minimal model program with an action of the group $G$, we can assume that $S$ is either a $G$-minimal del Pezzo surface or has a conic bundle structure $\pi: S\xrightarrow{}{} B$, and~$G\subset\Aut(S,\pi)$, see~\cite[Theorem 1G]{Isk}. Note that in the latter case, the curve $B$ is geometrically rational and contains a $\mathbb{Q}$-point, since~$S$ contains a $\mathbb{Q}$-point. Hence, we have an isomorphism $B\simeq\mathbb{P}^1_{\mathbb{Q}}$.

If $S$ is a del Pezzo surface, then by Lemma~\ref{lemma: del Pezzo surfaces}, we have $|G| \leqslant 432$. If $S$ has a conic bundle structure, then we have a standard exact sequence $$1 \xrightarrow{}{} G_F\xrightarrow{}{} G\xrightarrow{}{} G_B\xrightarrow{}{} 1,$$ where $G_B$ is a finite subgroup in $\Aut(B) \simeq \PGL_2(\mathbb{Q})$, and $G_F$ is a finite subgroup of automorphism group of a smooth conic by Lemma~\ref{lemma: smooth rational fiber in fibrations}. Applying Corollary~\ref{corollary: bound of finite subgroups of PGL_2(Q)} and Corollary~\ref{corollary: automorphsms of conic over Q} we obtain inequalities $|G_B| \leqslant 12$ and $|G_F| \leqslant 24$, which implies inequality~$\relpenalty=10000 |G| \leqslant 288$. \end{proof}

The last lemma of this section serves the case of rational threefolds, which has a structure of a $G\QQ$-Mori fiber space over the base of positive dimension.

\begin{lemma}\label{lemma: GQ-conic and GQ-del Pezzo}
    Let $G$ be a finite group and $X$ be a terminal $G\QQ$-factorial rational threefold over $\QQ$ with a structure of either a $G\QQ$-conic bundle or a $G\QQ$-del Pezzo fibration~$\relpenalty =10000 f: X \rightarrow{}{} \mathcal{B}$. Then~$\relpenalty = 10000 |G| \leqslant 5184$ and this bound is sharp. 
\end{lemma}

\begin{proof}
    Immediately note, that $\QQ$-points are dense on $X$, since $X$ is rational. Consider a standard short exact sequence $$1 \xrightarrow{}{} G_F\xrightarrow{}{} G\xrightarrow{}{} G_B\xrightarrow{}{} 1.$$ 
    
    If $f$ is a conic bundle, then $\mathcal{B}$ is a normal projective geometrically rationally connected, and therefore, geometrically rational surface with a $\QQ$-point. Also we have $G_B \subset \Aut(\mathcal{B})$ and~$G_F \subset \PGL_2(\QQ)$ by Lemma~\ref{lemma: smooth rational fiber in fibrations}. Applying Lemma~\ref{lemma: bound of auto of rational normal surfaces} and Corollary~\ref{corollary: bound of finite subgroups of PGL_2(Q)} we obtain~$\relpenalty = 10000 |G| \leqslant 5184$. 

    If $f$ is a del Pezzo fibration, then $\mathcal{B}$ is a smooth projective geometrically rational curve with a $\QQ$-point, which implies that $\mathcal{B} \simeq \PP^1_{\QQ}$. Also we have $G_B \subset \PGL_2(\QQ)$ and~$ \relpenalty= 10000 G_F \subset \Aut(Y)$, where $Y$ is a smooth del Pezzo surface with $\QQ$-point by Lemma~\ref{lemma: smooth rational fiber in fibrations}. Again, applying Lemma~\ref{lemma: bound of auto of rational normal surfaces} and Corollary~\ref{corollary: bound of finite subgroups of PGL_2(Q)} we obtain~$ \relpenalty = 10000 |G| \leqslant 5184$. 

    It remains to present a threefold, containing a finite subgroup of order $5148$ in its automorphism group. To do this, take a del Pezzo surface $S$ from Lemma~\ref{lemma: dp_6} and put~$\relpenalty=10000 X = S \times \PP^1_{\QQ}$. Since $S$ is rational, $X$ is also rational and has both structures of a conic bundle and a del Pezzo fibration. There is a finite subgroup $H \subset \Aut(S)$ containing $432$ elements according to Lemma~\ref{lemma: dp_6}. Therefore, $\Aut(X)$ contains a subgroup isomorphic to~$\binoppenalty=10000 H \times D_{12}$, containing $5148$ elements.
\end{proof}

%\begin{color}{blue}
%Если с леммами 3.1 и 3.3 все в порядке, то они вроде покрывают случаи, когда ММП для трифолда заканчивается на расслоении Мори с одномерной и двумерной базой. 
%\end{color}

%\section{Action on a tangent space to a singular point}

%As we have already mentioned in Section~\ref{sect: GL and PGL}, if we want to understand something about groups acting on a variety, it is useful to look at an induced action on a tangent space to a fixed point. No one guarantees that finite groups we are studying act on varieties with fixed points. But in case of singular varieties, we can get a fixed point by considering an extension of the base field and passing to subgroups.

\section{Non-Gorenstein Fano threefolds}\label{sect: non-Gorenstein Fano}
 From now on, we bound orders of finite groups acting on $G\mathbb{Q}$-factorial terminal Fano threefold. Recall that Fano variety is a projective variety whose anticanonical class is ample. Recall that variety is said to be Gorenstein if its canonical class is a Cartier divisor. In this section, we bound numbers of singular points on terminal non-Gorenstein Fano threefolds and bound orders of finite groups acting on such varieties.

Let us immediately note that singularities of a terminal threefold are isolated points. In particular, there are finitely many singular points on such threefolds. We start with the following lemma.

\begin{lemma}\label{lemma: passing to stabilizer and acting on a tangent space}
    Let $X$ be a variety over $\QQ$ such that $\overline{X}$ has $d$ singular points. Let $P$ be a singular point on $\overline{X}$, and $G \subset \Aut(X)$ be a finite subgroup. Then there exists a field extension $\QQ \subset K$ with $[K:\QQ] \leqslant d$ such that $P$ is a $K$-point. Also, there is a subgroup~$G_{\bullet} \subset G$ with~$[G:G_{\bullet}] \leqslant d$ such that~$G_{\bullet}$ acts on $X_K$ fixing~$P$ and, in particular, embeds in $\GL(T_PX_{K})$.
\end{lemma}

\begin{proof}    
    The group $\text{Gal} (\overline{\QQ}/\QQ)$ acts on singular points of $\overline{X}$. Denote by $H$ the stabilizer of~$P$, then $H$ is an open (and therefore closed) subgroup of index at most~$d$. Put~$\relpenalty= 10000 K=(\overline{\QQ})^H$, then $\QQ \subset K$ is a finite extension of degree at most $d$, see~\hbox{\cite[\href{https://stacks.math.columbia.edu/tag/0BML}{Theorem 0BML}]{stacks-project}} and $P$ is a $K$-point. If we consider $G$ as a finite subgroup of~$\Aut(X_K)$, then $G$ acts on singular $K$-points and we denote by $G_{\bullet}$ the stabilizer of~$P$. Clearly, we have $[G:G_{\bullet}] \leqslant d$, and $P$ is fixed under the action of $G_{\bullet}$ on $X_K$. Moreover,~$G_{\bullet}$~embeds in $\GL(T_PX_{K})$ by Theorem~\ref{theo: action with fixed point}.
\end{proof}

Note that the result of Lemma~\ref{lemma: passing to stabilizer and acting on a tangent space} is more pleasant for us as soon as the number of singularities is smaller (for example, if there is only one singular point on $\overline{X}$, then Lemma~\ref{lemma: passing to stabilizer and acting on a tangent space} provides a representation over $\QQ$ of a whole given finite subgroup). In this regard, it is important for us to bound the number of singular points on Fano threefolds. To do this, we use the orbifold Riemann--Roch theorem from~\cite{Reid}. Recall that the index of a singular point $P$ on a variety $X$ is a minimal positive integer number $r$ such that~$rK_X$ is Cartier at $P$.

\begin{theorem}[{\cite[Corollary 10.3]{Reid}}]\label{theo: Riemann-Roch}
    Let $X$ be a projective threefold over algebraically closed field with canonical singularities and let $$B_X = \{ P_{\alpha} \in X_{\alpha} \; \textit{of index} \; r_{\alpha} \}$$ be the basket for $X$ and $K_X$ in the sense of~\cite[Theorem 10.2(3)]{Reid}. Then $$\chi(\mathcal{O}_X) = \frac{1}{24}(-K_X)\cdot c_2(X) + \frac{1}{24}\sum_{P_{\alpha} \in B_X} \left( r_{\alpha} - \frac{1}{r_{\alpha}}\right).$$
\end{theorem}

\begin{cor}\label{corollary: number of singularities on non-Gorenstein Fano threefold}
    Let $X$ be a terminal non-Gorenstein Fano threefold. Let $N$ be the number of singular points on $\overline{X}$. Then, one has $N \leqslant 15$. Moreover, if there are no cyclic quotient singularities on $\overline{X}$, then $N \leqslant 7$.
\end{cor}
\begin{proof}
    Since $X$ is Fano, we have $\chi\left( \mathcal{O}_{\overline{X}} \right) = 1$ by Kawamata--Viehweg theorem and~$\relpenalty=10000 \binoppenalty = 10000 -K_{\overline{X}}\cdot c_2\left(\overline{X}\right) > 0$ (see~\cite{Kaw}). Then Theorem~\ref{theo: Riemann-Roch} gives an inequality $$\frac{3}{2}N \leqslant \sum \left( r_j - \frac{1}{r_j}\right) < 24,$$ which implies $N \leqslant 15$. If there are no cyclic quotient singularities on $\overline{X}$, then each basket of each point contains at least two points (one can check this deforming the equations from Mori's classification theorem~\hbox{\cite[\S 6, Theorem(S. Mori)]{Reid}}) and actually we have $N \leqslant 7$.
\end{proof}

The following lemma is needed in cases where a given singular threefold does not contain cyclic quotient singularities.

\begin{lemma}\label{lemma: cyclic cover}
    Let $K$ be a field of characteristic $0$. Let $X$ be a terminal threefold over $K$ and $P \in X$ be a singular $K$-point of index $r$. Let $G \subset \Aut(X)$ be a finite subgroup, which fixes $P$. Then there is a short exact sequence $$1 \rightarrow \mu_r \rightarrow \tilde{G} \rightarrow G \rightarrow 1,$$ where $\tilde{G} \subset \GL_4(K)$. Moreover, if $P$ is a cyclic quotient singularity, then $\tilde{G} \subset \GL_3(K)$.
\end{lemma}

\begin{proof}
    Replacing $X$ by a smaller $G$-invariant neighborhood $U$, we may assume that~$rK_U \sim 0$. Consider one-index cover $\pi \colon W \to U$ from~\cite[Proposition 3.6]{Reid}. Then~$Q  = \pi^{-1}(P)$ is a terminal $K$-point of index $1$. Therefore, $Q$ is a hypersurface singularity by~\cite[Corollary 3.12(i)]{Reid}. Moreover, $W$ is smooth at $Q$ if $P$ is a cyclic singularity. By construction of $W$ each element $g \in G$ admits $r$ lifts to $\Aut(W)$. Therefore, we have an exact sequence $$1 \rightarrow \mu_r \rightarrow \tilde{G} \rightarrow G \rightarrow 1,$$ where $\tilde{G} \subset \Aut(W)$, which fixes $Q$. By Theorem~\ref{theo: action with fixed point} we have $\tilde{G} \subset \GL_4(K)$ in general case and $\tilde{G} \subset \GL_3(K)$ if $Q$ is smooth.
\end{proof}

Suppose we are given a finite group $G$ acting on a terminal \hbox{non-Gorenstein} $G\QQ$-factorial Fano threefold $X$. In this case, our strategy is to pass to a small index subgroup, which acts on $X$ with a fixed singular $K$-point, where $\mathbb{Q} \subset K$ is an extension of small degree.

\begin{propos}\label{prop: bound for non-Gorenstein}
    Let $G$ be a finite group. Let $X$ be a terminal $G\QQ$-factorial non-Gorenstien Fano threefold over $\QQ$. Then $$|G| \leqslant 3\,962\,649\,600 < 10^{10}.$$
\end{propos}

\begin{proof}
Denote by $N$ the number of singularities on $\overline{X}$. If there are no cyclic quotient singularities on $\overline{X}$, then $N \leqslant 7$ by Corollary~\ref{corollary: number of singularities on non-Gorenstein Fano threefold}. Then by Lemma~\ref{lemma: passing to stabilizer and acting on a tangent space}, there is a field extension~$\QQ \subset K$ of degree~$d \relpenalty=10000 \leqslant 7$ and a subgroup $G_{\bullet} \subset G$ of index $i \leqslant 7$ such that~$X_K$ contains a singular~\hbox{$K$-point} $P$ and $G_{\bullet}$ acts on $X_K$ fixing $P$. Therefore, applying Lemma~\ref{lemma: cyclic cover}, we obtain an exact sequence $$1 \rightarrow \mu_r \rightarrow \Tilde{G}_{\bullet}\rightarrow G_{\bullet} \rightarrow 1,$$
where $r \geqslant 2$ is the index of $P$ and $\Tilde{G}_{\bullet} \subset \GL_4(K)$. Now, applying Corollary~\ref{corollary: additive bound GL_4(<=7)} we obtain~$|\tilde{G}_{\bullet}| \leqslant 1\,132\,185\,600$, which implies $$|G| \leqslant  \frac{i}{r}\cdot1\,132\,185\,600\leqslant \frac{7}{2} \cdot 1\,132\,185\,600 =  3\,962\,649\,600 < 10^{10}.$$

If there is a cyclic singular point $P$, then by Lemma~\ref{lemma: passing to stabilizer and acting on a tangent space}, there is a field extension~$\QQ \subset K$ of degree~$d \relpenalty=10000 \leqslant 15$ and a subgroup $G_{\bullet} \subset G$ of index $i \leqslant 15$ such that $P$ is a $K$-point and~$G_{\bullet}$ acts on $X_K$ fixing $P$. Again, applying Lemma~\ref{lemma: cyclic cover}, we obtain an exact sequence $$1 \rightarrow \mu_r \rightarrow \Tilde{G}_{\bullet}\rightarrow G_{\bullet} \rightarrow 1,$$
where $r \geqslant 2$ is an index of $P$ and $\Tilde{G}_{\bullet} \subset \GL_3(K)$. Now, applying Corollary~\ref{corollary: additive bound GL_3(<=15)} we obtain~$|\tilde{G}_{\bullet}| \leqslant 240\,045\,120 < 10^{9}$, which implies $$|G| \leqslant \frac{i}{r}\cdot 240\,045\,120 \leqslant \frac{15}{2}\cdot 240\,045\,120 = 1\,800\,338\,400 < 10^{10}.$$
\end{proof}

\section{Gorenstein Fano threefolds}\label{sect: Gorenstein Fano}
In this section we bound orders of finite groups acting on terminal Gorenstein Fano threefolds and prove Theorem~\ref{theo: cr_3}. Let us briefly recall the definitions of important invariants and fix the notation.

Let~$X$ be a Gorenstein Fano variety of dimension $n$. The \textit{index} $\iota(X)$ of $X$ is the maximal integer $i$ such that the equality~$-K_X = iD$ holds for some class of Cartier divisors $D$. The class of Cartier divisors~$H$ which satisfies the equality $-K_X = \iota(X)\cdot H$ is uniquely defined and is called the fundamental class. The number $d(X) = H^n$ is called the \textit{degree} of~$X$ and the number $g(X) = \frac{1}{2}H^n + 1$ is called the \textit{genus}.
\begin{rem}
    If $X$ is a Gorenstein variety over algebraically non-closed field $K$ of characteristic $0$, then clearly~$\iota(X)$ divides $\iota(\overline{X})$, but a priori these numbers could differ. However, if $X$ contains a~\mbox{$K$-point}, then $\iota(X) = \iota(\overline{X})$. To see this consider the exact sequence of groups \hbox{(see~\cite[Exercise 3.3.5(iii)]{G-S})}: $$0 \xrightarrow{} \text{Pic}(X) \xrightarrow{} \text{Pic}(X_{\overline{K}})^{\Gal(\overline{K}/K)} \xrightarrow{} \text{Br}(K) \xrightarrow{} \text{Br}(K(X))$$ and note that the last homomorphism has a section since $X$ contains a $K$-point. In particular, the last homomorphism is injective and we have an isomorphism $$\text{Pic}(X) \simeq \text{Pic}(X_{\overline{K}})^{\Gal(\overline{K}/K)},$$ which implies $\iota(X) = \iota(\overline{X})$, $d(X) = d(\overline{X})$ and $g(X) = g(\overline{X})$.

    Further, when working with varieties, which contain rational points, we will assume these equalities without references.
\end{rem}

The main advantage in the Gorenstein case is the existence of smoothing. 
\begin{theorem}[{\cite[Theorem 11]{Nam}}, {\cite[Theorem 1.4]{JR}}, {\cite[Proposition 2.5]{KuzPro}}]\label{theo: smoothing}
    Let $X$ be a terminal Gorenstein Fano threefold over $\CC$ with $N$ singular points. Then there exist a smoothing, that is, a flat projective morphism $$f: \mathfrak{X} \rightarrow B \ni 0$$ such that the general fiber $\mathfrak{X}_b$ is smooth and the central fiber $X_0$ is isomorphic to $X$. Moreover, $\Pic(X) \simeq \Pic(\mathfrak{X}_b)$, $\iota(X) = \iota(\mathfrak{X}_b)$, $g(X) = g(\mathfrak{X}_b)$ and $d(X) = d(\mathfrak{X}_b)$. Also there is a bound $$ N \leqslant 21 - \frac{1}{2}\chi_{\top}(\mathfrak{X}_b) = 20 - \rho(\mathfrak{X}_b) + h^{1,2}(\mathfrak{X}_b).$$
\end{theorem}

The following theorem specifies families of threefolds, which are interesting in terms of applications to our problem. 

\begin{theorem}[{cf. \cite[Theorem 6.6]{Pro II}}]\label{theo: prokhorov on G-Fanos}
    Let $K$ be a field of characteristic $0$. Let $X$ be a terminal Gorenstein Fano threefold over $K$ such that $\rho(\overline{X}) > 1$. Let $G$ and $G'$ be two finite groups acting on $\overline{X} = X \otimes_K \overline{K}$ through the first and through the second factor respectively. If $\Cl(\overline{X})^{G\times G'} \simeq \ZZ$, then $\overline{X}$ belongs to one of $8$ families from Table~\ref{table: possible minimal by Pro}.
\end{theorem}
\begin{rem}\label{remark: family (1.2.3) in Pro II}
  In~\cite{Pro II} family $(1.2.3)$ of smooth threefolds has a different description than in Table~\ref{table: possible minimal by Pro}, but this is the same family.
\end{rem}

\begin{table}[h]
\def\hline{\noalign{\hrule}}
\def\0{\hphantom0}
\def\d#1{\0\hbox to 0pt{\hss#1\hss}\0}
\def\vt#1{\vtop{\hsize=.6\textwidth\parindent0pt\normalbaselines
\parfillskip0pt plus 1fil\leftskip0pt\rightskip0pt
#1\unskip\nobreak\hskip0pt\lower3pt\copy\strutbox\par}}
{\offinterlineskip\tabskip0pt\relax
\halign to\textwidth{\vrule #\tabskip0pt plus 1fil\relax
&\hfil#\hfil&\vrule #%
&\hfil#\hfil&\vrule #%
&\hfil#\hfil&\vrule #%
&\hfil#\hfil&\vrule #%
&\hfil#\hfil&\vrule #%
&\hfil#\hfil&\vrule #%
&\hfil\vt{#}\hfil&\raise3pt\copy\strutbox\lower3pt\copy\strutbox
\vrule #\tabskip0pt\relax\cr
\hline
&\multispan3{\hfil No in\hfil}&&\multispan7{\hfil \hbox to 0pt{\hss Value \hss}\hfil}&&\omit&\cr
\multispan{13}{\hrulefill}&\omit&\omit\vrule\cr
&\cite{Pro II}&&\hbox to 33pt{\hss \cite{Isk Pro},\hskip1pt\cite{fanography}\hss}&&$\rho$&&$\iota$&&$g$&&\hbox to 0pt{\hss$h^{1,2}$\hss}&&\centering \leavevmode\smash{\raise10pt\hbox{Description of $X$}}&\cr
\hline
&(1.2.1)&&\href{https://www.fanography.info/2-6}{2-6\0}&&\d2&&\d1&&\07&&\d9&&Let $Z_6 \subset \mathbb{P}^8$ be the image of the Segre embedding of $\mathbb{P}^2 \times \mathbb{P}^2$ and let $Y_6 \subset \mathbb{P}^9$ be the projective cone over $Z_6$. Then $X$ is an intersection of $Y_6$ with a hyperplane and a quadric.&\cr
\hline
&(1.2.2)&&\href{https://www.fanography.info/2-12}{2-12}&&2&&1&&11&&3&&$X$ is an intersection of three divisors of bidegree $(1, 1)$ in $\mathbb{P}^3 \times \mathbb{P}^3$.&\cr
\hline
&(1.2.3)&&\href{https://www.fanography.info/2-21}{2-21}&&2&&1&&15&&0&& a) $X$ is smooth and isomorphic to a blow up of a smooth quadric $Q \subset \PP^4$ in a twisted quartic; \endgraf
b) $X$ is singular and isomorphic to a blow up of a quadratic cone $\relpenalty=10000 Q' \subset \PP^4$ with center a union of two conics that do not pass through the vertex and meet each other transversely.&\cr
\hline
&(1.2.4)&&\href{https://www.fanography.info/2-32}{2-32}&&2&&2&&25&&0&&$X \subset \mathbb{P}^2 \times \mathbb{P}^2$, a smooth divisor of bidegree $(1, 1)$.&\cr
\hline
&(1.2.5)&&\href{https://www.fanography.info/3-1}{3-1\0}&&3&&1&&\07&&8&&$X$ is a double cover of $\mathbb{P}^1 \times \mathbb{P}^1 \times \mathbb{P}^1$ whose branch locus is a member of $\mathopen|-K_{\mathbb{P}^1 \times \mathbb{P}^1 \times \mathbb{P}^1}\mathclose|$.&\cr
\hline
&(1.2.6)&&\href{https://www.fanography.info/3-13}{3-13}&&3&&1&&16&&0&&$X$ is an intersection of divisors of tridegrees $(0,1,1)$, $(1,0,1)$, $(1,1,0)$ in $\mathbb{P}^2 \times \mathbb{P}^2 \times \mathbb{P}^2$.&\cr
\hline
&(1.2.7)&&\href{https://www.fanography.info/3-27}{3-27}&&3&&2&&25&&0&&$X = \mathbb{P}^1 \times \mathbb{P}^1 \times \mathbb{P}^1$.&\cr
\hline
&(1.2.8)&&\href{https://www.fanography.info/4-1}{4-1\0}&&4&&1&&13&&1&&$X \subset \mathbb{P}^1 \times \mathbb{P}^1 \times \mathbb{P}^1 \times \mathbb{P}^1$ is a divisor of multidegree $(1,1,1,1)$.&\cr
\hline
}}\medskip
\caption{\centering Families of terminal Gorenstein Fano threefolds $X$ with $\rho(X) > 1$.}
\label{table: possible minimal by Pro}
\end{table}

The following lemma is useful for threefolds with small genus.

\begin{lemma}\label{lemma: action on sections}
    Let $k$ be a field of characteristic $0$. Let $X$ be a Gorenstein variety over $k$ such that the map $\Phi_{\mathopen|-K_X\mathclose|}$ given by the anticanonical linear system is regular and has degree~$d$. Let~$\relpenalty=10000 G \subset \Aut(X)$ be a subgroup. Then $G$ has a representation in $\GL(H^0(X,-K_X)^{\vee})$ with kernel of order at most $d!$. In particular, if $-K_X$ is very ample, then $G$ embeds to~$\GL(H^0(X,-K_X)^{\vee})$. 
\end{lemma}
\begin{proof}
    Denote by $X^{\sm}$ the smooth locus of $X$. Then  the group $G$ acts linearly on a space of vector fields on $X^{\sm}$ giving a group homomorphism $\alpha: G \rightarrow \GL(H^0(X,-K_X))$ and, therefore, a group homomorphism $\alpha^{\vee}:  G \rightarrow \GL(H^0(X,-K_X)^{\vee})$. It is left to show that~$\mathopen|\ker\alpha^{\vee}\mathclose| \leqslant d!$. 

Denote the image of $\Phi_{\mathopen|-K_X\mathclose|}$ by $Y$. Since the map $\Phi_{\mathopen|-K_X\mathclose|}$ is $G$-equivariant, then we have a short exact sequence of groups $$1 \xrightarrow{}{} G'\xrightarrow{}{} G\xrightarrow{}{} G_Y\xrightarrow{}{} 1,$$ where $G_Y$ acts faithfully on $Y$ and $G'$ permutes points in fibers. Since $G$ is finite, there exist a fiber over $\overline{k}$ such $G'$ acts faithfully on this fiber, i.e. there is an embedding of~$G'$ to a symmetric group $\mathfrak{S}_d$ and $|G'| \leqslant d!$. Note that elements from $\mathopen|\ker\alpha^{\vee}\mathclose|$ act trivially on~$Y$. Indeed, if we have an element $g\in G$ such that $\alpha^{\vee}(g) = \id$. Then composing with a natural projection $$p: \GL(H^0(X,-K_X)^{\vee}) \rightarrow \PGL(H^0(X,-K_X)^{\vee})$$ we obtain an equality $(p \circ \alpha^{\vee})(g) = \id$ and, in particular, $\relpenalty=10000 g = (p \circ \alpha^{\vee})(g)|_Y = \id$. Therefore, we have~$\relpenalty=10000 \ker \alpha^{\vee} \subset G'$ and $\mathopen|\ker\alpha^{\vee}\mathclose| \leqslant d!$.
\end{proof}

%Now, let us bound orders of finite subgroups of automorphism groups of threefolds of introduced types.

\begin{lemma}\label{lemma: (1.2.1) and (1.2.5)}
    Let $X$ be a terminal Gorenstein Fano threefold over $\QQ$ with $g(X)=7$. Let~$\relpenalty=10000 G \subset \Aut(X)$ be a finite subgroup. Then $$|G| \leqslant 5\,573\,836\,800 < 10^{10}.$$ 
\end{lemma}

\begin{proof}
We have $\dim H^0(X,-K_X) = g(X)+2 = 9.$ Then consider an anticanonical map $$\Phi_{\mathopen|-K_X\mathclose|}: X \dashrightarrow \PP^8.$$ 
    
    If $\Bs\mathopen|-K_X\mathclose| \neq \varnothing$, then there are two possibilities according to~\cite[Proposition 2.4.1]{Isk Pro}. Either $\Bs\mathopen|-K_X\mathclose|$ consists of one singular rational point $P$, or $\Bs\mathopen|-K_X\mathclose|$ is a smooth conic~$C$. In the former case group $G$ fixes point $P$, therefore $|G| \leqslant 5760$ by Theorem~\ref{theo: action with fixed point} and Theorem~\ref{theo: Minkowski}. In the latter case there exist a short exact sequence $$1 \rightarrow G' \rightarrow G \rightarrow G'' \rightarrow 1,$$ where $G'$ acts trivially on $C$ and $G'' \subset \Aut(C)$. We have $|G''| \leqslant 24$ by Corollary~\ref{corollary: automorphsms of conic over Q}, so it is left to bound the order of $G'$. Consider a field extension $\QQ \subset K$ of order $2$ such that~$C_K$ is rational and choose $K$-point $Q$ in $C_K \setminus \Sing(X)$. Then $Q$ is fixed by $G'$ and~$\relpenalty=10000 G' \subset \GL_3(K)$ by Theorem~\ref{theo: action with fixed point}. Therefore, $\relpenalty = 10000|G'| \leqslant 2\,903\,040$ by Corollary~\ref{corollary: additive bound GL_3(<=2)} and $$|G| = |G''|\cdot|G'| \leqslant 24 \cdot 2\,903\,040 = 69\,672\,960 < 10^8.$$

    If $\Bs\mathopen|-K_X\mathclose| = \varnothing$, then $\Phi_{\mathopen|-K_X\mathclose|}$ has degree $1$ or $2$ by~\cite[Proposition 2.1.15]{Isk Pro}. Therefore, by Lemma~\ref{lemma: action on sections} group $G$ has a representation in $H^0(X,-K_X) \simeq \QQ^9$ with kernel of order at most $2$. Therefore, $$|G| \leqslant 2\cdot 2^{16}\cdot3^5\cdot5^2\cdot7 = 
5\,573\,836\,800 < 10^{10}$$ by Theorem~\ref{theo: Minkowski}.
\end{proof}

In the following $8$ lemmas we bound finite subgroups of forms of varieties from Table~\ref{table: possible minimal by Pro}.

\begin{lemma}\label{lemma: (1.2.7)}
    Let $\QQ \subset K$ be a field extension of degree at most $4$ (possibly,~$\relpenalty=10000 K = \QQ$). Let~$X$ be a terminal Gorenstein Fano threefold over $K$ such that $\overline{X} \simeq \PP^1 \times \PP^1 \times \PP^1$. Let~$\relpenalty=10000 G \subset \Aut(X)$ be a finite subgroup. Then $$|G| \leqslant 444\,528\,000 < 10^9.$$
\end{lemma}

\begin{proof}
    The Mori cone of $\overline{X}$ contains three extremal rays corresponding to three $\PP^1$-bundle structures over $\PP^1\times \PP^1$. The group $\Gal(\overline{\QQ}/K)$ acts on these extremal rays, denote by~$\relpenalty = 10000 H \subset \Gal(\overline{\QQ}/K)$ the kernel of this action. Then $H$ has index at most $6$, field of invariants $L = (\overline{\QQ})^{H}$ has degree at most $4 \cdot 6 = 24$ over~$\QQ$ and $\rho(X_K) = 3$. Also, each extremal contraction is defined over $L$. Consider one of them
    $$f : X_L \rightarrow X',$$ which corresponds to a ray $R$.
    Then $X'$ is a surface over $L$ with $\rho(X') = 2$, containing an~\hbox{$L$-point}, such that $\overline{X'} \simeq \PP^1 \times \PP^1$. By~\cite[Lemma 7.3]{SV} we have $X' \simeq \PP^1_L \times \PP^1_L$. Moreover, group $G$ also acts on the Mori cone, and if we denote by $G' \subset G$ the stabilizer~$R$, then~$[G:G'] \leqslant 3$ and $f$ is $G'$-equivariant. By Lemma~\ref{lemma: smooth rational fiber in fibrations} there exists a short exact sequence $$1 \rightarrow G'_F \rightarrow G' \rightarrow G'_B \rightarrow 1,$$ where $G'_F \subset \PGL_2(L)$ and $G'_B \subset (\PGL_2(L) \times \PGL_2(L)) \rtimes \mu_2$. Then $|G'_F| \leqslant 420$ and~$\relpenalty=10000 |G'_B| \leqslant 420\cdot420\cdot2 = 352\,800$ both by Lemma~\ref{lemma: PGL_2(L), L <= 4,6,24}. Finally, $$|G| = [G:G']\cdot|G'| \leqslant 3 \cdot 420 \cdot 352\,800  = 444\,528\,000 < 10^9.$$
\end{proof}

\begin{lemma}\label{lemma: (1.2.2)}
    Let $X$ be a terminal Gorenstein Fano threefold over $\QQ$ with a $\QQ$-point such that $\overline{X}$ is isomorphic to an intersection of three divisors of bidegree $(1,1)$ in~$\PP^3 \times \PP^3$. Let~$\relpenalty=10000 G \subset \Aut(X)$ be a finite subgroup. Then $$|G| \leqslant 1\,886\,976\,000 < 10^{10}.$$
\end{lemma}

\begin{proof}
The Mori cone of $\overline{X}$ contains two extremal rays corresponding to divisorial contractions to $\PP^3$. The group $\Gal(\overline{\QQ}/\QQ)$ acts on these extremal rays, denote by~$\relpenalty = 10000 H \subset \Gal(\overline{\QQ}/\QQ)$ the kernel of this action. Then $H$ has index at most $2$ and the field of invariants $K = (\overline{\QQ})^{H}$ has degree at most $2$ over~$\QQ$. Also, each extremal contraction is defined over $K$. Consider one of them
$$f : X_K \rightarrow X',$$ which corresponds to a ray $R$. Then $X'$ is a Severi--Brauer threefold with a $K$-point. By Theorem~\ref{theo: Chatelet's theorem}, we obtain that $X'$ is isomorphic to $\mathbb{P}^3_{K}$. Moreover, group $G$ also acts on the Mori cone, and if we denote by $G' \subset G$ the stabilizer of $R$, then~$[G:G'] \leqslant 2$, contraction~$f$ is $G'$-equivariant and $G' \subset \PGL_4(K)$. Then $|G'| \leqslant 943\,488\,000$ by Corollary~\ref{corollary: serr for PGL_4(<=2)} and $$|G| = [G:G'] \cdot |G'| \leqslant 1\,886\,976\,000 < 10^{10}.$$
\end{proof}

\begin{lemma}\label{lemma: (1.2.4)}
Let~$X$ be a smooth Fano threefold over $\QQ$ with a $\QQ$-point such that $\overline{X}$ is isomorphic to a divisor of bidegree $(1,1)$ in $\PP^2 \times \PP^2$. Let~$\relpenalty = 10000 G \subset \Aut(X)$ be a finite subgroup. Then $$|G| \leqslant  125\,798\,400 < 10^{9}.$$
\end{lemma}
\begin{proof}
    The Mori cone of $\overline{X}$ contains two extremal rays corresponding to two \mbox{$\PP^1$-bundle} structures over $\PP^2$. The group $\Gal(\overline{\QQ}/\QQ)$ acts on these extremal rays, denote by~$\relpenalty = 10000 H \subset \Gal(\overline{\QQ}/\QQ)$ the kernel of this action. Then $H$ has index at most $2$ and the field of invariants $K = (\overline{\QQ})^{H}$ has degree at most $2$ over~$\QQ$. Also, each extremal contraction is defined over $K$. Consider one of them
    $$f : X_K \rightarrow X',$$ corresponding to a ray $R$. Then $X'$ is a Severi--Brauer surface with an $K$-point. By Theorem~\ref{theo: Chatelet's theorem}, we obtain that $X'$ is isomorphic to $\mathbb{P}^2_{K}$. Moreover, group $G$ also acts on the Mori cone, and if we denote by $G' \subset G$ the stabilizer of $R$, then~$[G:G'] \leqslant 2$ and the contraction $f$ is $G'$-equivariant. By Lemma~\ref{lemma: smooth rational fiber in fibrations} there exists a short exact sequence  $$1 \rightarrow G'_F \rightarrow G' \rightarrow G'_B \rightarrow 1,$$ where $G'_F \subset \PGL_2(K)$ and $G'_B \subset \PGL_3(K)$. Then $|G'_F| \leqslant 24$  by Lemma~\ref{lemma: PGL_2(Q(sqrt(d)))} and $$|G'_B| \leqslant 2\,620\,800 < 10^{7}$$ by Corollary~\ref{corollary: serr for PGL_3(<=2)}. And the final bound is $$|G| = [G:G']\cdot|G'| \leqslant 2 \cdot 24 \cdot 2\,620\,800 = 125\,798\,400 < 10^{9}.$$
\end{proof}

\begin{lemma}\label{lemma: (1,1) in P^2 x P^2 fixing curve}
    Let $\QQ \subset K$ be a field extension of degree at most $3$ (possibly, $K = \QQ$). Let~$W$ be an irreducible projective threefold over $K$ with a $K$-point such that~$\relpenalty=10000 \overline{W} \subset \PP^2 \times \PP^2$ is a divisor of bidegree $(1,1)$. Let $C \subset W$ be a curve of bidegree $(2,2)$ in $\PP^2 \times \PP^2$. Let~$\relpenalty=10000 G \subset \Aut(W)$ be a finite subgroup preserving curve $C$. Then $$|G| \leqslant 21\,337\,344 < 10^8.$$
\end{lemma}

\begin{proof}
    %If $W$ is singular, then choosing homogeneous coordinates $(X_0:X_1:X_2)$ on the first factor and $(Y_0:Y_1:Y_2)$ on the second, we can assume that $\overline{W} \subset \PP^2 \times \PP^2$ is given by equation $\{X_0Y_0 - X_1Y_1 = 0\}$. In this case, $\overline{W}$ conains only one singular point $P$. Therefore,~$P$ is defined over $\QQ$, also $P$ is fixed by $G$, and $\dim T_PX = 4$. By Theorem~\ref{theo: action with fixed point}, group $G$ embeds to $\GL_4(\QQ)$, and $|G| \leqslant 5\,760$ by Theorem~\ref{theo: Minkowski}.

    Denote rays of the Mori cone of $\overline{W}$ corresponding to projections on~$\PP^2$ by $R_1$ and~$R_2$. The group $\Gal(\overline{\QQ}/K)$ acts on these rays, denote by $\relpenalty = 10000 H \subset \Gal(\overline{\QQ}/K)$ the stabilizer of a ray~$R_1$. Then $H$ has index at most $2$ and the field of invariants $L = (\overline{\QQ})^{H}$ has degree at most~$2\cdot3 = 6$ over~$\QQ$. Also, extremal contraction $$f : W_L \rightarrow W'$$ corresponding to $R_1$ is defined over $L$. Then $X'$ is a Severi--Brauer surface with a $L$-point. By Theorem~\ref{theo: Chatelet's theorem}, we obtain that $X'$ is isomorphic to $\mathbb{P}^2_{L}$. Moreover, group $G$ also acts on the Mori cone, and if we denote by $G' \subset G$ the stabilizer of $R$, then~$[G:G'] \leqslant 2$ and the contraction $f$ is $G'$-equivariant. 
    
    Since $\overline{W}$ has bidegree $(1,1)$, general fiber of $f$ is isomorphic to $\PP^1$. Therefore, by Lemma~\ref{lemma: smooth rational fiber in fibrations} there is a short exact sequence $$1 \rightarrow G'_F \rightarrow G' \rightarrow G'_{\PP^2} \rightarrow 1,$$ where $G'_F \subset \PGL_2(L)$ and $G'_{\PP^2}$ acts on $\PP^2_L$ preserving $f(C)$. Immediately note that~$\relpenalty=10000 |G'_F| \leqslant 84$ by Lemma~\ref{lemma: PGL_2(L), L <= 4,6,24}. Since curve $C$ has degree $(2,2)$, then $f(C) \subset \PP^2_L$ is either a conic or a line. If $f(C)$ is a smooth conic, then $G'_{\PP^2} \subset \Aut(f(C))$ and~$\relpenalty=10000 |G'_{\PP^2}| \leqslant 180$ by Lemma~\ref{lemma: conic over L <=6}. Therefore, $$|G| = [G:G']\cdot|G'_F|\cdot|G'_{\PP^2}| \leqslant 2 \cdot 84\cdot 180 = 30\,240 < 10^5.$$ If $f(C)$ is a singular conic, then $f(C)$ contains only one singular point $\tilde{P}$. Therefore,~$\tilde{P}$ is an $L$-point, and it is fixed by $G'_{\PP^2}$. By Theorem~\ref{theo: action with fixed point} we have $G'_{\PP^2} \subset \GL_2(L)$ and~$\relpenalty=10000 |G'_{\PP^2}| \leqslant 1512$ by Lemma~\ref{lemma: GL_2(L), L <= 6} and $$|G| = [G:G']\cdot|G_F|\cdot|G'_{\PP^2}| \leqslant 2\cdot84 \cdot 1512 = 254\,016 < 10^6.$$ If $f(C)$ is a line, then there is a short exact sequence $$1 \rightarrow G'' \rightarrow G'_{\PP^2} \rightarrow G'_{\PP^1} \rightarrow 1,$$ where $G'_{\PP^1} \subset \PGL_2(L)$ and $G''$ acts trivially on $f(C)$. By Lemma~\ref{lemma: PGL_2(L), L <= 4,6,24} we have~$\relpenalty=10000 |G'_{\PP^1}| \leqslant 84$. Also, the group $G''$ fixes the whole line of $L$-points, in particular,~$G''$ embeds to $\GL_2(L)$ by Theorem~\ref{theo: action with fixed point} and $|G''| \leqslant 1512$ by Lemma~\ref{lemma: GL_2(L), L <= 6}. Therefore,~$|G'_{\PP^2}| = |G''|\cdot|G'_{\PP^1}| \leqslant 127\,008$ and $$|G| = [G:G']\cdot|G'_F| \cdot |G'_{\PP^2}| \leqslant 2 \cdot 84 \cdot 127\,008 = 21\,337\,344 < 10^8.$$
    \end{proof}

\begin{lemma}\label{lemma: (1.2.6)}
    Let $X$ be a Gorenstein threefold $\QQ$ with a $\QQ$-point such that $\overline{X}$ is isomorphic to an intersection of three divisors of tridegrees $(0,1,1)$, $(1,0,1)$, $(1,1,0)$ in $\PP^2 \times \PP^2 \times \PP^2$. Let~$\relpenalty = 10000 G \subset \Aut(X)$ be a finite subgroup. Then $$|G| \leqslant 64\,012\,032 < 10^8.$$
\end{lemma}

\begin{proof}
    Threefold $\overline{X}$ has three divisorial contractions, which are restrictions of three projections of $\PP^2 \times \PP^2 \times \PP^2$ to~$\binoppenalty=10000 \PP^2 \times \PP^2$. Denote rays of the Mori cone of $\overline{X}$, corresponding to these contractions by $R_1$, $R_2$ and~$R_3$. The group $\Gal(\overline{\QQ}/\QQ)$ acts on these rays, denote by $\relpenalty = 10000 H \subset \Gal(\overline{\QQ}/\QQ)$ the stabilizer of a ray~$R_1$. Then $H$ has index at most $3$ and the field of invariants $K = (\overline{\QQ})^{H}$ has degree at most~$3$ over~$\QQ$. Also, extremal contraction $$f : X_K \rightarrow W$$ corresponding to $R_1$ is defined over $K$. Here $\overline{W} \subset \PP^2 \times \PP^2$ is a divisor of bidegree~$(1,1)$. Moreover, if we denote the image of the exceptional divisor of $f$ by $C$, then $C$ has bidegree~$(2,2)$. The group $G$ also acts on the Mori cone, and if we denote by~$G' \subset G$ the stabilizer of $R_1$, then~$[G:G'] \leqslant 3$, contraction $f$ is $G'$-equivariant and $G' \subset \Aut(W)$ is a finite subgroup preserving $C$. Then $|G'| \leqslant 21\,337\,344$ by Lemma~\ref{lemma: (1,1) in P^2 x P^2 fixing curve} and $$|G| = [G:G']\cdot|G'| \leqslant 3 \cdot 21\,337\,344 = 64\,012\,032 < 10^8.$$
\end{proof}

\begin{lemma}\label{lemma: (1.2.3) smooth}
    Let $X$ be a smooth Fano threefold over $\QQ$ such that $\overline{X}$ is isomorphic to a blow up of a quadric $Q \subset \PP^4_{\overline{\QQ}}$ in a twisted quartic $C$. Let~$\relpenalty=10000 G \subset \Aut(X)$ be a finite subgroup. Then $|G| \leqslant 120.$
\end{lemma}

\begin{proof}
    The Mori cone of $\overline{X}$ contains two extremal rays, one of which corresponds to a blow up, which is mentioned in the lemma, denote this ray by $R$. The group $\Gal(\overline{\QQ}/\QQ)$ acts on these extremal rays, denote by $\relpenalty = 10000 H \subset \Gal(\overline{\QQ}/\QQ)$ the kernel of this action. Then~$H$ has index at most $2$, field of invariants $K = (\overline{\QQ})^{H}$ has degree at most $2$ over~$\QQ$ and~$\rho(X_K) = 2$. Therefore, each extremal contraction is defined over $K$. Consider one of them $$f : X_K \rightarrow X' \subset \PP^4_K,$$ which corresponds to a ray $R$. Group $G$ also acts on the Mori cone, and if we denote by~$\relpenalty = 10000 G' \subset G$ the stabilizer of $R$, then~$[G:G'] \leqslant 2$, morphism $f$ is $G'$-equivariant and~$ \relpenalty=10000 G' \subset \Aut(X')$. 

    Since $f$ is $G'$-equivariant, then group $G'$ fixes a curve~$C'$, which is the image of an exceptional divisor. Moreover, $G' \subset \Aut(C')$. Indeed, consider a sequence of embeddings $$C' \hookrightarrow X' \hookrightarrow \PP^4_K.$$ Passing to algebraic closure, we obtain a sequence of embeddings $$C \hookrightarrow Q \hookrightarrow \PP^4_{\overline{\QQ}}.$$ Suppose there is an element $g \in G'$, which acts on $C'$ trivially. Then identifying $g$ with its image in $\Aut(Q)$, we obtain an automorphism of a quadric, which acts trivially on a twisted quartic $C$. But then $g$ extends to an automorphism $\tilde{g}$ of $\PP^4_{\overline{\QQ}}$, which fixes a twisted quartic. Therefore, $\tilde{g}$ is the identity automorphism and $g$ is the identity automorphism. Now, considering a quadratic extension $K \subset K'$ such that $C'_{K'} \simeq \PP^1_{K'}$, we may assume that~$\relpenalty=10000 G' \subset \PGL_2(K')$ and then have $|G'| \leqslant 60$ by Lemma~\ref{lemma: PGL_2(L), L <= 4,6,24}. Finally, $$|G| = [G:G']\cdot|G'| \leqslant 120.$$  
\end{proof}

\begin{lemma}\label{lemma: (1.2.3) singular}
    Let $X$ be a threefold over $\QQ$ such that $\overline{X}$ is isomorphic to a blow up of a quadratic cone in $\PP^4$ with center a union of two conics that do not pass through the vertex and meet each other transversely. Let $G \subset \Aut(X)$ be a finite subgroup. Then $$|G| \leqslant 174\,182\,400 < 10^9.$$
\end{lemma}
\begin{proof}
    There are two singular points $P_1$ and $P_2$ on $\overline{X}$ with $\dim T_{P_1}\overline{X}= \dim T_{P_2}\overline{X} = 4$. One is the preimage of the vertex of the blowing cone, and another lies on the exceptional divisor over the intersection point of conics. The group $\Gal(\overline{\QQ}/\QQ)$ acts on these points, denote by $\relpenalty = 10000 H \subset \Gal(\overline{\QQ}/\QQ)$ the kernel of this action. Then $H$ has index at most $2$, field of invariants $K = (\overline{\QQ})^{H}$ has degree at most $2$ over~$\QQ$ and both singular points are defined over $K$. 
    
    The group $G$ also acts on these points, and if we define by $G'$ the kernel of this action, then $G'$ acts on $X_K$ fixing $P_1$. By Theorem~\ref{theo: action with fixed point}, group $G'$ embeds to $\GL_4(K)$, and~$\relpenalty=10000 |G| \leqslant 87\,091\,200$ by Corollary~\ref{corollary: additive bound GL_4(<=2)}. Finally, we have $$|G| = [G:G'] \cdot |G'| \leqslant 2 \cdot 87\,091\,200= 174\,182\,400 < 10^9.$$   
    %The Mori cone of $\overline{X}$ contains two extremal rays, one of which corresponds to a blow up, which is mentioned in the lemma, denote this ray by $R$. The group $\Gal(\overline{\QQ}/\QQ)$ acts on these extremal rays, denote by $\relpenalty = 10000 H \subset \Gal(\overline{\QQ}/\QQ)$ the kernel of this action. Then $H$ has index at most $2$, field of invariants $K = (\overline{\QQ})^{H}$ has degree at most $2$ over~$\QQ$ and $\rho(X_K) = 2$. Therefore, each extremal contraction is defined over $K$. Consider one of them $$f : X_K \rightarrow X' \subset \PP^4_K,$$ which corresponds to a ray $R$. Group $G$ also acts on the Mori cone, and if we denote by~$\relpenalty = 10000 G' \subset G$ the stabilizer of $R$, then~$[G:G'] \leqslant 2$, morphism $f$ is $G'$-equivariant and~$ \relpenalty=10000 G' \subset \Aut(X')$
    %Since $\overline{X'}$ contains only one singular point $P$, then $P$ is defined over $K$. Also $P$ is fixed by $G$, and $\dim T_PX = 4$. 
    %Therefore,~$P$ is defined over $\QQ$. 
\end{proof}

\begin{lemma}\label{lemma: (1.2.8)}
    Let $X$ be a terminal Gorenstein Fano threefold over $\QQ$ with a $\QQ$-point such that $\overline{X}$ is isomorphic to a divisor in~$\binoppenalty=10000 \PP^1\times \PP^1 \times \PP^1 \times \PP^1$ of multidegree $(1,1,1,1)$. Let~$G \subset \Aut(X)$ is a finite subgroup. Then $$|G| \leqslant 1\,778\,112\,000 < 10^{10}.$$
\end{lemma}

\begin{proof}
   Threefold $\overline{X}$ has four divisorial contractions, which are restrictions of four projections of $\PP^1 \times \PP^1 \times \PP^1 \times \PP^1$ to $\PP^1 \times \PP^1 \times \PP^1$. Denote rays of the Mori cone, corresponding to these contractions by $R_1$, $R_2$, $R_3$ and $R_4$. The group $\Gal(\overline{\QQ}/\QQ)$ acts on these rays, denote by~$\relpenalty = 10000 H \subset \Gal(\overline{\QQ}/\QQ)$ the stabilizer of the ray~$R_1$. Then $H$ has index at most~$4$ and the field of invariants $K = (\overline{\QQ})^{H}$ has degree at most~$4$ over~$\QQ$. Also, extremal contraction $$f : X_K \rightarrow X'$$ corresponding to $R_1$ is defined over $K$. Here $X'$ is as in Lemma~\ref{lemma: (1.2.7)}. Group $G$ also acts on the Mori cone, and if we denote by~$\relpenalty = 10000 G' \subset G$ the stabilizer of $R_1$, then~$[G:G'] \leqslant 4$, contraction~$f$ is $G'$-equivariant and $G' \subset \Aut(X')$. Then $|G'| \leqslant 444\,528\,000$ by Lemma~\ref{lemma: (1.2.7)} and $$|G| =[G:G']\cdot[G'] \leqslant 4 \cdot 444\,528\,000 =  1\,778\,112\,000 < 10^{10}.$$
\end{proof}

Now we are able to prove the general case of threefolds with $\rho(\overline{X}) > 1$.
%\begin{lemma}\label{lemma: 4-1}
 %    \textcolor{red}{Лемма для 4-1}
%\end{lemma}

\begin{propos}\label{prop: Gorenstein rho > 1}
    Let $G$ be a finite group. Let $X$ be a terminal $G\QQ$-factorial Gorenstein Fano threefold over $\QQ$ with a $\QQ$-point such that $\rho(\overline{X}) > 1$ and $\rho(X)^G = 1$. Then $$|G| \leqslant 5\,573\,836\,800 < 10^{10}.$$ 
\end{propos}

\begin{proof}
    Since $X$ is $G\QQ$-factorial and $\rho(X)^G = 1$, then $\Cl(X)^G \simeq \ZZ$. Consider the action of $\Gal(\overline{\QQ}/\QQ)$ on $\Cl(\overline{X})$ and denote by~$H$ the kernel of this action. Then $H$ is a normal open (and therefore closed) subgroup of finite index. Put~$K = (\overline{\QQ})^H$, then $\QQ \subset K$ is a finite Galois extension (see \cite[\href{https://stacks.math.columbia.edu/tag/0BML}{Theorem 0BML}]{stacks-project}) and the action of~$\Gal(\overline{\QQ}/\QQ)$ on $\overline{X}$ factors through the faithful action of a finite group~$ \relpenalty = 10000 G' = \Gal(K/\QQ)$. Since~$\Cl(X)^{G} \simeq \ZZ$, then we have $\Cl(\overline{X})^{G\times G'} \simeq \ZZ$.
    Therefore, by Theorem~\ref{theo: prokhorov on G-Fanos} we know that $\overline{X}$ belongs to one of the $8$ families from Table~\ref{table: possible minimal by Pro}. 

    If $\overline{X}$ belongs to family $(1.2.1)$ or $(1.2.5)$, then $g(X) = 7$ and we have $$|G| \leqslant 5\,573\,836\,800 < 10^{10}$$ by Lemma~\ref{lemma: (1.2.1) and (1.2.5)}. If $\overline{X}$ belongs to other $6$ families, then we have $$|G| \leqslant 1\,886\,976\,000 < 10^{10}$$ by Lemma \ref{lemma: (1.2.2)} in the case of family $(1.2.2)$, by Lemma~\ref{lemma: (1.2.3) singular} in the singular case of family~$(1.2.3)$, by Lemma~\ref{lemma: (1.2.3) smooth} in the smooth case of family $(1.2.3)$, by Lemma~\ref{lemma: (1.2.4)} in the case of family $(1.2.4)$, by Lemma~\ref{lemma: (1.2.6)} in the case of family $(1.2.6)$, by Lemma~\ref{lemma: (1.2.7)}  in the case of family $(1.2.7)$, and by Lemma~\ref{lemma: (1.2.8)} in the case of family $(1.2.8)$. 
\end{proof}

Now, let us deal with threefolds, whose Picard rank is equal to $1$ over $\overline{\QQ}$. The following three lemmas handle cases of threefolds of large degree or genus.
%Let us start with some results on smooth Fano threefolds. Of course, smooth Fano threefolds are much better understood than singular ones. In particular, we have a lot more options for bounding the orders of finite groups acting on smooth Fano threefolds rather than on singular ones. But in this section, we are not trying to get the best bounds, only those that will be sufficient for our purposes. Also, to avoid repetition, we do not consider all possible smooth Fano threefolds, but only those that require different reasoning than their singular Gorenstein deformations. In this regard, we consider only few smooth threefolds of Picard rank $1$ over $\overline{\QQ}$.

\begin{lemma}\label{lemma: smooth of degree 5}
Let $G$ be a finite group. Let $Y$ be a terminal $G\QQ$-factorial Gorenstein Fano threefold of index $2$ of degree $5$ over $\QQ$. Then $|G| \leqslant 24$.
\end{lemma}

\begin{proof}
    By Theorem~\cite[Theorem 1.7]{Pro I} threefold $\overline{Y}$ is smooth, then the Hilbert scheme of lines on $\overline{Y}$ is isomorphic to $\PP^2$ by~\mbox{\cite[Theorem 1.1.1]{Hilbert schemes}}. Therefore, the Hilbert scheme of lines on $Y$ is a Severi--Brauer surface $S_1$ over $\QQ$ and~$G$ is a subgroup in $\Aut(S_1)$ by~\mbox{\cite[Lemma 4.2.1]{Hilbert schemes}}. If $S_1$ is non-trivial, then $|G| \leqslant 3$ by~\mbox{\cite[Theorem 1.3]{Vikulova 2}}, if $S_1$ is trivial, then $|G| \leqslant 24$ by Lemma~\ref{lemma: P^2_Q}.
\end{proof}

\begin{lemma}\label{lemma: smooth of genus 12}
    Let $G$ be a finite group. Let $X$ be a terminal $G\QQ$-factorial Gorenstein Fano threefold of index~$1$ of genus~$12$ over $\QQ$. Then $|G| \leqslant 5760$. 
\end{lemma}

\begin{proof}
     By~\cite[Theorem 1.3]{Pro g = 12} threefold $\overline{X}$ is either smooth or contains only one singular point $P$, which is an ordinary double point. In the former case, the Hilbert scheme of conics on $\overline{X}$ is isomorphic to $\PP^2$ by~\cite[Theorem 1.1.1]{Hilbert schemes}. Therefore, the Hilbert scheme of conics on $X$ is a Severi--Brauer surface $S_2$ over $\QQ$ 
    and~$G$ is a subgroup in $\Aut(S_2)$ by~\mbox{~\cite[Lemma 4.3.4]{Hilbert schemes}}. If $S_2$ is non-trivial, then $|G| \leqslant 3$ by~\mbox{\cite[Theorem 1.3]{Vikulova 2}}, if $S_2$ is trivial, then $|G| \leqslant 24$ by Lemma~\ref{lemma: P^2_Q}. In the latter case, point $P$ is defined over $\QQ$ and the group $G$ acts on $X$ fixing $P$. Therefore, $G$ embeds to $\GL_4(\QQ)$ by Theorem~\ref{theo: action with fixed point} and~$\relpenalty=10000 |G| \leqslant 5760$ by Theorem~\ref{theo: Minkowski}.
\end{proof}

\begin{rem}\label{rem: schemes are P^2 always}
    Actually, Severi--Barauer surfaces in Lemmas~\ref{lemma: smooth of degree 5} and~\ref{lemma: smooth of genus 12} are  trivial. To see this, note, that there is a curve $\Gamma$ of degree $4$ on $S_2$ defined over $\QQ$ that corresponds to degenerate conics on $X$. Therefore, the positive generating class of $\Pic(\overline{X}) \simeq \ZZ$ is defined over $\QQ$ and provides an isomorphism $S_2 \simeq \PP^2_{\QQ}$. Triviality of $S_1$ follows from the existence of an isomorphism $S_1 \simeq S_2$, see~\cite[Proposition B.4.1]{Hilbert schemes}.
\end{rem}

\begin{lemma}\label{lemma: degree 4}
     Let $G$ be a finite group. Let $X$ be a terminal $G\QQ$-factorial Gorenstein Fano threefold of index~$2$ of degree~$4$ over $\QQ$. Then $$|G| \leqslant 1\,902\,071\,808 < 10^{10}.$$ 
\end{lemma}

\begin{proof}
    Threefold $\overline{X}$ contains at most $6$ singular points, see~\cite[\S2, Proposition 3]{Avilov}. By Lemma~\ref{lemma: passing to stabilizer and acting on a tangent space} there is a field extension $\QQ \subset K$ of degree at most $6$ and a subgroup $G_{\bullet} \subset G$ of index at most $6$ such that $G_{\bullet} \subset \GL_4(K)$. By Corollary~\ref{corollary: additive bound GL_4(<=7)} we have $$|G_{\bullet}| \leqslant 317\,011\,968 < 10^{9}$$ and $$|G| = [G:G_{\bullet}] \cdot |G_{\bullet}| \leqslant 6 \cdot 317\,011\,968 = 1\,902\,071\,808 < 10^{10}.$$
\end{proof}

To deal with remaining threefolds of Picard rank one over $\overline{\QQ}$, we mainly use embeddings given by a fundamental linear system, but this is not always the case.
The following proposition give an information on finite subgroups of automorphism groups of Gorenstein Fano threefolds of Picard rank $1$, whose fundamental class is not very ample.

\begin{propos}[{see~\cite[Theorem 0.6]{Shin}, \cite[Corollary 0.8]{Shin}, \cite[Theorem 1.1]{JR-2}, \cite[Theorem 1.4]{PCS}}]\label{prop: Gorenstein Fanos with not very ample anticanonical class}
    Let $X$ be a terminal Gorenstein Fano threefold over algebraically closed field with $\relpenalty = 10000 \rho(X) = 1$. Suppose that fundamental class~$H$ is not very ample, then one of the following possibilities holds:
\begin{enumerate}
    \item[(i)] $ \iota(X) = 2 $ and $ d(X) = 1 $;
    \item[(ii)] $ \iota(X) = 2 $ and $d(X) = 2 $;
    \item[(iii)] $ \iota(X) = 1 $ and $ g(X) = 2 $;
    \item[(iv)] $ \iota(X) = 1 $, $ g(X) = 3 $, and $ X $ is a double cover of a three-dimensional quadric.
\end{enumerate}
\end{propos}

\begin{comment}
\begin{cor}\label{corollary: embedding if d >=3 or g >=4}
    Let $X$ be a terminal Gorenstein Fano threefold over an algebraically closed field with $\rho(X) = 1$. If $\iota(X) = 2$ and $d(X) \geqslant 3$ or $\iota(X) = 1$ and $g(X) \geqslant 4$, then fundamental class $H$ is very ample.
\end{cor}
\end{comment}

\begin{propos}[cf. {\cite[Proposition 6.1.1]{PrShr Cr3}}]\label{prop: low degree and embedding}
Let $K$ be a field of characteristic $0$. Let $X$ be a Gorenstein Fano threefold over $K$ with a $K$-point. Suppose that one of the following cases holds 
\begin{enumerate}
    \item[(i)] $ \iota(X) = 2 $ and $ d(X) = 1 $;
    \item[(ii)] $ \iota(X) = 2 $ and $d(X) = 2 $;
    \item[(iii)] $ \iota(X) = 1 $ and $ g(X) = 2 $;
    \item[(iv)] $ \iota(X) = 1 $, $ g(X) = 3 $, and $X$ is a double cover of a three-dimensional quadric.
\end{enumerate}
Let $G \subset \Aut(X)$ be a finite group. Then for some positive integer $ r $ there is a central extension
$$ 1 \to \mu_r \to \tilde{G} \to G \to 1 $$
such that one has an embedding \( \tilde{G} \subset \mathrm{GL}_3(K) \times K^{\times} \) in case \textnormal{(i)}, an embedding \( \tilde{G} \subset \mathrm{GL}_4(K) \) in cases \textnormal{(ii)} and \textnormal{(iii)}, and an embedding \( \tilde{G} \subset \mathrm{GL}_5(K) \) in case \textnormal{(iv)}.
\end{propos}

\begin{rem}\label{rem: on low degree and genus}
    In paper~\cite{PrShr Cr3} Proposition~\ref{prop: low degree and embedding} is stated for algebraically closed fields. However, for the proof written there to work, it is sufficient to require that the class $H$ is defined over $K$, which is the case in our setting, since $X$ contains a $K$-point.
\end{rem}

The following lemma describes terminal Gorenstein Fano threefolds of large index. 

\begin{lemma}[cf. {\cite[Theorem 3.1.14]{Isk Pro}}]\label{lemma: terminal Gorenstein Fanos of large index}
    Let $X$ be a terminal Gorenstein Fano threefold over an algebraically closed field. 
    \begin{enumerate}
        \item If $\iota(X) = 4$, then $X \simeq \PP^3$;
        \item if $\iota(X) = 3$, then $X \simeq Q \subset \PP^4$, where $Q$ is either a smooth quadric or a cone over smooth $2$-dimensional quadric.
    \end{enumerate}
\end{lemma}

This allows us to complete the case of threefolds with $\rho(\overline{X}) = 1$.

\begin{propos}\label{propos: Gorenstein, rho = 1}
     Let $G$ be a finite group. Let $X$ be a terminal $G\QQ$-factorial Gorenstein Fano threefold over $\QQ$ with a $\QQ$-point such that $\rho(\overline{X}) = 1$. Then $$|G| \leqslant 24\,103\,053\,950\,976\,000 < 10^{17}.$$ Moreover, if we are not in the case $\iota(X) = 1$, $g(X) = 10$, then $$|G| \leqslant 735\,746\,457\,600 < 10^{12}.$$
\end{propos}

\begin{proof}
First of all, note, that indices, degrees and genera of terminal Gorenstein Fano threefolds coincides with indices, degrees and genera of smooth ones by Theorem~\ref{theo: smoothing}.

Suppose that $\iota(X) = 4$. Then $\overline{X}$ is isomorphic to $\PP^3$ by Lemma~\ref{lemma: terminal Gorenstein Fanos of large index}. Therefore,~$X$ is a Severi--Brauer threefold with a $\QQ$-point, which is isomorphic to $\PP^3_{\QQ}$ by Theorem~\ref{theo: Chatelet's theorem} and~$\relpenalty = 10000 G \subset \PGL_4(\QQ)$. By Corollary~\ref{corollary: PGL_5(Q) and PGL_6(Q)} we have $|G| \leqslant 362\,880.$

Suppose that $\iota(X) = 3$. Then $\overline{X}$ is isomorphic either to a smooth $3$-dimensional quadric or to a cone over a smooth $2$-dimensional quadric by Lemma~\ref{lemma: terminal Gorenstein Fanos of large index}. In the former case, since~$X$ contains a $\QQ$-point, then $X$ is isomorphic to a smooth rational $3$-dimensional quadric, in particular $G \subset \PGL_5(\QQ)$. By Corollary~\ref{corollary: PGL_5(Q) and PGL_6(Q)} we have $$|G| \leqslant 87\,178\,291\,200 < 10^{11}.$$ In the latter case, the unique singular point of $\overline{X}$ is defined over $\QQ$ and $G$ acts on $X$ fixing this point. Applying Theorem~\ref{theo: action with fixed point}, we obtain an embedding $G \subset \GL_4(\QQ)$. By Theorem~\ref{theo: Minkowski} we have $|G| \leqslant 5\,760$.

Suppose that $\iota(X) = 2$. If $d(X) = 5$, then we have $|G| \leqslant 24$ by Lemma~\ref{lemma: smooth of degree 5}. If~$\relpenalty=10000 d(X) = 4$, then we have $$|G| \leqslant 1\,902\,071\,808 < 10^{10}$$ by Lemma~\ref{lemma: degree 4}. If $d(X) = 3$, then according to Proposition~\ref{prop: Gorenstein Fanos with not very ample anticanonical class} the fundamental class is very ample, so that the map $\Phi_{|-\frac{1}{2}K_X|}$ is a~\hbox{$G$-equivariant} embedding to~$\mathbb{P}^{4}$ and $$|G| \relpenalty=10000\leqslant 87\,178\,291\,200 < 10^{11} $$ by Corollary~\ref{corollary: PGL_5(Q) and PGL_6(Q)}.

If $d(X) =1$ or $2$, then by Proposition~\ref{prop: low degree and embedding} there exists a central extension $$1 \rightarrow \mu_r \rightarrow \tilde{G} \rightarrow G \rightarrow 1,$$ where $\tilde{G} \subset \GL_3(\QQ) \times \QQ^{\times}$ in case $d(X) = 1$ and $\tilde{G} \subset \GL_4(\QQ)$ in case $d(X) = 2$. Applying Theorem~\ref{theo: Minkowski} we obtain a bound $|G| \leqslant |\tilde{G}| \leqslant 96$ in the former case and a bound~$\relpenalty = 10000 |G| \leqslant |\tilde{G}| \leqslant 5\,760$ in the latter case. 
%If $d(X) = 5$, we can repeat the previous argument, but we need a smaller bound. If $X$ is smooth, then $|G| \leqslant ?$ by Lemma~\ref{lemma: V_5}. If $X$ is singular with $N$ singular points over $\overline{\QQ}$ or, equivalently, over $\CC$, then by Theorem~\ref{theo: smoothing} we have  $$ N \leqslant 20 - \rho(\mathfrak{X}_b) + h^{1,2}(\mathfrak{X}_b),$$ where $\mathfrak{X}_b$ is a smoothing of $X_{\CC}$. Also, by Theorem~\ref{theo: smoothing} we have $\rho(\mathfrak{X}_b) = 1$, $\iota(\mathfrak{X}_b) = 2$ and~$d(\mathfrak{X}_b) = 5$. From the classification of smooth Fano threefolds it turns out, that~$\mathfrak{X}_b$ is a section of the Grassmanian $\Gr(2,5) \subset \PP^9$ by a subspace of codimension $3$. In particular,~$\relpenalty = 10000 h^{1,2}(\mathfrak{X}_b) = 0$ and $N \leqslant 19$. Since all singular points of~$\overline{X}$ are hypersurface singularities, then by Lemma~\ref{lemma: passing to stabilizer and acting on a tangent space} there exists a field extension $\QQ \subset K$ of degree $d \leqslant 19$ and a subgroup~$G_{\bullet} \subset G$ of index $i \leqslant 19$ such that $G_{\bullet}$ embeds to $\GL_4(K)$. Therefore, $$|G| \leqslant  2^{15}\cdot3^{9}\cdot5^2 \cdot 7^2 \cdot 13 \cdot 19^2 = 59326380520243200 = (17)$$ by \textcolor{red}{Corollary}~\ref{corollary: GL_4(<=27)}. 

Suppose that $\iota(X) = 1$. If $-K_X$ is very ample, then we have a faithful~$(g+2)$-dimensional representation of $G$ in~$H^0(X,-K_X)^{\vee}$ by Lemma~\ref{lemma: action on sections}. If $g(X) \leqslant 9$, then $$|G| \leqslant 2^{19}\cdot3^6\cdot5^2\cdot7\cdot11 = 735\,746\,457\,600 < 10^{12}$$ by Theorem~\ref{theo: Minkowski}. If $g(X) = 10$, then $$|G| \leqslant 2^{22}\cdot3^8\cdot5^3\cdot7^2\cdot11\cdot13 = 24\,103\,053\,950\,976\,000 < 10^{17}$$ by Theorem~\ref{theo: Minkowski}. If $g(X) = 12$, then $|G| \leqslant 5760$ by Lemma~\ref{lemma: smooth of genus 12}. If $-K_X$ is not very ample, then according to Proposition~\ref{prop: Gorenstein Fanos with not very ample anticanonical class} we are in one of two cases, either $g(X) = 2$, or~$\relpenalty=10000 g(X) = 3$ and $X$ is a double covering of a three-dimensional quadric. In the former case by Proposition~\ref{prop: low degree and embedding} there exists a central extension $$1 \rightarrow \mu_r \rightarrow \tilde{G} \rightarrow G \rightarrow 1,$$ where $\tilde{G} \subset \GL_4(\QQ)$. Then $|G| \leqslant |\tilde{G}| \leqslant 5\,760$ by Theorem~\ref{theo: Minkowski}. In the latter case by Proposition~\ref{prop: low degree and embedding} there exists a central extension $$1 \rightarrow \mu_r \rightarrow \tilde{G} \rightarrow G \rightarrow 1,$$ where $\tilde{G} \subset \GL_5(\QQ)$. Then $|G| \leqslant |\tilde{G}| \leqslant 11\,520$ by Theorem~\ref{theo: Minkowski}. \end{proof}

\begin{rem}\label{remark: alternative strategy for Gorenstein}
    Note that the argument which we use in Lemma~\ref{lemma: degree 4} could be used for all Gorenstein Fano threefolds together with the bound on the number of singularities given by Theorem~\ref{theo: smoothing}. But it turns out that the bounds obtained in this way are worse than the ones presented.
\end{rem}

\begin{rem}\label{remark: problems with i = 1, g = 10}
    The bound in case of threefolds of index $1$ and genus $10$ given in Proposition~\ref{propos: Gorenstein, rho = 1} seems to be too large. We think there should be a way to produce a smaller bound. Actually, we have a smaller bound if $\overline{X}$ is $\QQ$-factorial, namely, in this case $\overline{X}$ has at most $2$ singular points according to~\cite[Theorem 1.3]{Pro III} and we apply our standard argument with action on a tangents space to a singular point. But we do not know if~$\QQ$-factoriality of $\overline{X}$ is always the case or how to handle the non-$\QQ$-factorial one.
\end{rem}

Now everything is prepared to prove the main theorem.

\begin{proof}[of Theorem~\ref{theo: cr_3}]
We can assume $X$ to be projective. Regularizing the action of~$G$ and taking $G$-equivariant resolution of singularities, we can assume that $G$ is a subgroup of an automorphism group of a smooth rational treefold~$X$ (see~\mbox{\cite[Lemma 3.5]{D-I}}). %In this case, $X$ contains a $\mathbb{Q}$-point according to the Lang--Nishimura theorem (see for example~\mbox{\cite[Theorem 3.6.11]{Poo}}).
 Applying equivariant minimal model program, we obtain a terminal $G\QQ$-factorial threefold~$\tilde{X}$ with structure of a~\hbox{$G$-Mori} fiber space, i.e. there are three possible cases:
\begin{enumerate}
    \item $\tilde{X}$ has a structure of a $G\QQ$-del Pezzo fibration;
    \item $\tilde{X}$ has a structure of a $G\QQ$-conic bundle;
    \item $\tilde{X}$ is a $G\QQ$-Fano threefold.
\end{enumerate}
The first and the second cases are handled by Lemma~\ref{lemma: GQ-conic and GQ-del Pezzo}, and we have $|G| \leqslant 5148$. The third case follows from Proposition~\ref{prop: bound for non-Gorenstein}, if $\tilde{X}$ is non-Gorenstein, and from Propositions~\ref{prop: Gorenstein rho > 1} and~\ref{propos: Gorenstein, rho = 1}, if~$\tilde{X}$ is Gorenstein.
\end{proof}

Taking into account Remark~\ref{remark: problems with i = 1, g = 10} let us prove the following auxiliary theorem, which can be perceived as the first step towards a possible improvement of Theorem~\ref{theo: cr_3}.

\begin{theorem}\label{theo: Cr_3 for future}
    Let $X$ be a $3$-dimensional variety over $\QQ$ with a $\QQ$-point such that~$\relpenalty=10000 \overline{X}$ is rationally connected. Let $G \subset \Bir(X)$ be a finite subgroup. Assume that for any finite group $\tilde{G}$ such that there exists terminal $\tilde{G}\QQ$-factorial Gorenstein Fano threefold $X$ over~$\QQ$ with a $\QQ$-point with $\rho(\overline{X}) = 1$, $\iota(X) = 1$ and $g(X) = 10$, one has $|\tilde{G}| \leqslant M$. Then $$|G| \leqslant  \max\{735\,746\,457\,600, M\}.$$ In particular, this applies to finite subgroups of $\Cr_3(\QQ)$.
\end{theorem}

\begin{proof}
    The proof follows verbatim the proof of Theorem~\ref{theo: cr_3}, with the only difference that in the case where $\tilde{X}$ is Gorenstein, we apply Proposition~\ref{propos: Gorenstein, rho = 1} along with the restriction given in the condition of the theorem.
\end{proof}

\begin{comment}
\begin{rem}
    Следующее улучшение - $735\,746\,457\,600 < 10^{12}$. Для этого надо эффективно оценить порядки конечных групп действующих на Горенштейновом Фано индеса $2$ степени $4$ (сейчас это делается эквивариантным вложением фундаментальной системой и Серром для $\PGL_6(\QQ)$). 

    Следующее - $367\,873\,228\,800 < 10^{12}$. Для этого надо эффективно оценить порядки конечных групп действующих на Горенштейновом Фано индеса $1$ рода $9$ (сейчас это делается эквивариантным вложением антиканклассом и Минковским для $\GL_{11}(\QQ)$). 

    Следующее - $87\,178\,291\,200 < 10^{11}$. Для этого надо эффективно оценить порядки конечных групп действующих на Горенштейновом Фано индеса $1$ рода $8$ (сейчас это делается эквивариантным вложением антиканклассом и Минковским для $\GL_{10}(\QQ)$). 

    Следующее - $10^{10}$. Для этого надо эффективно оценить порядки конечных групп действующих на Горенштейновом Фано индеса $2$ степени $3$ (сейчас это делается эквивариантным вложением фундаментальной системой и Серром для $\PGL_5(\QQ)$). И оценить порядки конечных групп действующих на гладкой трехмерной квадрике (сейчас это тоже приходит из Серра для $\PGL_5(\QQ)$).
\end{rem}
\end{comment}

\appendix
\section{Galois theory}\label{section: appendix}\label{sect: appendix}

In this section we do a technical preliminary work for Corollaries~\ref{corollary: additive bound GL_4(<=7)},~\ref{corollary: additive bound GL_3(<=15)},~\ref{corollary: serr for PGL_3(<=2)} and~\ref{corollary: serr for PGL_4(<=2)}. In particular, we bound numbers, which appears in Theorem~\ref{theorem: schur} and Corollary~\ref{corollary: Serre for PGL} in some cases. To formulate the main goal of this section, let us recall and fix the notation. 

Suppose we are given a number field $K$ of degree $d$ over $\QQ$. First of all, with each prime number $p$ we associate cyclotomic invariants $t_{p}$ and~$m_{p}$ defined in Definition~\ref{def: cyclotomic invariants}. For convenience, we repeat this definition here:
 \begin{itemize}
        \item $t_p= t_{p}(K) = [K(\xi_p): K]$;
        \item $t_2= t_{2}(K) = [K(\xi_4): K]$;
        \item $m_p = m_{p}(K) = \sup\{n \mid \xi_{p^n} \in K(\xi_p)\}$;
        \item $m_2 = m_{2}(K) = \sup \{n \mid \xi_{2^n} \in K\}$, if $\xi_4 \in K$;
        \item $m_2 = m_{2}(K) = \sup \{n \mid \xi_{2^n} + \xi_{2^n}^{-1} \in K\}$, if $\xi_4 \not\in K$.
    \end{itemize}
Moreover, with each prime number~$p > 2$ we associate a positive integer number~$\relpenalty=10000 e_p = [K(\xi_{p}): \QQ(\xi_{p^{m_p}})]$.

It is convenient to fit these invariants into the following diagram
\begin{center}
\begin{tikzpicture}
    
    \node (Q1) at (0,0) {$\QQ$};
    \node (Q2) at (0 + 8/3,4/3) {$K$};
    \node (Q3) at (0 + 0,8/3) {$K(\xi_p) = K(\xi_{p^{m_p}})$};
    \node (Q4) at (0 -8/3,4/3) {$\QQ(\xi_{p^{m_p}})$};
    \node (Q5) at (-24/10,1/3) {$p^{m_p-1}(p-1)$};
    \draw (Q1)--(Q2) node [pos=0.6, below,inner sep=0.2cm] {$d$};
    \draw (Q1)--(Q4); % node [pos=0.9, below,inner sep=0.25cm] {$p^{m_p-1}(p-1)$};
    \draw (Q4)--(Q3) node [pos=0.2, above,inner sep=0.2cm]{$e_{p}$};
    \draw (Q2)--(Q3) node [pos=0.2, above,inner sep=0.2cm]{$t_{p}$};

\end{tikzpicture}
\end{center}
\begin{comment}
    \node (Q5) at (5/2,4/3) {$F_p$};
    \node (Q6) at (5/2 + 4/3,8/3) {$K$};
    \node (Q7) at (5/2 -4/3,8/3) {$\QQ(\xi_{p})$};
    \node (Q8) at (5/2, -0) {$\QQ$};
    \node (Q9) at (5/2 + 4/3, 4/3) {.};
    \end{comment}
        \begin{comment}
    \draw (Q5)--(Q6);
    \draw (Q5)--(Q7);
    \draw (Q5)--(Q8) node [pos = 0.5, left, inner sep=0.15cm]{$r_p$};
    \end{comment}
With $2$ we associate integer number $e_2 = [K(\xi_{4}): \QQ(\xi_{2^{m_2}} + \xi^{-1}_{2^{m_2}})]$, 
\begin{comment}
a field $F_2 = \QQ(\xi_4) \cap K$ and a positive integer number $r_2 = [F_2:\QQ]$. An analogous diagrams can be drawn.
\end{comment}
\begin{rem} There is an obvious relation between introduced numbers for primes $p>2$, which follows from the diagram: $$p^{m_p-1}(p-1)\cdot e_p = d \cdot t_p.$$ We will use this equation without reference.
\end{rem}

Also, with any prime number $p$ and any positive integer number $n$ we associate Schur's and Serre's values $\nu_{p,n}^{\sch}$, $\nu_p^{\sch}$, $\nu_{p,n}^{\se}$ and $\nu_p^{\se}$ defined in Definition~\ref{def: schur and serre invariants}. For convenience, we repeat this definition here (remind that $\mathcal{P}$ denotes the set of all prime numbers):
\begin{itemize}
        \item $\nu^{\sch}_{p,n} = \nu^{\sch}_{p,n}(K) = m_p\left\lfloor \frac{n}{t_p}\right\rfloor + \left\lfloor \frac{n}{pt_p}\right\rfloor + \left\lfloor \frac{n}{p^2t_p}\right\rfloor + \dots$, if $p \neq 2$;
        \item $\nu^{\sch}_{p,n} = \nu^{\sch}_{p,n}(K) = m_2n+ \left\lfloor \frac{n}{2}\right\rfloor + \left\lfloor \frac{n}{2^2}\right\rfloor + \dots$, if $p=2$ and $\xi_4 \in K$;
        \item $\nu^{\sch}_{p,n} = \nu^{\sch}_{p,n}(K) = n + m_2\left\lfloor \frac{n}{2}\right\rfloor + \left\lfloor \frac{n}{2^2}\right\rfloor + \dots$, if $p = 2$ and $\xi_4 \not\in K$;
        \item $\nu^{\se}_{p,n} = \nu^{\se}_{p,n}(K) = m_p\left\lfloor\frac{n-1}{\varphi(t_p)}\right\rfloor + \nu_p((n-1)!)$;
        \item $\nu^{\sch}_n = \nu^{\sch}_n(K) =  \prod_{p \in \mathcal{P}} p^{\nu^{\sch}_{p,n}}$; 
        \item  $\nu^{\se}_n = \nu^{\se}_n(K) = \prod_{p \in \mathcal{P}} p^{\nu^{\se}_{p,n}}$.
    \end{itemize}

Two main goals of this section are to prove the following propositions.

\begin{propos}\label{proposition: bound of schur GL_4 <= 7}
    Let $K$ be a number field of degree $d \leqslant 7$. Then $$\nu^{\sch}_4 \leqslant 1\,132\,185\,600 < 10^{10}.$$
    %2^{15}\cdot3^{5}\cdot5^{4}\cdot7\cdot11^{4} = 
\end{propos}

\begin{propos}\label{proposition: bound of schur GL_3 <= 15}
    Let $K$ be a number field of degree $d \leqslant 15$. Then $$\nu^{\sch}_3 \leqslant 240\,045\,120 < 10^{9}.$$
    %2^{15}\cdot3^{5}\cdot5^{4}\cdot7\cdot11^{4} = 
\end{propos}

The following two standard lemmas turn out to be useful.

\begin{lemma}[{see for example~\cite[Chapter 6, Theorem 1.12]{Lang}}]\label{lemma: natural irrationalities}
    Let $F \subset K$ be a finite Galois extension, and $F \subset L$ be an arbitrary field extension. If $KL$ is a composite field of $K$ and $L$, then $$[KL:L] = [K:K\cap L].$$
\end{lemma}

\begin{cor}\label{corollary: nu(L) <= nu(K)}
    Let $K \subset L$ be an extension of number fields and $n$ be a positive integer number. Then $\nu_n^{\sch}(K) \leqslant \nu_n^{\sch}(L)$. 
\end{cor}

\begin{proof}
    It is enough to show that for any prime $p$ we have $t_p(K) \geqslant t_p(L)$ and~$\relpenalty=10000 m_p(K) \leqslant m_p(L)$. This immediately follows from definitions and Lemma~\ref{lemma: natural irrationalities}.
\end{proof}

%\begin{proof}
 %   Доказывать?
%\end{proof}

\begin{lemma}[{see for example~\cite[Chapter 6, Corollary 3.2]{Lang}}]\label{lemma: Q(n) cap Q(m)}
    Let $n$ and $m$ be two coprime integers. Then $\QQ(\xi_n) \cap \QQ(\xi_m) = \QQ$.
\end{lemma}

The following lemma provides a multiplicative bound on Schur's values. Of course, this bound is far from being suitable for us, but it has some advantages: it depends only on a degree of a field, and is easy to compute.

\begin{lemma}\label{lemma: rough multiplicative bound}
    Let $K$ be a number field of degree $d$. Let $p$ be a prime number and $n$ be a positive integer number. Then
    \begin{enumerate}
        \item if $p \geqslant 3$, one has $$\nu_{p,n}^{\sch} \leqslant B_{n,d,p} = (\nu_p(d) + 1)\left\lfloor\frac{n}{\frac{p-1}{\gcd(p-1,d)}} \right\rfloor +  \left\lfloor \frac{n}{p}\right\rfloor + \left\lfloor \frac{n}{p^2}\right\rfloor + \dots;$$
        \item if $p = 2$, one has $$\nu_{2,n}^{\sch} \leqslant B_{n,d,2} = n(\nu_2(d) + 1)+  \left\lfloor \frac{n}{2}\right\rfloor + \left\lfloor \frac{n}{2^2}\right\rfloor + \dots, \;  \text{if}\; d \; \text{is even};$$
        $$\nu_{2,n}^{\sch} \leqslant B_{n,d,2} = n + 2\left\lfloor \frac{n}{2}\right\rfloor + \left\lfloor \frac{n}{2^2}\right\rfloor + \dots, \;  \text{if}\; d \; \text{is odd}.$$
    \end{enumerate}
\end{lemma}
\begin{proof}
    Consider the case $p \geqslant 3$. By definition, we have $$\nu_{p,n}^{\sch} = m_p\left\lfloor \frac{n}{t_p}\right\rfloor + \left\lfloor \frac{n}{pt_p}\right\rfloor + \left\lfloor \frac{n}{p^2t_p}\right\rfloor + \dots$$
     As we remember invariants $m_p$ and $t_p$ depend on a field $K$, but we want to bound this value depending only on a degree $d$. Since $t_p \geqslant 1$, we can write the following bound
     $$\nu_{p,n}^{\sch} \leqslant m_p\left\lfloor \frac{n}{t_p}\right\rfloor + \left\lfloor \frac{n}{p}\right\rfloor + \left\lfloor \frac{n}{p^2}\right\rfloor + \dots,$$ where all summand except the first one does not depend on $K$. It is left to bound the first summand by the value depending only on $d$.
     
     It follows from the diagram
\begin{center}\hspace*{-2.5cm}
\begin{tikzpicture}
    
    \node (Q1) at (0,0) {$\QQ$};
    \node (Q2) at (4/3,4/3) {$K$};
    \node (Q3) at (0,8/3) {$K(\xi_p)$};
    \node (Q4) at (-4/3,4/3) {$\QQ(\xi_{p^{m_p}})$};

    \draw (Q1)--(Q2) node [pos=0.9, below,inner sep=0.25cm] {$d$};
    \draw (Q1)--(Q4) node [pos=2.3, below,inner sep=1.1cm] {$p^{m_p-1}(p-1)$};
    \draw (Q3)--(Q4);
    \draw (Q2)--(Q3) node [pos=0.2, above,inner sep=0.25cm]{$t_p$};
\end{tikzpicture}
\end{center}
\begin{comment}

\begin{center}
\begin{tikzpicture}
    
   \node (Q1) at (0,0) {$\QQ$};
    \node (Q2) at (4/3,4/3) {$K$};
    \node (Q3) at (0,8/3) {$K(\xi_p)$};
    \node (Q4) at (-4/3,4/3) {$\QQ(\xi_p)$};

    \draw (Q1)--(Q2) node [pos=0.9, below,inner sep=0.25cm] {$d$};
    \draw (Q1)--(Q4) node [pos=1.3, below,inner sep=0.50cm] {$p-1$};
    \draw (Q3)--(Q4);
    \draw (Q2)--(Q3) node [pos=0.2, above,inner sep=0.25cm]{$t$};
\end{tikzpicture}
\end{center}
\end{comment}
that $t_p$ is divisible by $$\frac{p^{m_p-1}(p-1)}{\gcd(p^{m_p-1}(p-1),d)} = \frac{p^{m_p-1}(p-1)}{p^{\min(m_p-1,\nu_p(d))}\gcd(p-1,d)} = \frac{p^{\max(0,m_p-\nu_p(d) -1)}(p-1)}{\gcd(p-1,d)},$$ which implies that $$m_p\left\lfloor\frac{n}{t_p}\right\rfloor \leqslant m_p\left\lfloor\frac{n}{p^{\max(0,m_p-\nu_p(d) -1})\cdot\frac{p-1}{\gcd(p-1,d)}} \right\rfloor.$$ For $m_p \leqslant \nu_p(d) + 1$, we have $p^{\max(0,m_p - \nu_p(d) - 1)}) = 1$ and clearly
$$m_p\left\lfloor\frac{n}{t_p}\right\rfloor \leqslant m_p\left\lfloor\frac{n}{\frac{p-1}{\gcd(p-1,d)}} \right\rfloor \leqslant (\nu_p(d) + 1)\left\lfloor\frac{n}{\frac{p-1}{\gcd(p-1,d)}} \right\rfloor.$$ 
For $m_p > \nu_p(d) + 1$ the same inequality follows from induction on the value $m_p -\nu_p(d)- 1$ and this proves assertion $(1)$.

Consider the case $p = 2$. If $\xi_4 \in K$, then we have $$\nu_{2,n}^{\sch} = m_2 n + \left\lfloor \frac{n}{2}\right\rfloor + \left\lfloor \frac{n}{2^2}\right\rfloor + \dots$$
Also, by Definition~\ref{def: cyclotomic invariants} there is a tower of fields $$\QQ \subset \QQ(\xi_{2^{m_2}}) \subset K$$ which implies an inequality $m_2 - 1 \leqslant \nu_2(d)$ and a bound 
\begin{equation}\label{eq: bound for 2 in case (a)}
\nu_{2,n}^{\sch} \leqslant n(\nu_2(d) + 1) + \left\lfloor \frac{n}{2}\right\rfloor + \left\lfloor \frac{n}{2^2}\right\rfloor + \dots
\end{equation}
If $\xi_4 \not\in K$, then we have $$\nu_{2,n}^{\sch} = n + m_2\left\lfloor \frac{n}{2}\right\rfloor + \left\lfloor \frac{n}{2^2}\right\rfloor + \dots$$ Again, by Definition~\ref{def: cyclotomic invariants} there is tower of fields $$\QQ \subset \QQ(\xi_{2^{m_2}} + \xi^{-1}_{2^{m_2}}) \subset K$$ which implies inequality $m_2 - 2 \leqslant \nu_2(d)$ and a bound 
\begin{equation}\label{eq: bound for 2 in case (b) and (c)}
\nu_{2,n}^{\sch} \leqslant n + (\nu_2(d) + 2)\left\lfloor \frac{n}{2}\right\rfloor + \left\lfloor \frac{n}{2^2}\right\rfloor + \dots
\end{equation}
It remains to note, that if $d$ is even, then bound~\eqref{eq: bound for 2 in case (a)} is worse then bound~\eqref{eq: bound for 2 in case (b) and (c)} and vice versa if $d$ is odd. This proves assertion $(2)$.
\end{proof}

\begin{comment}
\begin{rem}\label{rem: bound for t_2 = 1 and t_2 = 2}
    It is convenient for us to extract from the given proof, that if $t_2 = 1$, then we have a bound $$\nu_2(G) \leqslant m_2 n + \left\lfloor \frac{n}{2}\right\rfloor + \left\lfloor \frac{n}{2^2}\right\rfloor + \dots$$ and if $t_2 = 2$, then we have a bound $$\nu_2(G) \leqslant n + \textcolor{red}{(\nu_2(d) + 2)}\left\lfloor \frac{n}{2}\right\rfloor + \left\lfloor \frac{n}{2^2}\right\rfloor + \dots$$
\end{rem}
\end{comment}

\begin{rem}\label{rem: multiplicative bound}
Suppose we are in the setup of Lemma~\ref{lemma: rough multiplicative bound}. If we put 
$$B_{n,d} = \prod_{p \in \mathcal{P}} p^{B_{n,d,p}},$$
then we obviously have a bound $\nu_{n}^{\sch} \leqslant B_{n,d}$ by Lemma~\ref{lemma: rough multiplicative bound}.
\end{rem}

In cases $n=3$ and $n = 4$ let us fit these bounds into Tables~\ref{tabular: tabular GL_3(<=15)} and~\ref{tabular: tabular d<=21}, we will need this for applications. 
\begin{comment}
The first table corresponds to the case $n = 3$. 

\begin{table}[h!]
\centering
\begin{tabular}{ | c | c | }
\hline
$d$& $B_{3,d}$ \\ \hline
$1$ & $2^{10}\cdot3^{2} = 9216  = (4)$ \\
$2$ & $2^{13}\cdot3^{4}\cdot5\cdot7 = 23224320 = (8)$ \\
$3$ & $2^{10}\cdot3^{3}\cdot7 = 193536 =(6)$\\
$4$ & $
2^{16}\cdot3^{4}\cdot5^{3}\cdot7\cdot13 = 60383232000 =(11)$\\
$5$ & $2^{10}\cdot3^{2}\cdot11 = 101376 =( 6 )$\\
$6$ & $ 2^{13}\cdot3^{7}\cdot5\cdot7^{3}\cdot13\cdot19 = 7589266513920 =( 13 )
$\\
$7$ & $2^{10}\cdot3^{2}\cdot = 9216 =( 4 )$\\
$8$ & $2^{19}\cdot3^{4}\cdot5^{3}\cdot7\cdot13\cdot17 = 8212119552000 =( 13 )
$ \\
$9$ & $2^{10}\cdot3^{4}\cdot7\cdot19 = 11031552 =( 8 )l$ \\
$10$ & $2^{13}\cdot3^{4}\cdot5^{2}\cdot7\cdot11^{3}\cdot31 = 4791293337600 =( 13 )
$\\
$11$ & $2^{10}\cdot3^{2}\cdot23 = 211968 =( 6 )$\\
$12$ & $2^{16}\cdot3^{7}\cdot5^{3}\cdot7^{3}\cdot13^{3}\cdot19\cdot37 = 9491136702308352000 =( 19 )$\\
$13$ & $2^{10}\cdot3^{2}\cdot = 9216 =( 4 )$\\
$14$ & $2^{13}\cdot3^{4}\cdot5\cdot7^{2}\cdot29\cdot43 = 202725089280 =( 12 )$\\
$15$ & $2^{10}\cdot3^{3}\cdot7\cdot11\cdot31 = 65995776 =( 8 )$\\
\hline
\end{tabular}
\caption{}
\label{table: n = 3}
\end{table}
\end{comment}
\begin{table}[h!]
\centering
\begin{tabular}{ | c | c | }
\hline
$d$& $B_{3,d}$ \\ \hline
$1$ & $ 2^{5}\cdot3^{2}= 288$\\
$2$ & $2^{7}\cdot3^{4}\cdot5\cdot7 = 362\,880 < 10^6$ \\
$3$ & $2^{5}\cdot3^{3}\cdot7 = 6048$\\
$4$ & $2^{10}\cdot3^{4}\cdot5^{3}\cdot7\cdot13 = 943\,488\,000 < 10^{9}$\\
$5$ & $2^{5}\cdot3^{2}\cdot11 = 3168$\\
$6$ & $2^{7}\cdot3^{7}\cdot5\cdot7^{3}\cdot13\cdot19 = 118\,582\,289\,280 < 10^{12}$\\
$7$ & $2^{5}\cdot3^{2}= 288$\\
$8$ & 
$2^{13}\cdot3^{4}\cdot5^{3}\cdot7\cdot13\cdot17 = 128\,314\,368\,000 < 10^{12}$ \\
$9$ & $2^{5}\cdot3^{4}\cdot7\cdot19 = 344\,736 < 10^6$ \\
$10$ & $2^{7}\cdot3^{4}\cdot5^{2}\cdot7\cdot11^{3}\cdot31 = 74\,863\,958\,400 < 10^{11}$ \\
$11$ & $2^{5}\cdot3^{2}\cdot23 = 6624$ \\
$12$ &  $2^{10}\cdot3^{7}\cdot5^{3}\cdot7^{3}\cdot13^{3}\cdot19\cdot37 = 148\,299\,010\,973\,568\,000 < 10^{18}$  \\
$13$ & $2^{5}\cdot3^{2} = 288$ \\
$14$ & $2^{7}\cdot3^{4}\cdot5\cdot7^{2}\cdot29\cdot43 = 3\,167\,579\,520< 10^{10}$ \\
$15$ & $2^{5}\cdot3^{3}\cdot7\cdot11\cdot31 = 2\,062\,368 < 10^7$ \\
\hline
\end{tabular}
\caption{Bounds of $\nu_3^{\sch}$ depending only on a degree of a number field.}
\label{tabular: tabular GL_3(<=15)}
\end{table}

\begin{table}[h!]
\centering
\begin{tabular}{ | c | c | }
\hline
$d$& $B_{4,d}$ \\ \hline
$1$ & $ 2^{9}\cdot3^{3}\cdot5 = 69\,120 < 10^5$ \\
$2$ & $2^{11}\cdot3^{5}\cdot5^{2}\cdot7 = 87\,091\,200 < 10^8$ \\
$3$ & $2^{9}\cdot3^{5}\cdot5\cdot7^{2}\cdot13 = 396\,264\,960 < 10^9$\\
$4$ & $2^{15}\cdot3^{5}\cdot5^{4}\cdot7\cdot13\cdot17 = 7\,698\,862\,080\,000 < 10^{13}$\\
$5$ & $2^{9}\cdot3^{3}\cdot5^{2}\cdot11^{2}= 41\,817\,600 < 10^8$\\
$6$ & $2^{11}\cdot3^{9}\cdot5^{2}\cdot7^{4}\cdot13^{2}\cdot19 = 7\,769\,511\,593\,625\,600 < 10^{16}$\\
$7$ & $2^{9}\cdot3^{3}\cdot5\cdot29 = 2\,004\,480 < 10^7$\\
\hline
\end{tabular}
\caption{Bounds of $\nu_4^{\sch}$ depending only on a degree of a number field.}
\label{tabular: tabular d<=21}
\end{table}

While Remark~\ref{rem: multiplicative bound} gives a ``multiplicative'' bound, we want to obtain an ``additive'' one, since it is stronger. We will improve our bounds separately for ranks $3$ and $4$. But before this, let us prove two lemmas, which establish a connection between invariants $m_2$, $t_q$ and~$e_p$, corresponding to different primes. 

\begin{lemma}\label{lemma: boundin of m_2}
    Let $p > 2$ be a prime number and $K$ be a number field.
    Then \begin{enumerate}
        \item $m_2 \leqslant \nu_2(e_p) + 1$ if $\xi_4 \in K$ (in particular, $e_p$ is even);
        \item $m_2 \leqslant \nu_2(e_p) + 2$ if $\xi_4 \not\in K$.
     \end{enumerate}
\end{lemma}

\begin{proof}
    If $\xi_4 \in K$, then there is a tower of fields $$\QQ(\xi_{p^{m_p}}) \subset \QQ(\xi_{p^{m_p}}, \xi_{2^{m_2}}) \subset K(\xi_p),$$ which shows that $e_p$ is divisible by $2^{m_2 - 1}$, or, in other words, $m_2 - 1 \leqslant \nu_2(e_p)$ and this proves assertion $(1)$. If $\xi_4 \not\in K$, then there is a tower of fields $$\QQ(\xi_p) \subset \QQ(\xi_{p^{m_p}}, \xi_{2^{m_2}} + \xi^{-1}_{2^{m_2}}) \subset K(\xi_p),$$ which shows that $e_p$ is divisible by $2^{m_2 - 2}$, or, in other words, $m_2 - 2 \leqslant \nu_2(e_p)$ and this proves assertion $(2)$.
\end{proof}

\begin{lemma}\label{lemma: divisibility of gcd}
    Let $p$, $q > 2$ be two different primes and $\QQ \subset K$ be an extension of degree~$d$.
    Then $t_q \geqslant \frac{q-1}{\gcd(q-1,d,e_p)}$.
\end{lemma}

\begin{proof} Denote $F_q = \QQ(\xi_q) \cap K$ and $r_q = [F_q:\QQ]$. Clearly, we have~$r_q \mid d$ and $r_q \mid (q-1)$. Let us show that $r_q$ divides $e_p$ too. Since $F_q$ is a subfield of~$\QQ(\xi_q)$, it is Galois over $\QQ$, also $F_q \cap \QQ(\xi_p) = \QQ$ by Lemma~\ref{lemma: Q(n) cap Q(m)}. Then consider a diagram 
\begin{center}
\begin{tikzpicture}
    
    \node (Q1) at (0,0) {$\QQ$};
    \node (Q2) at (4/3,4/3) {$F_q$};
    \node (Q3) at (0,8/3) {$F_q(\xi_{p})$};
    \node (Q4) at (-4/3,4/3) {$\QQ(\xi_{p})$};
    \node (Q5) at (-1.2,1/2) {$p-1$};
    \draw (Q1)--(Q2) node [pos=0.9, below,inner sep=0.25cm] {$r_q$};
    \draw (Q1)--(Q4);% node [pos=0.9, below,inner sep=0.25cm] {$p-1$};
    \draw (Q4)--(Q3);
    \draw (Q2)--(Q3);
\end{tikzpicture}
\end{center}
By Lemma~\ref{lemma: natural irrationalities} we have $[F_q(\xi_p):\QQ(\xi_p)] = r_q$, and the following tower of fields $$\QQ(\xi_{p}) \subset F_{q}(\xi_{p}) \subset K(\xi_{p})$$ shows that $r_q$ divides $e_{p}$. Therefore, $r_q$ divides $\gcd(q-1,d,e_p)$. Now consider another diagram 
\begin{center}
\begin{tikzpicture}
    
    \node (Q1) at (0,0) {$F_q$};
    \node (Q2) at (4/3,4/3) {$K$};
    \node (Q3) at (0,8/3) {$K(\xi_{q})$};
    \node (Q4) at (-4/3,4/3) {$\QQ(\xi_{q})$};
    \node (Q5) at (0, -4/3) {$\QQ$};

    \draw (Q1)--(Q2) node [pos=0.9, below,inner sep=0.25cm] {$\frac{d}{r_q}$};;
    \draw (Q1)--(Q4) node [pos=0.9, below,inner sep=0.25cm] {$\frac{q-1}{r_q}$};
    \draw (Q4)--(Q3); % node [pos=0.2, above,inner sep=0.25cm]{$e_{q}$};
    \draw (Q2)--(Q3) node [pos=0.2, above,inner sep=0.25cm]{$t_{q}$};
    \draw (Q1)--(Q5) node [pos = 0.5, left, inner sep=0.15cm]{$r_q$};
\end{tikzpicture}
\end{center}
Again, by Lemma~\ref{lemma: natural irrationalities} we have $$t_q = \frac{q-1}{r_q} \geqslant \frac{q-1}{\gcd(q-1,d,e_p)}.$$
\end{proof}

Let us comment on previous lemmas informally. If we find a prime number $p > 2$ with small invariant~$e_p$, then Lemma~\ref{lemma: boundin of m_2} and Lemma~\ref{lemma: divisibility of gcd} show that there is an obstacle to appearance of other primes in the value $\nu_{n}^{\sch}$(since $\nu_{q,n}^{\sch}$ is smaller if $t_q$ is bigger for primes~$\relpenalty=10000 q>2$, and $\nu_{2,n}^{\sch}$ is smaller if $m_2$ is smaller).

\textbf{Rank 3.} 

The following lemma handles cases of special fields.
\begin{lemma}\label{lemma: rank = 3 bounding G if K = Q(xi)}
    Let $p > 2$ be a prime number and $k$ be a positive integer number. Let $K$ be a subfield of $\QQ(\xi_{4p^k})$.
\begin{enumerate}
    \item If $p \geqslant 5$, then $\nu_3^{\sch} \leqslant 2^7 \cdot 3 \cdot p^{k\left\lfloor \frac{3}{t_p}\right\rfloor}$;
    \item if $p = 3$, then $\nu_3^{\sch} \leqslant 2^{7} \cdot 3^{k\left\lfloor \frac{3}{t_3}\right\rfloor + \left\lfloor \frac{1}{t_3}\right\rfloor} $.
\end{enumerate}
\end{lemma}
\begin{proof}
    Let us prove the case $p \geqslant 5$, other case is absolutely analogous.
    
     Immediately note that we have $$\nu_{p,3}^{\sch} = m_p \left\lfloor \frac{3}{t_p}\right\rfloor \leqslant k\left\lfloor \frac{3}{t_p}\right\rfloor.$$ For prime number $q \neq p$, $q > 2$ we have $\QQ(\xi_{q}) \cap K = \QQ$ by Lemma~\ref{lemma: Q(n) cap Q(m)}. Then consider the diagram 
\begin{center}
\begin{tikzpicture}
    
    \node (Q1) at (0,0) {$\QQ$};
    \node (Q2) at (4/3,4/3) {$K$};
    \node (Q3) at (0,8/3) {$K(\xi_q)$};
    \node (Q4) at (-4/3,4/3) {$\QQ(\xi_{q})$};
    \node (Q5) at (-1.2,1/2) {$q-1$};
    \draw (Q1)--(Q2);% node [pos=0.9, below,inner sep=0.25cm] {$q-1$};
    \draw (Q1)--(Q4); %node [pos=0.9, below,inner sep=0.25cm] {$q-1$};
    \draw (Q4)--(Q3) node [pos=0.2, above,inner sep=0.25cm]{};
    \draw (Q2)--(Q3) node [pos=0.2, above,inner sep=0.25cm]{$t_{q}$};
\end{tikzpicture}
\end{center}
By Lemma~\ref{lemma: natural irrationalities} we have an equality $t_q = q - 1$, which implies equalities for $\nu_{q,3}^{\sch}$:
\begin{enumerate}
    \item if $q \geqslant 5$, then $\nu_{q,3}^{\sch} = 0$;
    \item if $q = 3$, then $\nu_{q,3}^{\sch} = m_3$.
\end{enumerate}
Moreover, it is clear that $m_3 = 1$, since~$K(\xi_3)$ does not contain elements $\xi_{3^m}$ for $m \geqslant 2$. Therefore, we have~$\nu_{3,3}^{\sch} = 1$. 

If $K$ contains $\xi_4$, then we have $\nu_{2,3}^{\sch} = 1 + 3m_2$.
Also, it is clear, that $m_2 = 2$, since~$\QQ(\xi_{4p^k})$ does not contain elements $\xi_{2^m}$ for $m \geqslant 3$, which means, that $\nu_{2,3}^{\sch} = 7$. If $K$ does not contain $\xi_4$, then $\nu_{2,3}^{\sch} = 5$, and finally we obtain an inequality $$\nu_{4}^{\sch} \leqslant 2^{7} \cdot 3 \cdot p^{k\left\lfloor \frac{3}{t_p}\right\rfloor}.$$ As we have mentioned, other case is absolutely analogous.
\end{proof}

Since the worst bound in Table~\ref{tabular: tabular GL_3(<=15)} corresponds to degree $12$. Let us first deal with this case. The following lemma is a more or less direct consequence from Lemma~\ref{lemma: bounding G if K = Q(xi)}, but they also give a first look on the impact of invariant $e_p$. 

\begin{lemma}\label{lemma: rank = 3, d = 12, e_p = 1}
    Let $K$ be a number field of degree $12$.
    Assume that one of the conditions holds:
    \begin{enumerate}
        \item there is a prime number $p > 2$ such that $e_p = 1$;
        \item $\xi_4 \in K$ and there is a prime number $p > 2$ such that $e_p = 2$.
    \end{enumerate}
    Then $$\nu_{3}^{\sch} \leqslant 843\,648 < 10^6.$$ 
\end{lemma}

\begin{proof}
    If the first condition holds, then $K(\xi_p) = \QQ(\xi_{p^{m_p}})$. If the second condition holds, then the field $K(\xi_p)$ is obtained from the field $\QQ(\xi_{p^{m_p}})$ by adjoining $\xi_4$, which means that~$K(\xi_p) = \QQ(\xi_{4p^{m_p}})$ and $\relpenalty=10000 K \subset \QQ(\xi_{4p^{m_p}})$. In both cases applying Lemma~\ref{lemma: rank = 3 bounding G if K = Q(xi)} we obtain a bound, depending on the actual value of $p$. Going over all the values of $p$ and remembering that $[K:\QQ] = 12$ we see, that the worst bound we can obtain corresponds to the case~$\relpenalty=10000 p = 13$: $$\nu_3^{\sch} \leqslant 2^{7} \cdot 3 \cdot 13^3 = 843\,648 < 10^6.$$
\end{proof}

The following lemma provides bound if there is a prime $p$ with~$e_p = 2$.

\begin{lemma}\label{lemma: rank=3, d = 12, e_p = 2}
    Let $K$ be a number field of degree $12$. Suppose there exists a prime number~$\relpenalty=10000 p > 2$ such that $e_p = 2$. Then $$\nu_{3}^{\sch} \leqslant  240\,045\,120 < 10^{9}.$$
\end{lemma}

\begin{proof}
    If there is another prime $q>2$ with $e_q = 1$, then we are done by Lemma~\ref{lemma: rank = 3, d = 12, e_p = 1}. Therefore, we can assume that for any other prime $q>2$ we have $e_q \geqslant 2$. If $\xi_4 \in K$, then $$\nu_{3}^{\sch} \leqslant  843\,648 < 10^6$$ by Lemma~\ref{lemma: rank = 3, d = 12, e_p = 1}. If $\xi_4 \not\in K$, then by Lemma~\ref{lemma: boundin of m_2} we have $m_2 \leqslant 3$ and by Lemma~\ref{lemma: divisibility of gcd} for any prime $q > 2$ other than $p$ we have $$t_q \geqslant \frac{q-1}{2}.$$ In particular, we have $\nu_{q,3}^{\sch} = 0$ for any $q > 7$ other than $p$ and $\nu_{2,3}^{\sch} \leqslant 6$. Let us find bounds for small odd primes. 

    If $q = 7$ other than $p$, then the equality $$7^{m_7 - 1}\cdot6\cdot e_7 = 12 \cdot t_7$$ shows that $m_7 = 1$, $t_7 \geqslant 3$ or $m_7 \geqslant 2$, $t_7 \geqslant 7$. In both cases, we have a bound~$\relpenalty=10000 \nu_{7,3}^{\sch} \leqslant 1$. If~$\relpenalty=10000 q = 5$ other than~$p$, then the equality $$5^{m_5 - 1}\cdot4\cdot e_5 = 12 \cdot t_5$$ shows that $m_5 = 1$, $t_5 \geqslant 2$ or $m_5 \geqslant 2$, $t_5 \geqslant 5$. In both cases, we have~$\nu_{5,3}^{\sch} \leqslant 1$. If $q = 3$ other than~$p$, then the same analysis does not provide a better bound than in Lemma~\ref{lemma: rough multiplicative bound}, i.e. we have~$\relpenalty=10000 \nu_{3,3}^{\sch} \leqslant 7$. 

    It is left to estimate $\nu_{p,3}^{\sch}$. And the bound comes from the standard equality: $$p^{m_p -1}\cdot(p-1) \cdot 2 = 12 \cdot t_p.$$ Note that this equality is very restrictive. Since $m_p$ and $t_p$ are positive integer numbers, clearly this equality could not be satisfied for some primes $p>2$. Straightforward computations provide the following bounds: 
    \begin{enumerate}
        \item $\nu_{p,3}^{\sch} \leqslant 1$, if $p = 13$, $19$;
        \item $\nu_{p,3}^{\sch} \leqslant 3$, if $p = 7$;
        \item $\nu_{p,3}^{\sch} \leqslant 7$, if $p = 3$;
        \item $\nu_{p,3}^{\sch} = 0$ for all other possible values of $p > 2$.
    \end{enumerate}
    Summarizing all these estimates, we obtain bounds  depending on the actual value of $p$. The worst one we can obtain corresponds to the case $p = 7$: $$\nu_3^{\sch} \leqslant 2^{6} \cdot 3^7 \cdot 5 \cdot 7^3 = 240\,045\,120 < 10^{9}.$$
\end{proof}

Now we are ready for final improvements of our bound in case of degree $12$.

\begin{lemma}\label{lemma: rank=3, d = 12 final}
    Let $K$ be a number field of degree $12$. Then $$\nu_{3}^{\sch} \leqslant 240\,045\,120 < 10^{9}.$$
\end{lemma}
\begin{proof}
If there exists a prime number $p > 2$ with $e_p = 1$, then we have $$\nu_{3}^{\sch} \leqslant  843\,648 < 10^6$$ by Lemma~\ref{lemma: rank = 3, d = 12, e_p = 1}. If there exists a prime number $p > 2$ with $e_p = 2$, then we have $$\nu_{3}^{\sch} \leqslant 240\,045\,120 < 10^{9}$$ by Lemma~\ref{lemma: rank=3, d = 12, e_p = 2}. Then assume that for any prime number $p > 2$ we have~$\relpenalty=10000 e_p \geqslant 3$. 

Immediately note, that we have $\nu^{\sch}_{2,3} \leqslant 10$ by Lemma~\ref{lemma: rough multiplicative bound}. For odd primes consider the standard equation 
\begin{equation}\label{eq:rank=3, d = 12 final}    
p^{m_p -1}\cdot(p-1) \cdot e_p = 12 \cdot t_p.\end{equation} Due to Table~\ref{tabular: tabular GL_3(<=15)} we are interested only in primes: $3$, $5$, $7$, $13$, $19$ and $37$. Substituting these primes and solving equation, we see that $t_{37}$, $t_{19} > 3$, which implies $\nu_{37,3}^{\sch} = \nu_{19,3}^{\sch} = 0$.

\begin{comment}
Also, it follows that $$m_{17}\cdot\left\lfloor\frac{4}{t_{17}}\right\rfloor \leqslant \left\lfloor\frac{4}{t_{17}}\right\rfloor, \; m_{13}\cdot\left\lfloor\frac{4}{t_{13}}\right\rfloor \leqslant \left\lfloor\frac{4}{t_{13}}\right\rfloor, \; m_{11}\cdot\left\lfloor\frac{4}{t_{11}}\right\rfloor \leqslant \left\lfloor\frac{4}{t_{11}}\right\rfloor, \;m_{7}\cdot\left\lfloor\frac{4}{t_{7}}\right\rfloor \leqslant \left\lfloor\frac{4}{t_{7}}\right\rfloor,$$ $$m_{5}\cdot\left\lfloor\frac{4}{t_{5}}\right\rfloor \leqslant 2\left\lfloor\frac{4}{t_{5}}\right\rfloor.$$
\end{comment}

Suppose that $\nu_{13,3}^{\sch} \neq 0$, then we have $t_{13} \leqslant 3$. Under this condition the only solution of $$13^{m_{13} - 1}\cdot12\cdot e_{13} = 12\cdot t_{13}$$ is $m_{13} = 1$, $e_{13} = t_{7} = 3$. In particular, we have $\nu_{13,3}^{\sch} \leqslant 1$. Applying Lemma~\ref{lemma: boundin of m_2} we find out that $\xi_4 \notin K$ (since $e_{13}$ is odd) and $m_2 \leqslant 2$. Therefore, $\nu_{2,3}^{\sch} \leqslant 5$. Applying Lemma~\ref{lemma: divisibility of gcd} for $p = 13$ we obtain inequalities  $t_{7} \geqslant 2$, $t_5 \geqslant 4$, and $t_3 \geqslant 2.$ In particular, we have $\nu_{5,3}^{\sch} = 0$. Solving~\eqref{eq:rank=3, d = 12 final} for primes $3$ and $7$ we obtain a bound $$\nu_{3}^{\sch} \leqslant 2^{5} \cdot 3^3\cdot 7 = 6048.$$ Further, assume that $\nu_{13,3}^{\sch} = 0.$

Suppose that $\nu_{7,3}^{\sch} \neq 0$, then we have $t_{7} \leqslant 3$. Under this condition there are two solutions of $$7^{m_{7} - 1}\cdot6\cdot e_{7} = 12\cdot t_{7}.$$ We have $m_7 = 1$. Also, we have $e_7 = 4$, $t_7 = 2$ or $e_7 = 6$, $t_7 = 3$. In both cases we have~$\relpenalty=10000 \nu_{7,3}^{\sch} \leqslant 1$. In the former case, solving~\eqref{eq:rank=3, d = 12 final} for primes $3$ and $5$ we obtain a bound $$\nu_{3}^{\sch} \leqslant 2^{10} \cdot 3^4\cdot 5^3 \cdot 7 = 72\,576\,000 < 10^8.$$ In the latter case, applying Lemma~\ref{lemma: divisibility of gcd} for $p = 7$ we obtain inequality $t_5 \geqslant 2$. Solving~\eqref{eq:rank=3, d = 12 final} for primes $3$ and $5$ we obtain a bound $$\nu_{3}^{\sch} \leqslant 2^{10} \cdot 3^4\cdot 5\cdot 7 = 2\,903\,040< 10^7.$$ Further, assume that $\nu_{7,3}^{\sch} = 0.$

Under all assumptions we already have a bound $\nu_{3}^{\sch} \leqslant 2^{10} \cdot 3^7 \cdot 5^3$ by Remark~\ref{rem: multiplicative bound} and  Table~\ref{tabular: tabular GL_3(<=15)}. Let us decrease the power of $3$ in this bound. Consider~\ref{eq:rank=3, d = 12 final} for $p = 3$: $$3^{m_3 - 1}\cdot 2 \cdot e_3 = 12 \cdot t_3.$$ Since $e_3 \geqslant 3$, then for all solutions we have $\nu_{3,3}^{\sch} \leqslant 4$. This leads to the bound $$\nu_{3}^{\sch} \leqslant 2^{10} \cdot 3^4 \cdot 5^3 = 10\,368\,000 < 10^{8}$$ and the lemma is proved.
\end{proof}

Now it remains for us to prove analogous lemmas for degrees $8$, $10$ and $14$.

\begin{lemma}\label{lemma: rank = 3, d = 8, e_p = 1}
    Let $K$ be a number field of degree $8$.
    Assume that one of the conditions holds:
    \begin{enumerate}
        \item there is a prime number $p > 2$ such that $e_p = 1$;
        \item $\xi_4 \in K$ and there is a prime number $p > 2$ such that $e_p = 2$.
    \end{enumerate} Then $$\nu_{3}^{\sch} \leqslant 48\,000 < 10^5.$$ 
\end{lemma}

\begin{proof}
    If the first condition holds, then $K(\xi_p) = \QQ(\xi_{p^{m_p}})$. If the second condition holds, then the field $K(\xi_p)$ is obtained from the field $\QQ(\xi_{p^{m_p}})$ by adjoining $\xi_4$, which means that~$K(\xi_p) = \QQ(\xi_{4p^{m_p}})$ and $\relpenalty=10000 K \subset \QQ(\xi_{4p^{m_p}})$. In both cases applying Lemma~\ref{lemma: rank = 3 bounding G if K = Q(xi)} we obtain a bound, depending on the actual value of $p$. Going over all the values of $p$ and remembering that $[K:\QQ] = 8$ we see, that the worst bound we can obtain corresponds to the case~$\relpenalty=10000 p = 5$: $$\nu_3^{\sch} \leqslant 2^{7} \cdot 3 \cdot 5^3 = 48\,000 < 10^5.$$
\end{proof}

\begin{lemma}\label{lemma: rank = 3, d = 10, e_p = 1}
    Let $K$ be a number field of degree $10$.
    Assume that one of the conditions holds:
    \begin{enumerate}
        \item there is a prime number $p > 2$ such that $e_p = 1$;
        \item $\xi_4 \in K$ and there is a prime number $p > 2$ such that $e_p = 2$.
    \end{enumerate} Then $$\nu_{3}^{\sch} \leqslant  511\,104 < 10^6.$$ 
\end{lemma}

\begin{proof}
    If the first condition holds, then $K(\xi_p) = \QQ(\xi_{p^{m_p}})$. If the second condition holds, then the field $K(\xi_p)$ is obtained from the field $\QQ(\xi_{p^{m_p}})$ by adjoining $\xi_4$, which means that~$K(\xi_p) = \QQ(\xi_{4p^{m_p}})$ and $\relpenalty=10000 K \subset \QQ(\xi_{4p^{m_p}})$. In both cases applying Lemma~\ref{lemma: rank = 3 bounding G if K = Q(xi)} we obtain a bound, depending on the actual value of $p$. Going over all the values of $p$ and remembering that $[K:\QQ] = 10$ we see, that the worst bound we can obtain corresponds to the case~$\relpenalty=10000 p = 11$: $$\nu_3^{\sch} \leqslant 2^{7} \cdot 3 \cdot 11^3 = 511\,104 < 10^6.$$
\end{proof}

\begin{lemma}\label{lemma: rank = 3, d = 14, e_p = 1}
    Let $K$ be a number field of degree $14$.
    Assume that one of the conditions holds:
    \begin{enumerate}
        \item there is a prime number $p > 2$ such that $e_p = 1$;
        \item $\xi_4 \in K$ and there is a prime number $p > 2$ such that $e_p = 2$.
    \end{enumerate} Then $$\nu_{3}^{\sch} \leqslant 48\,000 < 10^5.$$ 
\end{lemma}

\begin{proof}
     If the first condition holds, then $K(\xi_p) = \QQ(\xi_{p^{m_p}})$. If the second condition holds, then the field $K(\xi_p)$ is obtained from the field $\QQ(\xi_{p^{m_p}})$ by adjoining $\xi_4$, which means that~$K(\xi_p) = \QQ(\xi_{4p^{m_p}})$ and $\relpenalty=10000 K \subset \QQ(\xi_{4p^{m_p}})$. In both cases applying Lemma~\ref{lemma: rank = 3 bounding G if K = Q(xi)} we obtain a bound, depending on the actual value of $p$. Going over all the values of $p$ and remembering that $[K:\QQ] = 14$ we see, that the worst bound we can obtain corresponds to the case~$\relpenalty=10000 p = 5$: $$\nu_3^{\sch} \leqslant 2^{7} \cdot 3 \cdot 5^3 = 48\,000 < 10^5.$$
\end{proof}

\begin{lemma}\label{lemma: rank=3, d = 8, e_p = 2}
    Let $K$ be a number field of degree $8$. Suppose there exists a prime number~$\relpenalty=10000 p > 2$ such that $e_p = 2$. Then $$\nu_{3}^{\sch} \leqslant  4\,536\,000 < 10^{7}.$$
\end{lemma}

\begin{proof}
    If there is another prime $q>2$ with $e_q = 1$, then we are done by Lemma~\ref{lemma: rank = 3, d = 8, e_p = 1}. Therefore, we can assume that for any other prime $q>2$ we have $e_q \geqslant 2$. If $\xi_4 \in K$, then $$\nu_{3}^{\sch} \leqslant  48\,000 < 10^5$$ by Lemma~\ref{lemma: rank = 3, d = 8, e_p = 1}. If $\xi_4 \not\in K$, then by Lemma~\ref{lemma: boundin of m_2} we have $m_2 \leqslant 3$ and by Lemma~\ref{lemma: divisibility of gcd} for any prime $q > 2$ other than $p$ we have $$t_q \geqslant \frac{q-1}{2}.$$ In particular, we have $\nu_{q,3}^{\sch} = 0$ for any $q > 7$ other than $p$ and $\nu_{2,3}^{\sch} \leqslant 6$. Let us find bounds for small odd primes. 

    If $q = 7$ other than $p$, then the equality $$7^{m_7 - 1}\cdot6\cdot e_7 = 8 \cdot t_7$$ shows that $m_7 = 1$, $t_7 \geqslant 3$ or $m_7 \geqslant 2$, $t_7 \geqslant 21$. In both cases, we have a bound~$\relpenalty=10000 \nu_{7,3}^{\sch} \leqslant 1$. If $q = 5$ other than~$p$, then the equality $$5^{m_5 - 1}\cdot4\cdot e_5 = 8 \cdot t_5$$ shows that $m_5 = 1$, $t_5 \geqslant 2$ (as we have mentioned before) or $m_5 \geqslant 2$, $t_5 \geqslant 5$. In both cases, we have~$\nu_{5,3}^{\sch} \leqslant 1$. If $q = 3$ other than~$p$, then the same analysis does not provide a better bound than in Lemma~\ref{lemma: rough multiplicative bound}, i.e. we have~$\relpenalty=10000 \nu_{3,3}^{\sch} \leqslant 4$. 

    It is left to estimate $\nu_{p,3}^{\sch}$. And the bound comes from the standard equality: $$p^{m_p -1}\cdot(p-1) \cdot 2 = 8 \cdot t_p.$$ Note that this equality is very restrictive. Since $m_p$ and $t_p$ are positive integer numbers, clearly this equality could not be satisfied for some primes $p>2$. Straightforward computations provide the following bounds: 
    \begin{enumerate}
        \item $\nu_{p,3}^{\sch} \leqslant 1$, if $p = 13$;
        \item $\nu_{p,3}^{\sch} \leqslant 3$, if $p = 5$;
        \item $\nu_{p,3}^{\sch} = 0$ for all other possible values of $p > 2$.
    \end{enumerate}
    Summarizing all these estimates, we obtain bounds  depending on the actual value of $p$. The worst one we can obtain corresponds to the case $p = 5$: $$\nu_3^{\sch} \leqslant 2^{6} \cdot 3^4 \cdot 5^3 \cdot 7 = 4\,536\,000 < 10^{7}.$$
\end{proof}

\begin{lemma}\label{lemma: rank=3, d = 10, e_p = 2}
    Let $K$ be a number field of degree $10$. Suppose there exists a prime number~$\relpenalty=10000 p > 2$ such that $e_p = 2$. Then $$\nu_{3}^{\sch} \leqslant  1\,995\,840 < 10^{7}.$$
\end{lemma}

\begin{proof}
     If there is another prime $q>2$ with $e_q = 1$, then we are done by Lemma~\ref{lemma: rank = 3, d = 10, e_p = 1}. Therefore, we can assume that for any other prime $q>2$ we have $e_q \geqslant 2$. If $\xi_4 \in K$, then $$\nu_{3}^{\sch} \leqslant   511\,104 < 10^6$$ by Lemma~\ref{lemma: rank = 3, d = 10, e_p = 1}. If $\xi_4 \not\in K$, then by Lemma~\ref{lemma: boundin of m_2} we have $m_2 \leqslant 3$ and by Lemma~\ref{lemma: divisibility of gcd} for any prime $q > 2$ other than $p$ we have $$t_q \geqslant \frac{q-1}{2}.$$ In particular, we have $\nu_{q,3}^{\sch} = 0$ for any $q > 7$ other than $p$ and $\nu_{2,3}^{\sch} \leqslant 6$. Let us find bounds for small odd primes. 

    If $q = 7$ other than $p$, then the equality $$7^{m_7 - 1}\cdot6\cdot e_7 = 10 \cdot t_7$$ shows that $m_7 = 1$, $t_7 \geqslant 3$ or $m_7 \geqslant 2$, $t_7 \geqslant 21$. In both cases, we have a bound~$\relpenalty=10000 \nu_{7,3}^{\sch} \leqslant 1$. If $q = 5$ other than~$p$, then the equality $$5^{m_5 - 1}\cdot4\cdot e_5 = 10 \cdot t_5$$ shows that $m_5 = 1$ and $t_5 \geqslant 2$ or $m_5 \geqslant 2$ and $t_5 \geqslant 10$. In both cases, we have~$\nu_{5,3}^{\sch} \leqslant 1$. If $q = 3$ other than~$p$, then the same analysis does not provide a better bound than in Lemma~\ref{lemma: rough multiplicative bound}, i.e. we have~$\relpenalty=10000 \nu_{3,3}^{\sch} \leqslant 4$. 

    It is left to estimate $\nu_{p,3}^{\sch}$. And the bound comes from the standard equality: $$p^{m_p -1}\cdot(p-1) \cdot 2 = 10 \cdot t_p.$$ Note that this equality is very restrictive. Since $m_p$ and $t_p$ are positive integer numbers, clearly this equality could not be satisfied for some primes $p>2$. Straightforward computations provide the following bounds: 
    \begin{enumerate}
        \item $\nu_{p,3}^{\sch} \leqslant 1$, if $p = 11$;
        \item $\nu_{p,3}^{\sch} = 0$ for all other possible values of $p > 2$.
    \end{enumerate}
    Summarizing all these estimates, we obtain bounds  depending on the actual value of $p$. The worst one we can obtain corresponds to the case $p = 11$: $$\nu_3^{\sch} \leqslant 2^{6} \cdot 3^4 \cdot 5 \cdot 7 \cdot 11 = 1\,995\,840 < 10^{7}.$$
\end{proof}

\begin{lemma}\label{lemma: rank=3, d = 14, e_p = 2}
    Let $K$ be a number field of degree $14$. Suppose there exists a prime number~$\relpenalty=10000 p > 2$ such that $e_p = 2$. Then $$\nu_{3}^{\sch} \leqslant  181\,440 < 10^{6}.$$
\end{lemma}

\begin{proof}
    If there is another prime $q>2$ with $e_q = 1$, then we are done by Lemma~\ref{lemma: rank = 3, d = 14, e_p = 1}. Therefore, we can assume that for any other prime $q>2$ we have $e_q \geqslant 2$. If $\xi_4 \in K$, then $$\nu_{3}^{\sch} \leqslant  48\,000 < 10^5$$ by Lemma~\ref{lemma: rank = 3, d = 14, e_p = 1}. If $\xi_4 \not\in K$, then by Lemma~\ref{lemma: boundin of m_2} we have $m_2 \leqslant 3$ and by Lemma~\ref{lemma: divisibility of gcd} for any prime $q > 2$ other than $p$ we have $$t_q \geqslant \frac{q-1}{2}.$$ In particular, we have $\nu_{q,3}^{\sch} = 0$ for any $q > 7$ other than $p$ and $\nu_{2,3}^{\sch} \leqslant 6$. Let us find bounds for small odd primes. 

    If $q = 7$ other than $p$, then the equality $$7^{m_7 - 1}\cdot6\cdot e_7 = 14 \cdot t_7$$ shows that $m_7 = 1$, $t_7 \geqslant 3$ or $m_7 \geqslant 2$, $t_7 \geqslant 6$ (remember that we assume $e_7 \geqslant 2$). In both cases, we have a bound~$\relpenalty=10000 \nu_{7,3}^{\sch} \leqslant 1$. If $q = 5$ other than~$p$, then the equality $$5^{m_5 - 1}\cdot4\cdot e_5 = 14 \cdot t_5$$ shows that $m_5 = 1$, $t_5 \geqslant 2$ or $m_5 \geqslant 2$, $t_5 \geqslant 10$. In both cases, we have~$\nu_{5,3}^{\sch} \leqslant 1$. If $q = 3$ other than~$p$, then the same analysis does not provide a better bound than in Lemma~\ref{lemma: rough multiplicative bound}, i.e. we have~$\relpenalty=10000 \nu_{3,3}^{\sch} \leqslant 4$. 

    It is left to estimate $\nu_{p,3}^{\sch}$. And the bound comes from the standard equality: $$p^{m_p -1}\cdot(p-1) \cdot 2 = 14 \cdot t_p.$$ Note that this equality is very restrictive. Since $m_p$ and $t_p$ are positive integer numbers, clearly this equality could not be satisfied for some primes $p>2$. Straightforward computations show that $\nu_{p,3}^{\sch} = 0$ for all possible values of $p$ and we have $$\nu_3^{\sch} \leqslant 2^{6} \cdot 3^4 \cdot 5 \cdot 7 = 181\,440 < 10^{6}.$$
\end{proof}

In the following three lemmas, we are not aiming for the best possible bound, it will suffice for us to have bounds that do not exceed the bounds from Lemma~\ref{lemma: rank=3, d = 12 final}.

\begin{lemma}\label{lemma: rank=3, d = 8 final}
    Let $K$ be a number field of degree $8$. Then $$\nu_{3}^{\sch} \leqslant 23\,224\,320 < 10^{8}.$$
\end{lemma}

\begin{proof}
If there exists a prime number $p > 2$ with $e_p = 1$, then we have $$\nu_{3}^{\sch} \leqslant  48\,000 < 10^5$$ by Lemma~\ref{lemma: rank = 3, d = 8, e_p = 1}. If there exists a prime number $p > 2$ with $e_p = 2$, then we have $$\nu_{3}^{\sch} \leqslant 4\,536\,000 < 10^{7}$$ by Lemma~\ref{lemma: rank=3, d = 8, e_p = 2}. Then assume that for any prime number $p > 2$ we have~$\relpenalty=10000 e_p \geqslant 3$. 

For odd primes consider the standard equation 
\begin{equation}\label{eq:rank=3, d = 8 final}    
p^{m_p -1}\cdot(p-1) \cdot e_p = 8 \cdot t_p.\end{equation} Due to Table~\ref{tabular: tabular GL_3(<=15)} we are interested only in primes: $3$, $5$, $7$, $13$ and $17$. Substituting these primes and solving the equation, we see that $t_{17}$, $t_{13} > 3$, which implies $\nu_{17,3}^{\sch} = \nu_{13,3}^{\sch} = 0$. In particular, we already have a bound $\nu_3^{\sch} \leqslant 2^{13} \cdot 3^4 \cdot 5^3 \cdot 7$ by Remark~\ref{rem: multiplicative bound} and  Table~\ref{tabular: tabular GL_3(<=15)}. Let us decrease the power of $5$ in this bound. Consider~\ref{eq:rank=3, d = 8 final} for $p = 5$: $$5^{m_5 - 1}\cdot 4 \cdot e_5 = 8 \cdot t_5.$$ Since $e_5 \geqslant 3$, then for all solutions we have $\nu_{5,3}^{\sch} \leqslant 1$. This leads to the bound $$\nu_{3}^{\sch} \leqslant 2^{13} \cdot 3^4 \cdot 5 \cdot 7 = 23\,224\,320 < 10^{8}$$ and the lemma is proved.
\end{proof}

\begin{lemma}\label{lemma: rank=3, d = 10 final}
    Let $K$ be a number field of degree $10$. Then $$\nu_{3}^{\sch} \leqslant 19\,958\,400 < 10^{8}.$$
\end{lemma}

\begin{proof}
If there exists a prime number $p > 2$ with $e_p = 1$, then we have $$\nu_{3}^{\sch} \leqslant  511\,104 < 10^6$$ by Lemma~\ref{lemma: rank = 3, d = 10, e_p = 1}. If there exists a prime number $p > 2$ with $e_p = 2$, then we have $$\nu_{3}^{\sch} \leqslant 1\,995\,840 < 10^{7}$$ by Lemma~\ref{lemma: rank=3, d = 8, e_p = 2}. Then assume that for any prime number $p > 2$ we have~$\relpenalty=10000 e_p \geqslant 3$. 

For odd primes consider the standard equation 
\begin{equation}\label{eq:rank=3, d = 10 final}    
p^{m_p -1}\cdot(p-1) \cdot e_p = 10 \cdot t_p.\end{equation} Due to Table~\ref{tabular: tabular GL_3(<=15)} we are interested only in primes: $3$, $5$, $7$, $11$ and $31$. Substituting these primes and solving the equation, we see that $t_{31} > 3$, which implies $\nu_{31,3}^{\sch} = 0$. In particular, we already have a bound $\nu_3^{\sch} \leqslant 2^{7} \cdot 3^4 \cdot 5^2 \cdot 7 \cdot 11^3$ by Remark~\ref{rem: multiplicative bound} and  Table~\ref{tabular: tabular GL_3(<=15)}. Let us decrease the power of $11$ in this bound. Consider~\ref{eq:rank=3, d = 10 final} for $p = 11$: $$11^{m_{11} - 1}\cdot 10 \cdot e_{11} = 10 \cdot t_{11}.$$ Since $e_{11} \geqslant 3$, then for all solutions we have $\nu_{11,3}^{\sch} \leqslant 1$. This leads to the bound $$\nu_{3}^{\sch} \leqslant 2^{7} \cdot 3^4 \cdot 5^2 \cdot 7 \cdot 11 = 19\,958\,400 < 10^{8}$$ and the lemma is proved.
\end{proof}

\begin{lemma}\label{lemma: rank=3, d = 14 final}
    Let $K$ be a number field of degree $14$. Then $$\nu_{3}^{\sch} \leqslant  2\,540\,160 < 10^7.$$
\end{lemma}

\begin{proof}
If there exists a prime number $p > 2$ with $e_p = 1$, then we have $$\nu_{3}^{\sch} \leqslant  48\,000 < 10^5$$ by Lemma~\ref{lemma: rank = 3, d = 14, e_p = 1}. If there exists a prime number $p > 2$ with $e_p = 2$, then we have $$\nu_{3}^{\sch} \leqslant 181\,440 < 10^{6}$$ by Lemma~\ref{lemma: rank=3, d = 8, e_p = 2}. Then assume that for any prime number $p > 2$ we have~$\relpenalty=10000 e_p \geqslant 3$. 

For odd primes consider the standard equation 
\begin{equation}\label{eq:rank=3, d = 14 final}    
p^{m_p -1}\cdot(p-1) \cdot e_p = 14 \cdot t_p.\end{equation} Due to Table~\ref{tabular: tabular GL_3(<=15)} we are interested only in primes: $3$, $5$, $7$, $29$ and $43$. Substituting these primes and solving the equation, we see that $t_{41}$, $t_{29} > 3$, which implies $\nu_{41,3}^{\sch} = \nu_{29,3}^{\sch} = 0$. Therefore, we have a bound $$\nu_3^{\sch} \leqslant 2^{7} \cdot 3^4 \cdot 5 \cdot 7^2 = 2\,540\,160 < 10^7$$ by Remark~\ref{rem: multiplicative bound} and Table~\ref{tabular: tabular GL_3(<=15)}.
\end{proof}

Now everything is prepared to prove Proposition~\ref{proposition: bound of schur GL_3 <= 15}

\begin{proof}[of Proposition~\ref{proposition: bound of schur GL_3 <= 15}]
    For $d = 12$ result follows from Lemma~\ref{lemma: rank=3, d = 12 final}. For $d = 1$, $2$, $3$, $5$, $7$, $9$, $11$, $13$, $15$ result follows from Remark~\ref{rem: multiplicative bound} and  Table~\ref{tabular: tabular GL_3(<=15)}. For $d = 4$ result follows from Corollary~\ref{corollary: nu(L) <= nu(K)} applied to the extension $K \subset K(\sqrt[3]{2})$ and Lemma~\ref{lemma: rank=3, d = 12 final}. For~$\relpenalty=10000 d = 6$ note, that $K$ could not contain $\sqrt{2}$ and $\xi_4$ at the same time, therefore, result follows from Corollary~\ref{corollary: nu(L) <= nu(K)} applied either to the extension $K \subset K(\sqrt{2})$ or to $K \subset K(\xi_4)$ and Lemma~\ref{lemma: rank=3, d = 12 final}. For $d = 8$, $10$, $14$ result follows from Lemmas~\ref{lemma: rank=3, d = 8 final},~\ref{lemma: rank=3, d = 10 final} and~\ref{lemma: rank=3, d = 14 final}.
\end{proof}

\textbf{Rank 4.} 

The following lemma handles cases of special fields.

\begin{lemma}\label{lemma: bounding G if K = Q(xi)}
    Let $p > 2$ be a prime number and $k$ be a positive integer number. Let $K$ be a subfield of $\QQ(\xi_{4p^k})$.
\begin{enumerate}
    \item If $p > 5$, then $\nu_4^{\sch} \leqslant 2^{11} \cdot 3^2 \cdot 5 \cdot p^{k\left\lfloor \frac{4}{t_p}\right\rfloor}$;
    \item if $p = 5$, then $\nu_4^{\sch} \leqslant 2^{11} \cdot 3^2 \cdot 5^{k\left\lfloor \frac{4}{t_5}\right\rfloor}$;
    \item if $p = 3$, then $\nu_4^{\sch} \leqslant 2^{11} \cdot 5 \cdot 3^{k\left\lfloor \frac{4}{t_3}\right\rfloor+ \left\lfloor \frac{4}{3t_3}\right\rfloor}$.
\end{enumerate}
\end{lemma}
\begin{proof}
    Let us prove the case $p \geqslant 5$, other cases are absolutely analogous.
    
     Immediately note that we have $$\nu_{p,4}^{\sch} = m_p \left\lfloor \frac{4}{t_p}\right\rfloor \leqslant k\left\lfloor \frac{4}{t_p}\right\rfloor.$$ For prime number $q \neq p$, $q > 2$ we have $\QQ(\xi_{q}) \cap K = \QQ$ by Lemma~\ref{lemma: Q(n) cap Q(m)}. Then consider the diagram 
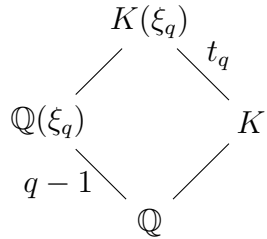
\begin{center}
\begin{tikzpicture}
    
    \node (Q1) at (0,0) {$\QQ$};
    \node (Q2) at (4/3,4/3) {$K$};
    \node (Q3) at (0,8/3) {$K(\xi_q)$};
    \node (Q4) at (-4/3,4/3) {$\QQ(\xi_{q})$};
    \node (Q5) at (-1.2,1/2) {$q-1$};
    \draw (Q1)--(Q2);% node [pos=0.9, below,inner sep=0.25cm] {$q-1$};
    \draw (Q1)--(Q4); %node [pos=0.9, below,inner sep=0.25cm] {$q-1$};
    \draw (Q4)--(Q3) node [pos=0.2, above,inner sep=0.25cm]{};
    \draw (Q2)--(Q3) node [pos=0.2, above,inner sep=0.25cm]{$t_{q}$};
\end{tikzpicture}
\end{center}
By Lemma~\ref{lemma: natural irrationalities} we have an equality $t_q = q - 1$, which implies inequalities for $\nu_{q,4}^{\sch}$:
\begin{enumerate}
    \item if $q > 5$, then $\nu_{q,4}^{\sch} = 0$;
    \item if $q = 5$, then $\nu_{q,4}^{\sch} = m_5$;
    \item if $q = 3$, then $\nu_{q,4}^{\sch} = 2m_3$.
\end{enumerate}
Moreover, it is clear that $m_5 = m_3 = 1$, since $K(\xi_5)$ does not contain elements $\xi_{5^m}$ for~$m \geqslant 2$ and~$K(\xi_3)$ does not contain elements $\xi_{3^m}$ for $m \geqslant 2$. Therefore, we have~$\relpenalty=10000 \nu_{5,4}^{\sch} = 1$ and $\nu_{3,4}^{\sch} = 2$. 

If $K$ contains $\xi_4$, then we have $\nu_{2,4}^{\sch} = 3 + 4m_2$.
Also, it is clear, that $m_2 = 2$, since~$\QQ(\xi_{4p^k})$ does not contain elements $\xi_{2^m}$ for $m \geqslant 3$, which means, that $\nu_{2,4}^{\sch} = 11$. If $K$ does not contain $\xi_4$, then $\nu_{2,4}^{\sch} = 9$, and finally we obtain an inequality $$\nu_{4}^{\sch} \leqslant 2^{11} \cdot 3^2 \cdot 5 \cdot p^{k\left\lfloor \frac{4}{t_p}\right\rfloor}.$$ As we have mentioned, other cases are absolutely analogous.
\end{proof}

The following lemmas are more or less direct consequences from Lemma~\ref{lemma: bounding G if K = Q(xi)}, but they also give a first look on the impact of invariant $e_p$.

\begin{lemma}\label{lemma: d = 6, e_p = 1}
    Let $K$ be a number field of degree $6$.
    Assume that one of the conditions holds:
    \begin{enumerate}
        \item there is a prime number $p > 2$ such that $e_p = 1$;
        \item $\xi_4 \in K$ and there is a prime number $p > 2$ such that $e_p = 2$.
    \end{enumerate} Then $$\nu_{4}^{\sch} \leqslant  221\,276\,160 < 10^9.$$ 
\end{lemma}

\begin{proof}
    If the first condition holds, then $K(\xi_p) = \QQ(\xi_{p^{m_p}})$. If the second condition holds, then the field $K(\xi_p)$ is obtained from the field $\QQ(\xi_{p^{m_p}})$ by adjoining $\xi_4$, which means that~$K(\xi_p) = \QQ(\xi_{4p^{m_p}})$ and $\relpenalty=10000 K \subset \QQ(\xi_{4p^{m_p}})$. In both cases applying Lemma~\ref{lemma: bounding G if K = Q(xi)} we obtain a bound, depending on the actual value of $p$. Going over all the values of $p$ and remembering that $[K:\QQ] = 6$ we see, that the worst bound we can obtain corresponds to the case~$\relpenalty=10000 p = 7$: $$\nu_4^{\sch} \leqslant 2^{11} \cdot 3^2 \cdot 5 \cdot 7^4 = 221\,276\,160 < 10^{9}.$$
\end{proof}

\begin{lemma}\label{lemma: d = 4, e_p = 1}
    Let $K$ be a number field of degree $4$.
    Assume that one of the conditions holds:
    \begin{enumerate}
        \item there is a prime number $p > 2$ such that $e_p = 1$;
        \item $\xi_4 \in K$ and there is a prime number $p > 2$ such that $e_p = 2$.
    \end{enumerate} Then $$\nu_{4}^{\sch} \leqslant  11\,520\,000 < 10^{8}.$$ 
\end{lemma}

\begin{proof}
    If the first condition holds, then $K(\xi_p) = \QQ(\xi_{p^{m_p}})$. If the second condition holds, then the field $K(\xi_p)$ is obtained from the field $\QQ(\xi_{p^{m_p}})$ by adjoining $\xi_4$, which means that~$K(\xi_p) = \QQ(\xi_{4p^{m_p}})$ and $\relpenalty=10000 K \subset \QQ(\xi_{4p^{m_p}})$. In both cases applying Lemma~\ref{lemma: bounding G if K = Q(xi)} we obtain a bound, depending on the actual value of $p$. Going over all the values of $p$ and remembering that $[K:\QQ] = 4$ we see, that the worst bound we can obtain corresponds to the case~$\relpenalty=10000 p = 5$: $$\nu_4^{\sch} \leqslant 2^{11} \cdot 3^2 \cdot 5^4 =11\,520\,000 < 10^{8}.$$
\end{proof}

\begin{lemma}\label{lemma: d = 6, e_p = 2}
    Let $K$ be a number field of degree $6$. Suppose there exists a prime number~$\relpenalty=10000 p > 2$ such that $e_p = 2$. Then $$\nu_{4}^{\sch} \leqslant 1\,132\,185\,600 < 10^{10}.$$
\end{lemma}

\begin{proof}
    If there is another prime $q>2$ with $e_q = 1$, then we are done by Lemma~\ref{lemma: d = 6, e_p = 1}. Therefore, we can assume that for any other prime $q>2$ we have $e_q \geqslant 2$. If $\xi_4 \in K$, then $$\nu_{4}^{\sch} \leqslant  221\,276\,160 < 10^9$$ by Lemma~\ref{lemma: d = 6, e_p = 1}. If $\xi_4 \not\in K$, then by Lemma~\ref{lemma: boundin of m_2} we have $m_2 \leqslant 3$ and by Lemma~\ref{lemma: divisibility of gcd} for any prime $q > 2$ other than $p$ we have $$t_q \geqslant \frac{q-1}{2}.$$ In particular, we have $\nu_{q,4}^{\sch} = 0$ for any $q > 7$ other than $p$ and $\nu_{2,4}^{\sch} \leqslant 11$. Let us find bounds for small odd primes. 

    If $q = 7$ other than $p$, then the equality $$7^{m_7 - 1}\cdot6\cdot e_7 = 6 \cdot t_7$$ shows that $m_7 = 1$, $t_7 \geqslant 3$ (as we have mentioned before) or $m_7 \geqslant 2$, $t_7 \geqslant 14$. In both cases, we have a bound~$\relpenalty=10000 \nu_{7,4}^{\sch} \leqslant 1$. If $q = 5$ other than~$p$, then the equality $$5^{m_5 - 1}\cdot4\cdot e_5 = 6 \cdot t_5$$ shows that $m_5 = 1$, $t_5 \geqslant 2$ or $m_5 \geqslant 2$, $t_5 \geqslant 10$. In both cases, we have~$\nu_{5,4}^{\sch} \leqslant 2$. If~$\relpenalty=10000 q = 3$ other than~$p$, then the equality $$3^{m_3 - 1}\cdot 2 \cdot e_3 = 6 \cdot t_3$$ shows that $m_3 = 1$, $t_3 = 1$ or $m_3 = 2$, $t_3 \geqslant 2$ or $m_3 \geqslant 3$, $t_3 \geqslant 6$. In all cases, we have~$\nu_{3,4}^{\sch} \leqslant 5$.

    It is left to estimate $\nu_{p,4}^{\sch}$. And the bound comes from the standard equality: $$p^{m_p -1}\cdot(p-1) \cdot 2 = 6 \cdot t_p.$$ Note that this equality is very restrictive. Since $m_p$ and $t_p$ are positive integer numbers, clearly this equality could not be satisfied for some primes $p>2$. Straightforward computations provide the following bounds: 
    \begin{enumerate}
        \item $\nu_{p,4}^{\sch} \leqslant 1$, if $p = 13$;
        \item $\nu_{p,4}^{\sch} \leqslant 2$, if $p = 7$;
        \item $\nu_{p,4}^{\sch} \leqslant 5$, if $p = 3$;
        \item $\nu_{p,4}^{\sch} = 0$ for all other possible values of $p > 2$.
    \end{enumerate}
    Summarizing all these estimates, we obtain bounds  depending on the actual value of $p$. The worst one we can obtain corresponds to the case $p = 13$: $$\nu_4^{\sch} \leqslant 2^{11} \cdot 3^5 \cdot 5^2 \cdot 7 \cdot 13 = 1\,132\,185\,600 < 10^{10}.$$
\end{proof}

\begin{lemma}\label{lemma: d = 4, e_p = 2}
    Let $K$ be a number field of degree $4$. Suppose there exists a prime number~$\relpenalty=10000 p > 2$ such that $e_p = 2$. Then $$\nu_{4}^{\sch} \leqslant 87\,091\,200 < 10^{8}.$$
\end{lemma}

\begin{proof}
    If there is another prime $q>2$ with $e_q = 1$, then we are done by Lemma~\ref{lemma: d = 4, e_p = 1}. Therefore, we can assume that for any other prime $q>2$ we have $e_q \geqslant 2$. If $\xi_4 \in K$, then $$\nu_{4}^{\sch} \leqslant  11\,520\,000 < 10^{8}$$ by Lemma~\ref{lemma: d = 4, e_p = 1}. If $\xi_4 \not\in K$, then by Lemma~\ref{lemma: boundin of m_2} we have $m_2 \leqslant 3$ and by Lemma~\ref{lemma: divisibility of gcd} for any prime $q > 2$ other than $p$ we have $$t_q \geqslant \frac{q-1}{2}.$$ In particular, we have $\nu_{q,4}^{\sch} = 0$ for any $q > 7$ other than $p$ and $\nu_{2,4}^{\sch} \leqslant 11$. Let us find bounds for small odd primes. 

    If $q = 7$ other than $p$, then the equality $$7^{m_7 - 1}\cdot6\cdot e_7 = 4 \cdot t_7$$ shows that $m_7 = 1$, $t_7 \geqslant 3$ or $m_7 \geqslant 2$, $t_7 \geqslant 21$. In both cases we have $\relpenalty=10000 \nu_{7,4}^{\sch} \leqslant 1$. If $q = 5$ other than~$p$, then the equality $$5^{m_5 - 1}\cdot4\cdot e_5 = 4 \cdot t_5$$ shows that $m_5 = 1$, $t_5 \geqslant 2$ or $m_5 \geqslant 2$, $t_5 \geqslant 10$. In both cases, we have~$\nu_{5,4}^{\sch} \leqslant 2$. If $q = 3$ other than~$p$, then the same analysis does not provide a better bound than in Lemma~\ref{lemma: rough multiplicative bound}, i.e. we have~$\relpenalty=10000 \nu_{3,4}^{\sch} \leqslant 5$. 

    It is left to estimate $\nu_{p,4}^{\sch}$. And the bound comes from the standard equality: $$p^{m_p -1}\cdot(p-1) \cdot 2 = 4 \cdot t_p.$$ Note that this equality is very restrictive. Since $m_p$ and $t_p$ are positive integer numbers, clearly this equality could not be satisfied for some primes $p>2$. Straightforward computations provide the following bounds: 
    \begin{enumerate}
        \item $\nu_{p,4}^{\sch} \leqslant 1$, if $p = 7$;
        \item $\nu_{p,4}^{\sch} \leqslant 2$, if $p = 5$;
        \item $\nu_{p,4}^{\sch} \leqslant 5$, if $p = 3$;
        \item $\nu_{p,4}^{\sch} = 0$ for all other possible values of $p > 2$.
    \end{enumerate}
    Summarizing all these estimates, we obtain bounds  depending on the actual value of $p$. The worst one we can obtain corresponds to the cases $p = 7$, $5$ or $3$: $$\nu_4^{\sch} \leqslant 2^{11} \cdot 3^5 \cdot 5^2 \cdot 7 = 87\,091\,200 < 10^{8}.$$
\end{proof}

Now we are ready for final improvements of our bounds in cases of degree $6$ and $4$.

\begin{lemma}\label{lemma: d = 6 final}
    Let $K$ be a number field of degree $6$. Then $$\nu_{4}^{\sch} \leqslant 1\,132\,185\,600 < 10^{10}.$$
\end{lemma}
\begin{proof}
If there exists a prime number $p > 2$ with $e_p = 1$, then we have $$\nu_{4}^{\sch} \leqslant  221\,276\,160 < 10^9$$ by Lemma~\ref{lemma: d = 6, e_p = 1}. If there exists a prime number $p > 2$ with $e_p = 2$, then we have $$\nu_{4}^{\sch} \leqslant 1\,132\,185\,600 < 10^{10}$$ by Lemma~\ref{lemma: d = 6, e_p = 2}. Then assume that for any prime number $p > 2$ we have~$\relpenalty=10000 e_p \geqslant 3$. 

Immediately note, that we have $\nu^{\sch}_{2,4} \leqslant 11$ by Lemma~\ref{lemma: rough multiplicative bound}. For odd primes consider the standard equation 
\begin{equation}\label{eq: d = 6 final}    
p^{m_p -1}\cdot(p-1) \cdot e_p = 6 \cdot t_p.\end{equation} Due to Table~\ref{tabular: tabular d<=21} we are interested only in primes: $3$, $5$, $7$, $13$ and $19$. Substituting these primes and solving equation, we see that $t_{19}$, $t_{13} > 4$, which implies $\nu_{19,4}^{\sch} = \nu_{13,4}^{\sch} = 0$. 

\begin{comment}
Also, it follows that $$m_{17}\cdot\left\lfloor\frac{4}{t_{17}}\right\rfloor \leqslant \left\lfloor\frac{4}{t_{17}}\right\rfloor, \; m_{13}\cdot\left\lfloor\frac{4}{t_{13}}\right\rfloor \leqslant \left\lfloor\frac{4}{t_{13}}\right\rfloor, \; m_{11}\cdot\left\lfloor\frac{4}{t_{11}}\right\rfloor \leqslant \left\lfloor\frac{4}{t_{11}}\right\rfloor, \;m_{7}\cdot\left\lfloor\frac{4}{t_{7}}\right\rfloor \leqslant \left\lfloor\frac{4}{t_{7}}\right\rfloor,$$ $$m_{5}\cdot\left\lfloor\frac{4}{t_{5}}\right\rfloor \leqslant 2\left\lfloor\frac{4}{t_{5}}\right\rfloor.$$
\end{comment}

Suppose that $\nu_{7,4}^{\sch} \neq 0$, then we have $t_{7} \leqslant 4$. Therefore, there are two solutions of $$7^{m_{7} - 1}\cdot6\cdot e_{7} = 6\cdot t_{7}.$$ Either $m_{7} = 1$ and $e_{7} = t_{7} = 3$ or $m_{7} = 1$ and $e_{7} = t_{7} = 4$. In both cases, we have~$\relpenalty=10000 \nu_{7,4}^{\sch} = 1$. In the former case, applying Lemma~\ref{lemma: boundin of m_2} for $p = 7$ we find out that $\xi_4 \not\in K$ (since $e_{7}$ is odd) and $m_2 \leqslant 2$. Therefore, we have $\nu_{2,4}^{\sch} \leqslant 9$. Also, applying Lemma~\ref{lemma: divisibility of gcd} for $p = 7$ we obtain inequality $t_{5} \geqslant 4$. Solving~\eqref{eq: d = 6 final} for primes $3$ and $5$ we obtain a bound $$\nu_{4}^{\sch} \leqslant 2^{9} \cdot 3^5 \cdot 5 \cdot 7 = 4\,354\,560 <10^{7}.$$ 
In the latter case, applying Lemma~\ref{lemma: divisibility of gcd} we obtain inequality $t_{5} \geqslant 2$. Solving~\eqref{eq: d = 6 final} for primes $3$ and $5$ we obtain a bound $$\nu_4^{\sch} \leqslant 2^{11} \cdot 3^5 \cdot 5^2 \cdot 7 = 87\,091\,200 < 10^{8}.$$ Further, assume that $\nu_{7,4}^{\sch} = 0.$

Under all assumptions we already have a bound $$\nu_{4}^{\sch} \leqslant 2^{11} \cdot 3^9 \cdot 5^2  = 1\,007\,769\,600 < 10^{10}$$ by Lemma~\ref{lemma: rough multiplicative bound} and the lemma is proved.
\end{proof}

\begin{lemma}\label{lemma: d = 4 final}
    Let $K$ be a number field of degree $4$. Then  $$\nu_{4}^{\sch} \leqslant  87\,091\,200 < 10^{8}.$$ 
\end{lemma}
\begin{proof}
If there exists a prime number $p > 2$ with $e_p = 1$, then we have $$\nu_{4}^{\sch} \leqslant  11\,520\,000 < 10^{8}$$ by Lemma~\ref{lemma: d = 4, e_p = 1}. If there exists a prime number $p > 2$ with $e_p = 2$, then we have $$\nu_{4}^{\sch} \leqslant  87\,091\,200 < 10^{8}$$ by Lemma~\ref{lemma: d = 4, e_p = 2}. Then assume that for any prime number $p > 2$ we have~$\relpenalty=10000 e_p \geqslant 3$. 

Immediately note, that we have $\nu^{\sch}_{2,4} \leqslant 15$ by Lemma~\ref{lemma: rough multiplicative bound}. For odd primes consider the standard equation 
\begin{equation}\label{eq: d = 4 final}    
p^{m_p -1}\cdot(p-1) \cdot e_p = 4 \cdot t_p.\end{equation} Due to Table~\ref{tabular: tabular d<=21} we are interested only in primes: $3$, $5$, $7$, $13$ and $17$. Substituting these primes and solving equation, we see that $t_{17}$, $t_{13}$ $t_{7} > 4$, which implies~$\relpenalty=10000 \nu_{17,4}^{\sch} = \nu_{13,4}^{\sch} = \nu_{7,4}^{\sch} = 0$. Therefore, we already have a bound $\nu_4^{sch} \leqslant 2^{15} \cdot 3^{5} \cdot 5^4$ by Lemma~\ref{lemma: rough multiplicative bound}. Let us decrease the power of $5$ in this bound. Consider~\ref{eq: d = 4 final} for $p = 5$: $$5^{m_5 - 1}\cdot 4 \cdot e_5 = 4 \cdot t_5.$$ Since $e_5 \geqslant 3$, then for all solutions we have $\nu_{5,4}^{\sch} \leqslant 1$. This leads to the bound $$\nu_{4}^{\sch} \leqslant 2^{15} \cdot 3^5 \cdot 5 = 39\,813\,120 < 10^{8}$$ and the lemma is proved.
\end{proof}

Now everything is prepared to prove Proposition~\ref{proposition: bound of schur GL_4 <= 7}

\begin{proof}[of Proposition~\ref{proposition: bound of schur GL_4 <= 7}]
    For $d = 6$ and $d = 4$ result follows from Lemma~\ref{lemma: d = 6 final} and Lemma~\ref{lemma: d = 4 final}. For $d = 1, 2, 5, 7$ result follows from Remark~\ref{rem: multiplicative bound} and  Table~\ref{tabular: tabular d<=21}. For $d = 3$ result follows Corollary~\ref{corollary: nu(L) <= nu(K)} applied to the extension $K \subset K(\sqrt{2})$ and Lemma~\ref{lemma: d = 6 final}.
\end{proof}

Also, we need the following rough bounds.

\begin{lemma}\label{lemma: corollary: bound of schur GL_4 <= 2}
     Let $K$ be a number field of degree $d \leqslant 2$. Then $$\nu^{\sch}_4 \leqslant 87\,091\,200 < 10^{8}.$$
\end{lemma}

\begin{proof}
    Result follows from Remark~\ref{rem: multiplicative bound} and  Table~\ref{tabular: tabular d<=21}. 
\end{proof}

\begin{lemma}\label{lemma: corollary: bound of schur GL_3 <= 2}
     Let $K$ be a number field of degree $d \leqslant 2$. Then $$\nu^{\sch}_3 \leqslant 2\,903\,040 < 10^{7}.$$
\end{lemma}

\begin{proof}
    Result follows from Remark~\ref{rem: multiplicative bound} and  Table~\ref{tabular: tabular GL_3(<=15)}. 
\end{proof}

Finally, we need the following bounds for $\nu^{\se}$.

\begin{lemma}\label{lemma: serr for PGL_3(<=2)}
Let $K$ be a number field of degree $d \leqslant 2$.
Then $$\nu_{3}^{\se} \leqslant 2\,620\,800 < 10^{7}.$$
\end{lemma}

\begin{proof}
    For convenience, let us write down the formula:
    \begin{equation}\label{eq: Serre for PGL3}
   \nu_{p,3}^{\se} = m_p\left\lfloor\frac{2}{\varphi(t_p)}\right\rfloor + \nu_p(2). \end{equation} 
    If $d = 1$ we know actual values of invariants $m_p$ and $t_p$ and straightforward computations gives us an equality:

    $$\nu_{3}^{\se} = 2^5 \cdot 3^2 \cdot 5 \cdot 7 = 10\,080.$$
    
    Consider the case $d = 2$. For $p = 2$ we have a diagram
\begin{center}
\begin{tikzpicture}
    
    \node (Q1) at (0,0) {$\QQ$};
    \node (Q2) at (0 + 4/3,4/3) {$K$};
    \node (Q3) at (0 + 0,8/3) {$K(\xi_4) = K(\xi_{2^{m_2}})$};
    \node (Q4) at (0 -4/3,4/3) {$\QQ(\xi_{2^{m_2}})$};
    \node (Q5) at (-13/10,1/2) {$2^{m_2-1}$};
    \draw (Q1)--(Q2) node [pos=0.9, below,inner sep=0.25cm] {$2$};
    \draw (Q1)--(Q4); % node [pos=0.9, below,inner sep=0.25cm] {$p^{m_p-1}(p-1)$};
    \draw (Q4)--(Q3);% node [pos=0, above,inner sep=0.25cm]{$e_{p}$};
    \draw (Q2)--(Q3) node [pos=0, above,inner sep=0.25cm]{$t_{2}$};

\end{tikzpicture}
\end{center}
    Since $t_2 = 1$ or $2$, then $\varphi(t_2) =1$, and it follows from the diagram, that $m_2-1 \leqslant \nu_2(2t_2)$. Also, $\nu_2(2t_2) \leqslant 2$, which implies the inequality $m_2 \leqslant 3$ and we have $\nu_{2,3}^{\se} \leqslant 7$ by~\eqref{eq: Serre for PGL3}.

        For each prime $p>2$ consider a standard diagram 

\begin{center}
\begin{tikzpicture}
    
    \node (Q1) at (0,0) {$\QQ$};
    \node (Q2) at (0 + 4/3,4/3) {$K$};
    \node (Q3) at (0 + 0,8/3) {$K(\xi_p) = K(\xi_{p^{m_p}})$};
    \node (Q4) at (0 -4/3,4/3) {$\QQ(\xi_{p^{m_p}})$};
    \node (Q5) at (-19/10,1/2) {$p^{m_p-1}(p-1)$};
    \draw (Q1)--(Q2) node [pos=0.9, below,inner sep=0.25cm] {$2$};
    \draw (Q1)--(Q4); % node [pos=0.9, below,inner sep=0.25cm] {$p^{m_p-1}(p-1)$};
    \draw (Q4)--(Q3) node [pos=0, above,inner sep=0.25cm]{$e_{p}$};
    \draw (Q2)--(Q3) node [pos=0, above,inner sep=0.25cm]{$t_{p}$};

\end{tikzpicture}
\end{center}
It follows that $t_p$ is divisible by $$\frac{p^{m_p-1}(p-1)}{\gcd(p^{m_p-1}(p-1),2)} = \frac{p^{m_p-1}(p-1)}{\gcd(p-1,2)}$$ and $$\varphi(t_p) \geqslant \varphi\left( \frac{p^{m_p -1}(p-1)}{\gcd(p-1,2)}\right).$$

If $p \geqslant 17$ or $p = 11$, then the number $\frac{p^{m_p-1}(p-1)}{\gcd(p-1,2)}$ is either equal to $5$ or greater than $6$, which implies $$\varphi(t_p) \geqslant \varphi \left( \frac{p^{m_p-1}(p-1)}{\gcd(p-1,2)}\right) > 2$$ and we have $\nu_{p,3}^{\se} = 0$ by~\eqref{eq: Serre for PGL3}.

If $p=5$, $7$, or $13$, then we have $\nu_p(2) = 0$. If $m_p \geqslant 2$, then $t_p$ is divisible by $p$, which implies inequality $\varphi(t_p) \geqslant \varphi(p) > 2$ and $\nu_{p,3}^{\se} = 0$ by~\eqref{eq: Serre for PGL3}. Therefore, assuming that~$\relpenalty=10000 m_p = 1$, we have an inequality $$\varphi(t_p) \geqslant \varphi \left( \frac{p-1}{\gcd(p-1,2)}\right),$$ which guarantees that $\varphi(t_p) \geqslant 2$ for $p = 7$ or $13$. Substituting this in~\eqref{eq: Serre for PGL3} we obtain inequalities~$\nu_{13,3}^{\se} \leqslant 1$, $\nu_{7,3}^{\se} \leqslant 1$ and $\nu_{5,3}^{\se} \leqslant 2$.

If $p = 3$, then $\nu_3(2) = 0$. Let us consider all possible values of $m_3$. If~$m_3 \geqslant 3$, then~$t_3$ is divisible by $9$, which implies inequality $\varphi(t_3) \geqslant \varphi(9) = 6$ and~$\relpenalty= 10000 \nu_{3,3}^{\se} = 0$ by~\eqref{eq: Serre for PGL3}. If~$m_3 =2$, then $t_3$ is divisible by $3$, which implies inequality $\varphi(t_3) \geqslant \varphi(3) = 2$ and~$\relpenalty= 10000 \nu_{3,3}^{\se}\leqslant 2$ by~\eqref{eq: Serre for PGL3}. If~$m_3 = 1$, then we clearly have $\nu_{3,3}^{\se} \leqslant 2$ by~\eqref{eq: Serre for PGL3}. Thus, in all cases we have a bound~$\nu_{3,3}^{\se} \leqslant 2$.

Finally, we have $$\nu_3^{\se} \leqslant 2^7\cdot 3^2\cdot 5^2 \cdot 7 \cdot 13 =2\,620\,800 < 10^{7}.$$
\end{proof}

\begin{lemma}\label{lemma: serr for PGL_4(<=2)}
    Let $K$ be a number field of degree $d \leqslant 2$. Then $$\nu_4^{\se} \leqslant 943\,488\,000 < 10^{9}.$$
\end{lemma}

\begin{proof}
    For convenience, let us write down the formula:
     \begin{equation}\label{eq: Serre for PGL4}
   \nu_{p,4}^{\se} = m_p\left\lfloor\frac{3}{\varphi(t_p)}\right\rfloor + \nu_p(6). \end{equation} 
    If $d = 1$ we know actual values of invariants $m_p$ and $t_p$ and straightforward computations gives us an equality:

    $$\nu_{4}^{\se} = 2^7 \cdot 3^4 \cdot 5 \cdot 7 = 362\,880.$$
    
    Consider the case $d = 2$. For $p = 2$ we have a diagram
\begin{center}
\begin{tikzpicture}
    
    \node (Q1) at (0,0) {$\QQ$};
    \node (Q2) at (0 + 4/3,4/3) {$K$};
    \node (Q3) at (0 + 0,8/3) {$K(\xi_4) = K(\xi_{2^{m_2}})$};
    \node (Q4) at (0 -4/3,4/3) {$\QQ(\xi_{2^{m_2}})$};
    \node (Q5) at (-13/10,1/2) {$2^{m_2-1}$};
    \draw (Q1)--(Q2) node [pos=0.9, below,inner sep=0.25cm] {$2$};
    \draw (Q1)--(Q4); % node [pos=0.9, below,inner sep=0.25cm] {$p^{m_p-1}(p-1)$};
    \draw (Q4)--(Q3);% node [pos=0, above,inner sep=0.25cm]{$e_{p}$};
    \draw (Q2)--(Q3) node [pos=0, above,inner sep=0.25cm]{$t_{2}$};

\end{tikzpicture}
\end{center}
    Since $t_2 = 1$ or $2$, then $\varphi(t_2) =1$, and it follows from the diagram, that $m_2-1 \leqslant \nu_2(2t_2)$. Also, $\nu_2(2t_2) \leqslant 2$, which implies the inequality $m_2 \leqslant 3$ and we have $\nu_{2,4}^{\se} \leqslant 10$ by~\eqref{eq: Serre for PGL4}.

        For each prime $p>2$ consider a standard diagram 

\begin{center}
\begin{tikzpicture}
    
    \node (Q1) at (0,0) {$\QQ$};
    \node (Q2) at (0 + 4/3,4/3) {$K$};
    \node (Q3) at (0 + 0,8/3) {$K(\xi_p) = K(\xi_{p^{m_p}})$};
    \node (Q4) at (0 -4/3,4/3) {$\QQ(\xi_{p^{m_p}})$};
    \node (Q5) at (-19/10,1/2) {$p^{m_p-1}(p-1)$};
    \draw (Q1)--(Q2) node [pos=0.9, below,inner sep=0.25cm] {$2$};
    \draw (Q1)--(Q4); % node [pos=0.9, below,inner sep=0.25cm] {$p^{m_p-1}(p-1)$};
    \draw (Q4)--(Q3) node [pos=0, above,inner sep=0.25cm]{$e_{p}$};
    \draw (Q2)--(Q3) node [pos=0, above,inner sep=0.25cm]{$t_{p}$};

\end{tikzpicture}
\end{center}
It follows that $t_p$ is divisible by $$\frac{p^{m_p-1}(p-1)}{\gcd(p^{m_p-1}(p-1),2)} = \frac{p^{m_p-1}(p-1)}{\gcd(p-1,2)}$$ and $$\varphi(t_p) \geqslant \varphi\left( \frac{p^{m_p -1}(p-1)}{\gcd(p-1,2)}\right).$$

If $p \geqslant 17$ or $p = 11$, then the number $\frac{p^{m_p-1}(p-1)}{\gcd(p-1,2)}$ is either equal to $5$ or greater than $6$, which implies $$\varphi(t_p) \geqslant \varphi \left( \frac{p^{m_p-1}(p-1)}{\gcd(p-1,2)}\right) > 3$$ and we have $\nu_{p,4}^{\se} = 0$ by~\eqref{eq: Serre for PGL4}.

If $p=5$, $7$, or $13$, then we have $\nu_p(6) = 0$. If $m_p \geqslant 2$, then $t_p$ is divisible by $p$, which implies inequality $\varphi(t_p) \geqslant \varphi(p) > 3$ and $\nu_{p,4}^{\se} = 0$ by~\eqref{eq: Serre for PGL4}. Therefore, assuming that~$\relpenalty=10000 m_p = 1$, we have an inequality $$\varphi(t_p) \geqslant \varphi \left( \frac{p-1}{\gcd(p-1,2)}\right),$$ which guarantees that $\varphi(t_p) \geqslant 2$ for $p = 7$ or $13$. Substituting this in~\eqref{eq: Serre for PGL4} we obtain inequalities~$\nu_{13,4}^{\se} \leqslant 1$, $\nu_{7,4}^{\se} \leqslant 1$ and $\nu_{5,4}^{\se} \leqslant 3$.

If $p = 3$, then $\nu_3(6) = 1$. Let us consider all possible values of $m_3$. If~$m_3 \geqslant 3$, then~$t_3$ is divisible by $9$, which implies inequality $\varphi(t_3) \geqslant \varphi(9) = 6$ and~$\relpenalty= 10000 \nu_{3,4}^{\se}\leqslant 1$ by~\eqref{eq: Serre for PGL4}. If~$m_3 =2$, then $t_3$ is divisible by $3$, which implies inequality $\varphi(t_3) \geqslant \varphi(3) = 2$ and~$\relpenalty= 10000 \nu_{3,4}^{\se}\leqslant 3$ by~\eqref{eq: Serre for PGL4}. If~$m_3 = 1$, then we clearly have $\nu_{3,4}^{\se} \leqslant 4$ by~\eqref{eq: Serre for PGL4}. Thus, in all cases we have a bound~$\nu_{3,4}^{\se} \leqslant 4$.

Finally, we have $$\nu_3^{\se} \leqslant 2^{10}\cdot 3^4\cdot 5^3 \cdot 7 \cdot 13 = 943\,488\,000 < 10^{9}.$$
\end{proof}

\end{document}